%% file: a-superflows-dis.tex
\documentclass{dis}

\usepackage{amssymb,amstext,amsmath,amsthm,mathrsfs,url,fancybox}
\usepackage{latexsym}
\usepackage{graphicx}
\usepackage{pgf}
\usepackage{tikz}

\usepackage[utf8]{inputenc}
\usetikzlibrary{arrows,automata}
\usetikzlibrary{positioning}

\tikzset{
    state/.style={
           rectangle,
           rounded corners,
           draw=black, very thick,
           minimum height=2em,
           inner sep=2pt,
           text centered,
           },
}

\newcounter{noteno}
\newcounter{exam}
\newcounter{polin}

\usepackage{verbatim}

\newtheorem{thm}{Theorem}[chapter]
\newtheorem{cor}[thm]{Corollary}
\newtheorem{lem}[thm]{Lemma}
\newtheorem{prop}[thm]{Proposition}
\newtheorem{defin}[thm]{Definition}

\newtheorem*{thmm}{Theorem}

\newtheorem{prob}[thm]{Problem}

\chaptersnewpage
\refnewpage
\makeindex

\newcommand{\m}{\mathbf}

\newcommand{\bl}{\bullet}
\newcommand{\p}{\partial}
\def\d{\,{\rm{d}}}

\def\v{{\varsigma}}
\def\sm{\,{\rm{sm}}}
\def\cm{\,{\rm{cm}}}

\def\sn{\,{\rm{sn}}}
\def\cn{\,{\rm{cn}}}
\def\dn{\,{\rm{dn}}}

\newenvironment{Note}%
	{\refstepcounter{noteno}%
	
	\medbreak\par\noindent{{\bf Note~\thenoteno}.}}%
	{\hfill{$\Box$}
	
	\par\medbreak}

\newenvironment{Example}%
	{\refstepcounter{exam}%
	
	\medbreak\par\noindent{{\bf Example~\theexam}.}}%
	{\hfill{$\Box$}
	
	\par\medbreak}
	
	\newenvironment{Poly}%
	{\refstepcounter{polin}%
	\medbreak\par\noindent{{\bf Polynomial superflows~\thepolin}.}}%
	{\hfill{$\Box$}
	\par\medbreak}

\begin{document}

\thanks{The research of the author was supported by the Research Council of Lithuania grant No. MIP-072/2015}

\date{Started: October 5, 2015. Final version before proofs: December 22, 2017}

\mathclass{Primary 14H70, 20G05, 33E05, 34A05, 39B12, 37C10, 53C12. Secondary 14H45, 37J15,  	37J35, 58E15}


\abbrevauthors{G. Alkauskas}
\abbrevtitle{Superflows}

\title{Projective and polynomial superflows. I}

\author{Giedrius Alkauskas}
\address{Vilnius University,\\ 
Department of Mathematics and Informatics, \\
Institute of Informatics, \\
Naugarduko 24, LT-03225 Vilnius, Lithuania\\
E-mail: giedrius.alkauskas@mif.vu.lt}

\maketitledis

\tableofcontents

\begin{abstract}
Let $\m{x}\in\mathbb{R}^{n}$. For $\phi:\mathbb{R}^{n}\mapsto\mathbb{R}^{n}$ and $t\in\mathbb{R}$, we put $\phi^{t}=t^{-1}\phi(\m{x}t)$. \emph{A projective  flow} is a solution to the projective translation equation $\phi^{t+s}=\phi^{t}\circ\phi^{s}$, $t,s\in\mathbb{R}$. This is equivalent to the fact that a vector field of a flow is $2$-homogeneous.\\  
\indent Previously we have developed an arithmetic, topological and analytic theory of $2$-dimensional projective flows (over $\mathbb{R}$ and $\mathbb{C}$): rational, algebraic, unramified, abelian flows, commuting flows. The current paper is devoted to highly symmetric flows - \emph{projective superflows}. Within flows with a given symmetry, superflows are unique and  optimal, and this is in essence their definition.\\
\indent  Our first result classifies all $2$-dimensional superflows. For any $d\in\mathbb{N}$, there exists the superflow $\phi_{\mathbb{D}_{2d+1}}$ whose group of symmetries is the dihedral group $\mathbb{D}_{2d+1}$. In the current paper we explore the superflow $\phi_{\mathbb{D}_{5}}$, which leads to investigation of abelian functions over a curve of genus $6$. \\
\indent We further investigate two different $3$-dimensional superflows, both solenoidal, whose group of symmetries are, respectively, the full tetrahedral group $\widehat{\mathbb{T}}$, and the octahedral group $\mathbb{O}$. The generic orbits of the first flow are space curves of genus $1$, and the flow itself can be analytically described in terms of Jacobi elliptic functions. The generic orbits of the second flow are curves of genus $9$, and the flow itself can be described in terms of Weierstrass elliptic functions, via a triple reduction, which occurs in this particular case.\\
\indent We also introduce the notion of a \emph{polynomial superflow} as a flow with polynomial vector field whose components are homogeneous of even degree (not necessarily $2$), such that a vector field is also unique and optimal. This notion seems to be even more fundamental than a notion of a projective superflow. However, in many cases projective superflows and polynomial superflows can be easily converted into one another. However, this is not the case with hyper-octahedral superflows in odd dimension $n\geq 5$: projective superflows do not exist, while polynomial ones do. We investigate this in more detail.\\
\indent The main emphasis of our approach, thus distancing it from the main topics of interest usually dealt with in differential geometry, is an explicit integration of the corresponding differential system, minding that in a superflow setting this system appears to possess a maximal set of independent polynomial first integrals. This relates our results to Noether's theorem about how differentiable symmetries of a Lagrangian function give rise to conservation laws (first integrals).\\ 
\indent We also introduce the notion of an \emph{unramified flow}. The latter property depends on a very subtle arithmetic structure of a vector field.\\
\indent In the second part of this work we will classify all $3$-dimensional superflows (including the icosahedral superflow), in the third we investigate superflows over $\mathbb{C}$ in dimension $2$, and the fourth part is dealing with arithmetic of the orbits.
\end{abstract}

\makeabstract

\include{sf-chap1}

\include{sf-chap2}

\include{sf-chap3}

\include{sf-chap4}

\include{sf-chap5}

\include{sf-chap6}

\include{sf-chap7}

\include{sf-chap8}

\include{sf-chap9}

\include{sf-chap10}

\appendix
\include{appendix-pirm}


\end{document}

%% file: sf-chap1.tex
%
%
%

\chapter{Projective flows}
\label{sec-pirmas}
\section{Preliminaries}
\label{prelim}
We always write $F(x,y)\bl G(x,y)$ instead of $\big{(}F(x,y),G(x,y)\big{)}$, and also $\m{x}=x\bl y$. Equally, a vector field $F\frac{\p}{\p x}+G\frac{\p}{\p y}$, when not considered as a derivation (on the space of algebra of $C^{\infty}$ germs), is also frequently denoted by $F\bl G$ for typographical convenience. The same convention applies to a $3$-dimensional case, $\m{x}=x\bl y\bl z$. We use an environment \textbf{Note} to make comments which are not crucial for the main narrative of the text, but which give a deeper insight. We use an environment \textbf{Example} to double check the validity of theoretical results, or to illustrate them in particular cases.\\

The current work is a continuation of \cite{alkauskas-t,alkauskas, alkauskas-rat2,alkauskas-un,alkauskas-ab, alkauskas-comm}, but it is completely independent from the cited papers (including all classification theorems), apart from the fact that we use explicit methods to integrate vector fields developed there. The notion of \emph{the superflow} was first formulated in \cite{alkauskas-un}. Yet, the final formulas we obtain can be double-verified, so the current work is self-contained.\\

 We start with the orthogonal groups $O(2)$ (the dimension of the corresponding Lie algebra is $1$) and $O(3)$ (the dimension is $3$). The latter is continued in \cite{alkauskas-super2}. In \cite{alkauskas-super3} the case of the unitary group $U(2)$ is investigated (the dimension is $4$). The fourth paper \cite{alkauskas-super4} is of a different nature and concerns the arithmetic of the orbits, Abel-Jacobi theorem and abelian functions over algebraic curves.\\

\emph{The translation equation}, or a flow equation, is the equation \cite{aczel, conlon, nikolaev}

\begin{eqnarray}\setlength{\shadowsize}{2pt}\shadowbox{$\displaystyle{\quad
F(F(\m{x},s),t)=F(\m{x},s+t)},\quad t,s\in\mathbb{R}\text{ or }\mathbb{C}$.\quad}
\label{tre-classical}
\end{eqnarray}
One also requires the boundary condition
\begin{eqnarray*}
F(\m{x},0)=\m{x}.
\end{eqnarray*}
When these two are satisified, we have \emph{a flow}. Note that it is of an interest to talk about special solutions to (\ref{tre-classical}), like rational or algebraic functions, which do not satisfy the boundary condition.\\

\emph{The projective translation equation}, or PrTE, is a special case of a general translation equation, and it was first introduced in \cite{alkauskas-t}. It is given by $F(\m{x},t)=t^{-1}\phi(\m{x}t)$, and therefore it is the equation of the form

\begin{eqnarray}\setlength{\shadowsize}{2pt}\shadowbox{$\displaystyle{\quad
\frac{1}{t+s}\,\phi\big{(}\m{x}(t+s)\big{)}=\frac{1}{s}\,\phi\Big{(}\phi(\m{x}t)\frac{s}{t}\Big{)}},\quad t,s\in\mathbb{R}\text{ or }\mathbb{C}$.\quad}\label{funk}
\end{eqnarray}

This should be satisfied for $s,t$ small enough (\emph{local flows}). However, we are interested in \emph{global flows}. With this in mind, in the above equation one can confine to the case $s=1-t$ without altering the set of solutions, though some complications arise concerning ramification, if we are dealing with flows over $\mathbb{C}$. 

\begin{Note} In a PrTE we have a subtle interplay between an additive structure of the reals (that is, $t+s$), and a multiplicative structure (that is, $\frac{s}{t}$). This hints towards a Jordan normal form of matrices, nilpotent elements, and the Jacobson-Morozov theorem concerning $\mathfrak{sl}_{2}$-triples in Lie algebras. This relation needs further clarification. Still, Lie algebras do appear in our investigations as finite-dimensional algebras of infinite linear symmetry groups of reducible superflows; see Proposition \ref{interim}. The tricky point is to show that these groups, if non-abelian, do not have finite non-abelian subgroups \cite{alkauskas-super3}. However, symmetry groups of irreducible superflows (the topic of our main interest, see Definitions \ref{defin-proj} and \ref{defin-poly}) are always finite groups.
\end{Note}
 In a $3$-dimensional case, $\phi(x,y,z)=u(x,y,z)\bl v(x,y,z)\bl w(x,y,z)$ is a triple of functions in three real (or complex) variables. A non-singular solution of this equation is called \emph{a projective flow}.  The \emph{non-singularity} means that a flow satisfies the boundary condition
\begin{eqnarray}
\lim\limits_{t\rightarrow 0}\frac{\phi(\m{x}t)}{t}=\m{x}.
\label{init}
\end{eqnarray}

However, as explained in Section \ref{sec6.3} (see Proposition \ref{prop-ner}), and much more exhaustively in \cite{alkauskas-rat2}, the phrase \emph{``is a special case of a general translation equation"} is misleading. For the explanation, see cited references, but we will emphasize that, as far as the space $\mathbb{R}^{n}$ is concerned, the projective translation equation describes general flows not less, and in some cases (when a vector field is $2$-homogeneous) more efficiently than the affine translation equation. Moreover, some questions, like rationality or algebraicity of a flow, are much more natural in the projective flow setting \cite{alkauskas-rat2} - see Note \ref{note-alg} for an example and an explanation. 
\begin{Poly} This paper, apart from the final Section \ref{hyper-5}, is written from the perspective of projective translation equation. The notion of polynomial superflows, which is related to translation equation but not generally to PrTE, however, seems to be even more fundamental: these exist in huge variety. Indeed, sometimes the degree of homogenuity of a vector field with a given symmetry is too large to ensure the existence of the unique invariant of the group in consideration, whose degree is less by exactly $2$ and which then could serve as a denominator for the vector field in a projective superflow case. In the other direction, numerator for the vector field of projective superflow, if it is of even degree, always produces a polynomial superflow. This has only few exceptions - the $4$-antiprismal superflow $\phi_{\mathbb{A}_{4}}$ \cite{alkauskas-super2}, and a series of reducible superflows (see \cite{alkauskas-super3}, Proposition 4), where the numerator is of odd degree. Hence, not to confuse these two notions, we single out what applies to polynomial superflows and general translation equation into separate remarks with a distinguished subject line \textbf{Polynomial superflows}. When no adjective is used with a noun ``superflow", we mean the projective one. When the denominator of the vector field of the latter is the first integral of the corresponding differential system and it is of even degree, like $\phi_{\mathbb{D}_{3}}$  in Section \ref{sec-2d}, $\phi_{S_{n+1}}$ in Section \ref{S4}, $\phi_{\widehat{\mathbb{T}}}$ in Section \ref{sub4.2}, $\phi_{\mathbb{O}}$ in Sections \ref{octahedral}, \ref{partiii} and \ref{partiv}, or $\phi_{\mathbb{I}}$ in \cite{alkauskas-super2}, these two notions essentially coincide. 
\end{Poly}
  
The main object accompanying a projective flow is its \emph{vector field} given by
\begin{eqnarray}
\varpi(x,y,z)\bl\varrho(x,y,z)\bl\sigma(x,y,z)=\frac{\d}{\d t}\frac{\phi(xt,yt,zt)}{t}\Big{|}_{t=0}.
\label{vec}
\end{eqnarray}
Vector field is necessarily a triple of $2$-homogeneous functions, and this is exactly what distinguishes projective flows from general flows. Such vector fields were investigated in the literature: for example, the paper \cite{camacho} contains geometric analysis of polynomial $2$-homogeneous vector fields in $\mathbb{R}^{3}$. A research of the corresponding differential system, closely related the topic of the current study, is carried out in \cite{hopkins1, hopkins2}.\\

Each point under a flow possesses the orbit, which is defined by
\begin{eqnarray*}
\mathscr{V}(\m{x})=\bigcup\limits_{t\in\mathbb{R}}\phi^{t}(\m{x}),\quad \phi^{t}(\m{x})=\frac{\phi(\m{x}t)}{t}.
\end{eqnarray*}
\\

Our approach to the analytic structure of vector fields differs from the one usually considered in differential geometry: while normally one considers vector fields being smooth (of class $C^{\infty}$), we limit ourself to rational vector fields, allowing discontinuities and indeterminacy at zero locus of denominators. This gives a strong algebro-geometric side.

\begin{Note}
\label{note-just}
 As explained in Section \ref{sec6.3}, we can use a stereographic projection (or any other birational map) to get a vector field in $\mathbb{R}^{2}$, not homogeneous anymore, whose group of symmetries is a finite subgroup of the group of birational transformations of $\mathbb{R}^{2}$, and thus our formulas give explicit expressions for the latter flows. Therefore, projective flows ramify much wider than can be expected from the starting point of limiting our vector fields to $2$-homogeneous rational functions. For an additional motivation, see the end of Section \ref{sub-inv}.
\end{Note}
\begin{Note} The projective flow $\phi$ with rational vector field is said to be \emph{rational} or \emph{algebraic}, if all of its coordinates are rational or algebraic functions. (Note that algebraic vector fields can also produce algebraic flows; we did and do not deal with this situation, it is less important from the point of view of ramification and algebraic geometry).\\

 For example, the vector field $2x^2+xy\bl xy+2y^2$ produces the flow $\phi(\m{x})=u(x,y)\bl u(y,x)$, where $u$ is a third degree algebraic function, given by \cite{alkauskas-comm}
\begin{eqnarray*}
u(x,y)=\frac{\Bigg{(}\sqrt[3]{x+y\sqrt{\frac{y-3x+6x^2}{x-3y+6y^2}}}+\sqrt[3]{x-y\sqrt{\frac{y-3x+6x^2}{x-3y+6y^2}}}\Bigg{)}x(x-y)}{\sqrt[3]{x-3y+6y^2}\cdot(y-3x+6x^2)}
-\frac{2x^2}{y-3x+6x^2}.
\end{eqnarray*}
The orbits of this flow are algebraic curves $x^2(x-y)^{-3}y^{2}=\mathrm{const.}$ If we consider a local flow ($t^{-1}\phi(\m{x}t)$ for $t$ small enough), this is valid for, say, $0<\frac{x}{3}<y<3x$; see \cite{alkauskas-comm} for the double verification that this is indeed the correct explicit expression. The ``time" variable is accommodated within space variables $x,y$, and thus the notion of the flow being algebraic (or rational) is unambiguous.\\
 
 However, we have a dilemma what to call \emph{an algebraic flow} in a general setting: if $F(\m{x},t)$ is such a flow, then we may require that:
\begin{itemize}
\item[i)] $F(\m{x},t_{0})$ is a vector of algebraic functions in $\m{x}$ for any particular $t_{0}$; or
\item[ii)] $F(\m{x},t)$ is an algebraic function in all variables $(\m{x},t)$.
\end{itemize}
\indent For example, consider the general $1$-dimensional translation equation, and the vector field $\varpi(x)=\frac{(x^2+1)x}{x^2-1}$. Integrating the differential equation $\frac{\d x}{\d t}=\frac{(x^2+1)x}{x^2-1}$ and finding its special solution $x(0)=x_{0}$, we find that $\varpi$ generates the flow
\begin{eqnarray*}
F(x,t)=e^t\frac{x^2+1}{2x}+\frac{x^2-1}{2x}\sqrt{1+(e^{2t}-1)\frac{(x^2+1)^2}{(x^2-1)^2}},\quad |t|<C(x).
\end{eqnarray*}
Indeed, then $F(x,0)=x$. The square root for small $t$ is assumed to be positive. This flow is algebraic according to the first definition, but not algebraic according to the second.
\label{note-alg}
\end{Note}

If a function is smooth, the functional equation (\ref{funk}) and the condition (\ref{init}) imply the PDE \cite{alkauskas}
\begin{eqnarray}
u_{x}(\varpi-x)+u_{y}(\varrho-y)+u_{z}(\sigma-z)=-u,\label{pde}
\end{eqnarray}
and the same PDE for $v$ and $w$, with the boundary conditions as given by (\ref{init}). That is,
\begin{eqnarray}
\lim\limits_{t\rightarrow 0}\frac{u(xt,yt,zt)}{t}=x,\quad
\lim\limits_{t\rightarrow 0}\frac{v(xt,yt,zt)}{t}=y,\quad
\lim\limits_{t\rightarrow 0}\frac{w(xt,yt,zt)}{t}=z.
\label{boundr}
\end{eqnarray}
These three PDEs (\ref{pde}) with the above boundary conditions are equivalent to (\ref{funk}) for $s,t$ small enough \cite{alkauskas}.\\

We can give here an alternative proof of (\ref{pde}) than the one presented in \cite{alkauskas}. Let $F(\m{x},t):\mathbb{R}^{3}\times\mathbb{R}\mapsto\mathbb{R}^{3}$ be a general flow, 
$F=u\bl v\bl w$, and $\varpi\bl\varrho\bl\sigma$ be its vector field; see (\ref{tre-classical}). Then we have the flow equation
\begin{eqnarray*}
u\Big{(}u(\m{x},s),v(\m{x},s),w(\m{x},s);t\Big{)}=u(\m{x};s+t).
\end{eqnarray*}
Differentiate now with respect to $s$, and put $s=0$ afterwards. This gives
\begin{eqnarray}
u_{x}(\m{x};t)\varpi+u_{y}(\m{x};t)\varrho+u_{z}(\m{x};t)\sigma=u_{t}(\m{x};t).
\label{PDE-gen}
\end{eqnarray}
In a projective flow case $u(\m{x};t)=t^{-1}\hat{u}(xt,yt,zt)$, hence the expression $u_{t}$, if we put $t=1$, reads as $u_{x}x+u_{y}y+u_{z}z-u$, where $u(\m{x})=u(\m{x};1)$. Whence (\ref{pde}). The same holds for $v$ and $w$.\\

The important feature of the equation (\ref{pde}) is that we can calculate the Taylor series of the solution $u(x,y,z)$ without knowing its explicit expression in closed form. Indeed, as was shown in \cite{alkauskas}, and this equality is essentially equivalent to (\ref{pde}) coupled with the boundary condition (\ref{boundr}), that
\begin{eqnarray}
u(xt,yt,zt)&=&xt+\sum\limits_{i=2}^{\infty}t^{i}\varpi^{(i)}(x,y,z),\nonumber\\
\varpi^{(i+1)}(x,y,z)&=&\frac{1}{i}[\varpi^{(i)}_{x}\varpi+\varpi^{(i)}_{y}\varrho+\varpi^{(i)}_{z}\sigma],\quad i\geq 2,\label{expl-taylor}
\end{eqnarray}
where $\varpi^{(1)}=x$. This converges for $t$ small enough. Thus, $\varpi^{(i)}$ is a homogeneous function of degree $i$. To get functions $v$ and $w$, we use analogous recurrence, only replace $\varpi^{(i)}$ with $\varrho^{(i)}$ and $\sigma^{(i)}$, respectively, and use $\varrho^{(1)}=y$, $\sigma^{(1)}=z$. For general flows, the corresponding Taylor series is given by (\ref{taylor-gen}).\\

After integrating the vector field in case orbits of the flow are algebraic curves (see Problem \ref{orbits-algebro} in the end of Section \ref{sec1.6}), we obtain the expression involving abelian functions whose Taylor series can be calculated recurrently by knowing the curve which is parametrized by this abelian function and its derivative. Plugging the latter series into the closed-form formula we get after the integration, we always verify that the obtained series is identical to the one given by (\ref{expl-taylor}), thus checking the validity of the closed-form identity.

\section{Previous results}
Section \ref{prelim} is written from the perspective of $3$-dimensional projective flows, but the same applies to $2$-dimensional  flows. The orbits of the flow with the vector field $\varpi\bl\varrho$ are given by $\mathscr{W}(x,y)=\mathrm{const}.$, where the function $\mathscr{W}$ can be found from the differential equation
\begin{eqnarray}
\mathscr{W}(x,y)\varrho(x,y)+\mathscr{W}_{x}(x,y)[y\varpi(x,y)-x\varrho(x,y)]=0.
\label{orbits}
\end{eqnarray}
$\mathscr{W}$ is uniquely (up to a scalar multiple) defined from this ODE and the condition that it is a $1$-homogeneous function.\\

Quadratic autonomous differential equations on the plane is a rich and ramified subject. It encompasses algebraic, dynamic, asymptotic and qualitative aspects. We may refer to \cite{date,gine,shlomiuk,winkel} - more than 2000 papers on quadratic vector fields on the plane alone.\\
 
The projective translation equation gives a fresh and new perspective. Many things are already known in a $2$-dimensional case, when, as always, vector field is given by a pair of $2$-homogeneous rational functions. In this Section we will confine to presenting two results related to arithmetic side of these.

\begin{thm}[\cite{alkauskas,alkauskas-rat2}]
Let $\phi(x,y)=u(x,y)\bl v(x,y)$ be a pair of rational functions in $\mathbb{R}(x,y)$ which satisfies the functional equation (\ref{funk})
and the boundary conditions (\ref{init}). Assume that $\phi(\m{x})\neq \phi_{{\rm id}}(\m{x}):=x\bl y$.
Then there exists an integer $N\geq 0$, which  is the basic invariant of the flow, called \emph{the level}. Such a flow $\phi(x,y)$ can be given by
\begin{eqnarray*}
\phi(x,y)=\ell^{-1}\circ\phi_{N}\circ\ell(x,y),
\end{eqnarray*}
where $\ell$ is a $1-$BIR ($1$-homogeneous birational plane transformation), and $\phi_{N}$ is the canonical solution of level $N$ given by
\begin{eqnarray*}
\phi_{N}(x,y)=x(y+1)^{N-1}\bl \frac{y}{y+1}.
\end{eqnarray*}
\label{mthm}
\end{thm}
\begin{thm}[\cite{alkauskas-comm}]Suppose, two projective flows $\phi(\m{x})\neq x\bl y$ and $\psi(\m{x})\neq x\bl y$ with rational vector fields $\varpi\bl\varrho$ and $\alpha\bl\beta$ commute. Suppose $y\varpi-x\varrho\neq 0$. Then $\phi$ and $\psi$ are level $1$ algebraic flows. For any vector field of a level $1$ algebraic flow, the set of vector fields which commute with it, form a $2$-dimensional real vector space.
\end{thm}
The commutativity of vector fields is more thoroughly treated in Section \ref{lie-brackets}. The reference \cite{alkauskas-comm} contains explicit construction of such commuting algebraic projective flows. The $2$-dimensional flow is said to be of level $N$, if its orbits are given by $\mathscr{W}=\mathrm{const.}$ for a rational $N$-homogeneous function $\mathscr{W}$; see Definition \ref{ab-3-d} for a $3$-dimensional analogue. We note that the notion of ``level" in both theorems in this section is the same, and it means the degree of homogeneity of orbits. Indeed, the orbits of the canonical flow $\phi_{N}$ are curves $xy^{N-1}=\mathrm{const}.$
\begin{Example} The cubic algebraic flow in Note \ref{note-alg} is exactly of level $1$, since in this case $\mathscr{W}=x^{2}(x-y)^{-3}y^2$. Algebraic flow which commutes with it is constructed in \cite{alkauskas-comm}. It is given by cubic functions $a\bl b$, where implicitly $a$ is given by
\begin{eqnarray*}
(9ax-8x^2-3ay)x^2y^2+(3y+6y^2-x)a(ay-3ax+3x^2)^2=0,
\end{eqnarray*}
only the branch with $\lim\limits_{z\rightarrow 0}\frac{a(xz,yz)}{z}=x$ is chosen, and the function $b$ is given by
\begin{eqnarray*}
b=\frac{a^2(y-3x)}{x^2}+3a.
\end{eqnarray*}
\end{Example}
\section{Projective flows as non-linear PDE's}In this section we will generalize the result given by (Proposition 3, \cite{alkauskas}) and show how to derive a system of non-linear second order PDEs for a projective flow. 
\begin{prop}
\label{prop-trans}
Let $n\in\mathbb{N}$, and $\varpi^{(i)}$, $i=1,2,\ldots,n$, be arbitrary smooth\\ $2$-homogeneous functions, and assume $u:\mathbb{R}^{n}\mapsto\mathbb{R}$ is a smooth function which satisfies 
\begin{eqnarray}
\sum\limits_{i=1}^{n}u_{x_{i}}(\varpi^{(i)}-x_{i})=-u.
\label{n-dim}
\end{eqnarray}
Let us define $\mathcal{U}=\sum_{i=1}^{n}x_{i}u_{x_{i}}-u$. 
Then
\begin{eqnarray}
\sum\limits_{i=1}^{n}\mathcal{U}_{x_{i}}(\varpi^{(i)}-x_{i})=-2\mathcal{U}.
\label{n-dim2}
\end{eqnarray}
\end{prop}
So, the function $\mathcal{U}$, which is a transform of $u$, satisfies the same linear PDE as $u$, only with an additional factor ``2".
\begin{proof}
Indeed, let us write the equation for $u$ as 
\begin{eqnarray*}
\sum\limits_{i=1}^{n}u_{x_{i}}\varpi^{(i)}=\mathcal{U}.
\end{eqnarray*}
Now, differentiate this with respect to $x_{j}$, $j=1,\ldots,n$. We obtain
\begin{eqnarray*}
\sum\limits_{i=1}^{n}u_{x_{i}x_{j}}\varpi^{(i)}+\sum\limits_{i=1}^{n}u_{x_{i}}
\varpi^{(i)}_{x_{j}}
=\mathcal{U}_{x_{j}},\quad j=1,2,\ldots,n.
\end{eqnarray*}
Now, multiply this equality by $x_{j}$, and add over $j$. Since 
$\sum_{j=1}^{n}x_{j}\varpi^{(i)}_{x_{j}}=2\varpi^{(i)}$ because of $2-$homogeneity (Euler's identity), and also 
$\mathcal{U}_{x_{i}}=\sum_{j=1}^{n}x_{j}u_{x_{i}x_{j}}$, we obtain
\begin{eqnarray*}
\sum\limits_{i=1}^{n}\mathcal{U}_{x_{i}}\varpi^{(i)}+\sum\limits_{i=1}^{n}2u_{x_{i}}\varpi^{(i)}= \sum\limits_{i=1}^{n}x_{i}\mathcal{U}_{x_{i}}.
\end{eqnarray*}
This can be rewritten as 
\begin{eqnarray*}
\sum\limits_{i=1}^{n}\mathcal{U}_{x_{i}}(\varpi^{(i)}-x_{i})=-2\sum\limits_{i=1}^{n}u_{x_{i}}\varpi^{(i)}=-2\mathcal{U},
\end{eqnarray*}
because of (\ref{n-dim}), and this is the needed identity.
\end{proof}
\begin{Note}
\label{legendre}
 Let $n=3$, and $u=u(x,y,z)$. Consider the transformation 
\begin{eqnarray*}
\mathcal{X}=u_{x},\quad\mathcal{Y}=u_{y},\quad\mathcal{Z}=u_{z},\quad
\mathcal{U}=xu_{x}+yu_{y}+zu_{z}-u.
\end{eqnarray*}
So, $\mathcal{U}$ is as above for $n=3$. Suppose, $u$ satisfies a certain PDE. We can rewrite this PDE for a new function $\mathcal{U}$ as a function in three new variables $\mathcal{X}$, $\mathcal{Y}$ and $\mathcal{Z}$. Suppose that
\begin{eqnarray*}
\frac{D(\mathcal{X},\mathcal{Y},\mathcal{Z})}{D(x,y,z)}=\begin{vmatrix}
u_{xx} & u_{xy} & u_{xz}\\
u_{xy} & u_{yy} & u_{zy}\\
u_{xz} & u_{yz} & u_{zz} 
\end{vmatrix}\neq 0.
\end{eqnarray*} 
Then the inverse formulas, as can be checked, are exactly reciprocal; see (\cite{ficht1}, Chapter VI, Section 4.) That is,
\begin{eqnarray*}
x=\mathcal{U}_{\mathcal{X}},\quad y=\mathcal{U}_{\mathcal{Y}},\quad 
z=\mathcal{U}_{\mathcal{Z}},\quad 
u=\mathcal{X}\mathcal{U}_{\mathcal{X}}+\mathcal{Y}\mathcal{U}_{\mathcal{Y}}
+\mathcal{Z}\mathcal{U}_{\mathcal{Z}}-\mathcal{U}. 
\end{eqnarray*} 
This transformation is known as \emph{Legendre's transformation}, and is just a glimpse into a classical, wide and profound subject of \emph{change of variables} in a differential calculus (\cite{ficht1}, Chapter VI, Section 4). \\

For example, let $u$ satisfies (\ref{pde}). In terms of new variables, this rewrites as a non-linear PDE
\begin{eqnarray*}
\mathcal{X}\varpi\Big{(}\mathcal{U}_{\mathcal{X}},\mathcal{U}_{\mathcal{Y}},\mathcal{U}_{\mathcal{Z}}\Big{)}
+\mathcal{Y}\varrho\Big{(}\mathcal{U}_{\mathcal{X}},\mathcal{U}_{\mathcal{Y}},\mathcal{U}_{\mathcal{Z}}\Big{)}
+\mathcal{Z}\sigma\Big{(}\mathcal{U}_{\mathcal{X}},\mathcal{U}_{\mathcal{Y}},\mathcal{U}_{\mathcal{Z}}\Big{)}=
\mathcal{U}.
\end{eqnarray*}
In the simplest $1$-dimensional case, let $\varpi=-x^2$, $\varrho=\sigma=0$. Then $u=\frac{x}{x+1}$. Calculating directly, we get that
\begin{eqnarray*}
\mathcal{U}(\mathcal{X})=-\Big{(}1-\sqrt{\mathcal{X}}\Big{)}^{2},
\end{eqnarray*}
and this indeed satisfies $-\mathcal{X}\mathcal{U}^{2}_{\mathcal{X}}=\mathcal{U}$. 
\end{Note}
To continue from Proposition \ref{prop-trans}, Cramer's rule and formulas (\ref{n-dim}) give values for $(x_{i}-\varpi^{(i)})$, and substituting this into (\ref{n-dim2}) gives the needed system of PDEs which does not explicitly involve a vector field, as in (\ref{pde}). In a $3$-dimensional case, which is the topic of the most of the current paper, this reads as follows.
\begin{prop}
Let $\phi(\m{x})=u(x,y,z)\bl v(x,y,z)\bl w(x,y,z)$ be a $3$-dimensional smooth projective flow. Then it satisfies the boundary condition (\ref{boundr}), and also the non-linear system of second order PDEs as follows. Let 
\begin{eqnarray*}
D=\begin{vmatrix}
u_{x} & u_{y} & u_{z}\\
v_{x} & v_{y} & v_{z}\\
w_{x} & w_{y} & w_{z} 
\end{vmatrix},
\end{eqnarray*}
\begin{eqnarray*}
D^{(x)}=\begin{vmatrix}
u & u_{y} & u_{z}\\
v & v_{y} & v_{z}\\
w & w_{y} & w_{z} 
\end{vmatrix},
\quad 
D^{(y)}=\begin{vmatrix}
u_{x} & u & u_{z}\\
v_{x} & v & v_{z}\\
w_{x} & w & w_{z} 
\end{vmatrix},
\quad
D^{(z)}=\begin{vmatrix}
u_{x} & u_{y} & u\\
v_{x} & v_{y} & v\\
w_{x} & w_{y} & w 
\end{vmatrix}.
\end{eqnarray*}
Then
\begin{eqnarray*}
& &(xu_{xx}+yu_{xy}+zu_{xz})D^{(x)}+
(xu_{xy}+yu_{yy}+zu_{yz})D^{(y)}\\
&+&(xu_{xz}+yu_{yz}+zu_{zz})D^{(z)}=2(xu_{x}+yu_{y}+zu_{z}-u)D.
\end{eqnarray*}
The same holds for $v$ and $w$ instead of $u$ (the multipliers $D,D^{(x)},D^{(y)},D^{(z)}$ remain intact).
\label{prop-2}
\end{prop}
This result generalizes and is compatible with (Proposition 3, \cite{alkauskas}), where it is given in a $2$-dimensional case. Note that the equation of Proposition \ref{prop-2} can be written as\scriptsize 
\begin{eqnarray*}
\begin{vmatrix}
2(xu_{x}+yu_{y}+zu_{z}-u)&xu_{xx}+yu_{xy}+zu_{xz} &xu_{xy}+yu_{yy}+zu_{yz} &xu_{xz}+yu_{yz}+zu_{zz}\\
u& u_{x} & u_{y} & u_{z}\\
v& v_{x} & v_{y} & v_{z}\\
w& w_{x} & w_{y} & w_{z} 
\end{vmatrix}=0.
\end{eqnarray*}\normalsize
\section{General setting for $3-$dimensional flows}
\label{sub-1.4}
The main weight of this series of 4 works falls on explicit integration of corresponding vector fields. In this section we will briefly outline the method to tackle the PDE (\ref{pde}) in dimension $3$. The dimension $2$ case is a bit simpler, and the method itself was developed in \cite{alkauskas-un}. Note that it is analogous to the method of integrating P\'{o}lya urns via an autonomous system of ODEs, the method which was developed in \cite{fdp}. A close relative of the PDE in question was also treated in the framework of P\'{o}lya urns \cite{flajolet-g-p}. What is to come next is but a brief sketch, and all the necessary details will become clear in special cases in Sections \ref{S4} and \ref{octahedral}.  \\
  
Let $\varpi\bl\varrho\bl\sigma$ be a vector field defined by a triple of $2$-homogeneous rational functions, giving rise to the projective flow $u\bl v\bl w$. With the PDE (\ref{pde}) and the boundary condition (\ref{init}) we associate an autonomous (where an independent variable $t$ is not present explicitly) and homogeneous system of ODEs:
\begin{eqnarray}
\left\{\begin{array}{l}
p'=-\varpi(p,q,r),\\
q'=-\varrho(p,q,r),\\
r'=-\sigma(p,q,r).
\end{array}
\right.
\label{sys-in}
\end{eqnarray}
This is an example of a \emph{quadratic differential equation}. We refer to \cite{hopkins1,hopkins2,kin,maier,mencinger} for more information on algebraic and dynamic side of this system.\\

Let $u(x,y,z)$ be the solution to the PDE (\ref{pde}) with the boundary condition as in the first equality of (\ref{boundr}). Similarly as in \cite{alkauskas-un}, we find that the function $u$ satisfies
\begin{eqnarray*}
u\Big{(}p(s)\v,q(s)\v,r(s)\v\Big{)}=p(s-\v)\v.
\end{eqnarray*}
In order this expression to yield the full closed-form formula for $u(x,y,z)$, we need to involve one more variable (the space where the function $u$ is defined is $\mathbb{R}^{3}$), since explicitly the above equation contains only $s$ and $\v$. We act in the following way. Suppose that the system (\ref{sys-in}) possesses two independent first integrals $\mathscr{W}(p,q,r)$ and $\mathscr{V}(p,q,r)$; that is, as a function in $(x,y,z)$, $\mathscr{W}$ satisfies
\begin{eqnarray}
\mathscr{W}_{x}\varpi+\mathscr{W}_{y}\varrho+\mathscr{W}_{z}\sigma=0,
\label{ffirst-int}
\end{eqnarray}
analogously for $\mathscr{V}$. 
Therefore, to fully solve the flow in explicit terms, together with the system (\ref{sys-in}), we require
\begin{eqnarray*}
\mathscr{W}(p,q,r)=1,\quad \mathscr{V}(p,q,r)=\xi,
\end{eqnarray*}
where $\xi$ is arbitrary, but fixed. Note that we can always achieve that $\mathscr{W}$ and $\mathscr{V}$ are homogeneous functions.  This gives $3$ variables which describe a $3$-variable function.
\begin{defin}
\label{ab-3-d}
Suppose that the differential system (\ref{sys-in}) possesses two first integrals $\mathscr{W}$ and $\mathscr{V}$ which are rational functions of homogeneity degrees $a$ and $b$, respectively. We say then that $\phi=u\bl v\bl w$ is \emph{an abelian flow} of level $(a,b)$. 
\end{defin}
In other words, abelian flows are projective flows whose orbits are algebraic space curves. Note, however, that differently form a $2$-dimensional case, the level is not an invariant under birational transformation. As before, we are dealing only with $1$-homogeneous birational transformations ($1-$BIR), since this is compatible with projective flows: if $\ell$ is such a transformation and $\phi(\m{x})$ is a projective flow, then so is $\ell^{-1}\circ\phi\circ\ell(\m{x})$ \cite{alkauskas}. To illustrate the non-invariance of the level, consider the following rational flow, as given in \cite{alkauskas}:
\begin{eqnarray*}
\psi_{N,M}(x,y,z)=x(z+1)^{N-1}\bl y(z+1)^{M-1}\bl\frac{z}{z+1},\quad N,M\in\mathbb{N}.
\end{eqnarray*}
The vector field of this flow is $\varpi\bl\varrho\bl\sigma=(N-1)xz\bl(M-1)yz\bl(-z^2)$, and the first integrals are $\mathscr{W}=xz^{N-1}$, $\mathscr{V}=yz^{M-1}$. So, the level of this abelian flow is $(N,M)$. However, as shown in \cite{alkauskas}, the flow $\psi_{N,M}$ is $1-$BIR equivalent to the flow $\psi_{\mathrm{gcd}(N,M),0}$. Thus, we need to know the classification of rational projective flows up to conjugation with $1-$BIR, and this is yet unsolved. Note, however, that in \cite{alkauskas-rat2} the method of constructing rational and algebraic projective flows is presented, inductively on a dimension. It is plausible that this method accounts for all rational or algebraic flows. Thus, the complete classification is much more feasible than it was highlighted in the end of \cite{alkauskas}. 

\section{Arithmetic and integrability of $3$-dimensional flows}
\label{jou}
In Figure \ref{diagram} we present an example of integral flow (meaning whose vector field is a triple of integral quadratic forms), which is non-abelian. In fact, this example belongs to Jouanolou \cite{jouanolou}; see also \cite{fdp, maciejewski, ollagnier,zolandek}. Jouanolou theorem claims that the system of ODEs
\begin{eqnarray*}
\left\{\begin{array}{l}
p'=-q^2,\\
q'=-r^2,\\
r'=-p^2,
\end{array}
\right.
\end{eqnarray*}
does not have a rational first integral. The factor $``-1"$, trivially, does not alter a conclusion; just consider $(p(-t),q(-t),r(-t))$. \emph{A posteriori}, it does not have two independent rational first integrals. So the projective flow with the vector field $y^2\bl z^2\bl x^2$ is non-abelian integral flow. Properly, Jouanolou theorem claims a bit more (see \cite{ollagnier} for details): for every polynomial $P\in\mathbb{C}[x,y,z]$, the equation
\begin{eqnarray}
F_{x}\cdot y^2+F_{y}\cdot z^2+F_{z}\cdot x^2=PF
\label{darboux}
\end{eqnarray}  
does not admit a non-trivial solution $F\in\mathbb{C}[x,y,z]$.  The ideas stem from the works of Jean-Gaston Darboux in 1878.  However, this scenario does not seem to occur in a superflow setting; see Problem \ref{prob-two} in the end of Section \ref{sec1.6} and the remark preceding it. See also a closely related paper \cite{pereira}, where the author shows how to compute inflection and higher order inflection points for holomorphic vector fields on the complex projective plane, and also  presents bounds for the number of first integrals on families of vector fields. Also, the paper \cite{ribeiro} asks for algorithms to check whether holomorphic foliations of the complex projective plane have an algebraic solution. It gives an alternative to the method of Jouanolou to construct foliations without such. The paper \cite{moulin} is closely related to our topic. It investigates polynomial homogeneous $3-$dimensional vector fields and gives necessary and sufficient conditions for the existence of $0$-homogeneous first integrals.   
\small
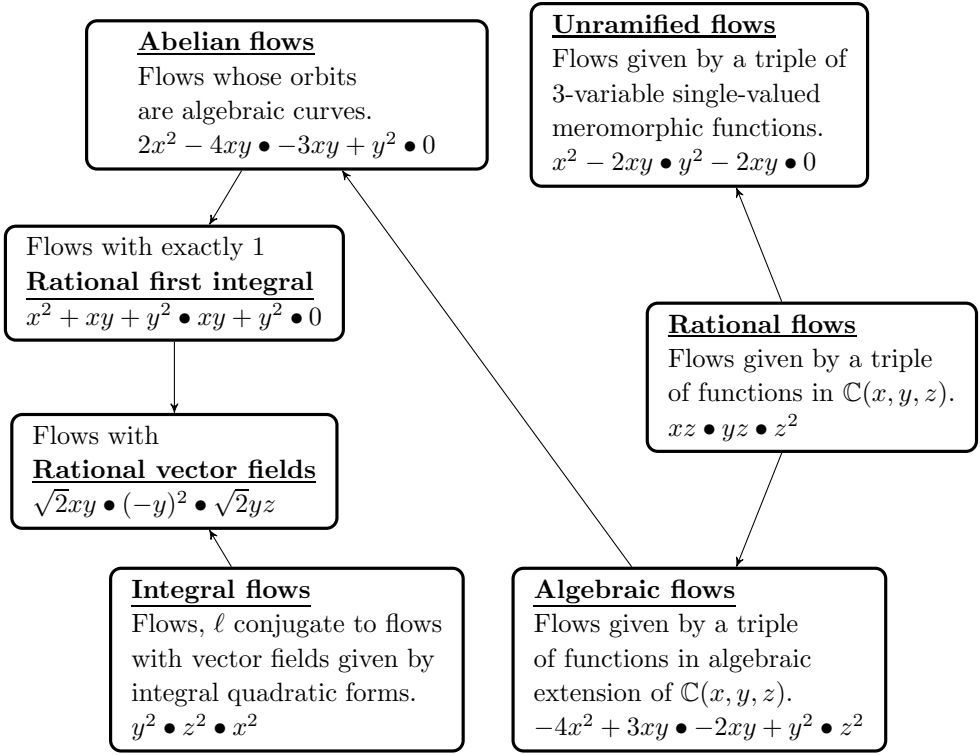
\begin{figure}
\begin{tikzpicture}[->,>=stealth']

 \node[state] (RVF) 
 {\begin{tabular}{l}
 Flows with exactly $1$\\
\underline{\textbf{Rational first integral}}\\
$x^2+xy+y^2\bl xy+y^2 \bl 0$ 
 \end{tabular}}; 
 
 \node[state,
  yshift=2.5cm,
  xshift=1.5cm,
  anchor=center,
  text width=5.2cm] (ABF) 
 {
 \begin{tabular}{l}
  \underline{\textbf{Abelian flows}}\\
  Flows whose orbits\\
  are algebraic curves.\\
  $2x^2-4xy\bl-3xy+y^2\bl 0$\\

 \end{tabular}
 };
 
 \node[state,
 yshift=-2.5cm,
 xshift=0cm,] (RVF2) 
 {\begin{tabular}{l}
 Flows with\\
\underline{\textbf{Rational vector fields}}\\ 
 $\sqrt{2}xy\bl(-y)^2\bl\sqrt{2}yz$ 
 \end{tabular}};
 
 \node[state,
  xshift=7.0cm,
  yshift=2.5cm,
  node distance=5cm,
  anchor=center] (UF) 
 {
 \begin{tabular}{l}
 \underline{\textbf{Unramified flows}}\\
 Flows given by a triple of\\
 $3$-variable single-valued\\
 meromorphic functions.\\
 $x^2-2xy\bl y^2-2xy\bl 0 $  
 \end{tabular}
 };
 
\node[state,
  yshift=-1.25cm,
  xshift=8.5cm,
  node distance=2cm,
  anchor=center] (RF) 
 {
 \begin{tabular}{l}
 \underline{\textbf{Rational flows}}\\
 Flows given by a triple\\
 of functions in $\mathbb{C}(x,y,z)$.\\
 $xz\bl yz\bl z^2$ 
 \end{tabular}
 }; 
  
 \node[state,
  xshift=7.0cm,
  yshift=-5.0cm,
  anchor=center] (ALF) 
 {
 \begin{tabular}{l}
  \underline{\textbf{Algebraic flows}}\\
 Flows given by a triple\\
 of functions in algebraic\\ extension of $\mathbb{C}(x,y,z)$.\\
 $-4x^2+3xy\bl-2xy+y^{2}\bl z^2$
 \end{tabular}
 };

\node[state,
  xshift=1.5cm,
  yshift=-5.0cm,
  anchor=center] (OK) 
 {
 \begin{tabular}{l}
  \underline{\textbf{Integral flows}}\\
 Flows, $\ell$ conjugate to flows\\
 with vector fields given by\\
integral quadratic forms.\\
$y^2\bl z^2\bl x^2$
 \end{tabular}
 };

 \path  (ABF)	edge  node[anchor=south,above]{} (RVF)
(RF)    edge                 (UF)         
(RF)    edge                 (ALF)                             
(ALF)  	edge                (ABF)
(RVF)  	edge                 (RVF2)
(OK)  	edge                 (RVF2);
\end{tikzpicture}
\caption{Diagram of arithmetic classification of projective $3-$dimensional flows (vector fields are given). The arrow stands for the proper inclusion - the facts which we do know. The intersection of unramified and algebraic flows is exactly the set of rational flows by an obvious reason. Everywhere we assume that a vector field is a triple of rational functions, necessarily $2-$homogeneous. The algebraic flow given is an extension (a direct sum with a flow $\frac{z}{1-z}$) of a flow taken from (\cite{alkauskas-ab}, Proposition 2). An abelian flow is also an extension made from (\cite{alkauskas-ab}, Proposition 4). It has two first integrals: $z$ and $x(x-y)^2y^2$. (\cite{alkauskas-ab}, Proposition 6) produces a flow with exactly $1$ rational first integral. Two first integrals of this flow are given by $\exp(-\frac{x}{y}-\frac{x^2}{2y^2})y$ and $z$, and only one of them is rational. The example of a flow with a vector field $y^2\bl z^2\bl z^2$ belongs to Jouanolou; it does not have a single rational first integral (see Section \ref{jou}). Finally, the flow with a vector field $\sqrt{2}xy\bl(-y)^2\bl\sqrt{2}yz$ can be explicitly integrated as $x(y+1)^{\sqrt{2}}\bl\frac{y}{y+1}\bl z(y+1)^{\sqrt{2}}$. It has no finite arithmetic structure and has an infinite monodromy; thus, not of interest in the framework of our approach. Note that, differently from a $2-$dimensional case, we know very little. 1) Is unramified flow necessarily an abelian flow? 2) Is a rational flow an integral flow? 3) Is unramified flow an integral flow? 
}
\label{diagram}
\end{figure}
\normalsize
\section{Superflows}
\label{sec1.6}
Now we arrive to the first main definition of this paper. The notion of projective superflows was introduced in \cite{alkauskas-un}.
\begin{defin}
\label{defin-proj}
Let $n\in\mathbb{N}$, $n\geq 2$, and $\Gamma\hookrightarrow{\rm GL}(n,\mathbb{R})$ be an exact representation of a finite group, and we identify $\Gamma$ with the image. We call the flow $\phi(\m{x})$ \emph{the projective $\Gamma$-superflow}, if 
\begin{itemize}
\item[i)]there exists a vector field $\mathbf{Q}(\m{x})=Q_{1}\bl\cdots\bl Q_{n}\neq 0\bl\cdots\bl 0$ whose components are $2$-homogeneous rational functions and which is exactly the vector field of the flow $\phi(\m{x})$, such that
\begin{eqnarray}
\gamma^{-1}\circ\mathbf{Q}\circ\gamma(\mathbf{x})=
\mathbf{Q}(\mathbf{x})\label{kappa}
\end{eqnarray}
 is satisfied for all $\gamma\in\Gamma$, and
\item[ii)] every other vector field $\mathbf{Q}'$ which satisfies (\ref{kappa}) for all $\gamma\in\Gamma$ is either a scalar multiple of $\mathbf{Q}$, or  its degree of a common denominator is higher than that of $\mathbf{Q}$. 
\end{itemize} 
The superflow is said to be \emph{reducible or irreducible}, if the representation  $\Gamma\hookrightarrow{\rm GL}(n,\mathbb{R})$ (considered as a complex representation) is reducible or, respectively, irreducible.
\end{defin} 
Thus, if $\phi$ is a superflow for a group $\Gamma$,  then it is uniquely defined up to conjugation with a linear map $\m{x}\mapsto t\m{x}$. This corresponds to multiplying all components of $\mathbf{Q}$ by $t$.
There is a slight abuse of notation in (\ref{kappa}). If $\gamma=(a_{i,j})_{i,j=1}^{n}$ and $\m{x}=x_{1}\bl x_{2}\bl\cdots\bl x_{n}$, by $\gamma(\m{x})$ or $\gamma\circ(\m{x})$ we mean $\sum_{j=1}^{n}a_{1j}x_{j}\bl\sum_{j=1}^{n}a_{2j}x_{j}\bl\cdots\bl \sum_{j=1}^{n}a_{nj}x_{j}$. In other words, if we consider a vector field $\mathbf{Q}$ as a map $\mathbb{R}^{n}\mapsto\mathbb{R}^{n}$, then $\gamma^{-1}\circ\mathbf{Q}\circ\gamma$ is just a composition of maps. Even more enlighteningly, the vector field of the superflow looks exactly the same in new coordinates corresponding to any linear change from the matrix group $\Gamma$.\\

Another strong motivation for considering $2$-homogeneous vector fields, apart from the arguments presented in (\cite{alkauskas-rat2}, Section 4), comes from the following observation. If a general rational vector field is invariant under conjugating with all matrices $\gamma$ from the group $\Gamma$, then each homogeneous component is invariant, too. However, suppose the representation is irreducible, and $q$ is a linear map. Then if $\gamma^{-1}\circ q\circ\gamma=q$ for any $\gamma$, \emph{Schur's lemma} \cite{kostrikin} tells that $\gamma$ is a multiple of the unit matrix. Thus, a vector field $x_{1}\bl x_{2}\bl\cdots\bl x_{n}$, and the flow it generates (not projective, of course), namely,
\begin{eqnarray*}
 F(\m{x};t)=e^{t}x_{1}\bl e^{t}x_{2}\bl\cdots\bl e^{t}x_{n},
\end{eqnarray*} 
  has the needed symmetry, and it is the unique, up to scaling, vector field without denominators with the needed property. Hence, the first interesting case appears only for $2$-homogeneous vector fields. 
\begin{Poly} We also introduce polynomial superflows as follows.
\begin{defin}
\label{defin-poly}
\emph{Polynomial $\Gamma$-superflows} are defined almost identically, with the difference that components of their vector fields should be homegeneous polynomials of the same even degree, and any other polynomial homogeneous (of even degree) vector field which has the given $\Gamma$-symmetry, is either a scalar multiple of the first, or its degree of homogenuity is higher.
\end{defin}
Since vector field is homogeneous, all the remarks from this and previous section, for example, about the orbits and arithmetic, apply to the polynomial superflows almost \emph{verbatim}.\\

The choice of even degree is the correct one. First, as is implied by the example referring to Shur's lemma, we cannot allow first degree vector field, since this will produce only a flow $F(\m{x},t)=\m{x}e^{t}$ whose group of symmetries is the whole $\mathrm{GL}(n,\mathbb{R})$. Suppose we restrict to odd degrees $\geq 3$. But then even in the simplest examples of groups polynomial superflows do not exist. For example, the following $1$-parameter family of vector fields
\begin{eqnarray}
x^{3}+ax(y^2+z^2)\bl y^3+ay(z^2+x^2)\bl z^{3}+az(x^2+y^2)
\label{family}
\end{eqnarray}  
has the full octahedral symmetry of order $48$, see Section \ref{octahedral}. The existence of a $1$-parameter family is not compatible with the definition of the superflow. We could additionally require solenoidality of a vector field, which is satisfied only for $a=-\frac{3}{2}$. Moreover, none of these flows has $x^2+y^2+z^2$ as its first integral. However, in even degree polynomial superflow case solenoidality seems always to be satisfied, and it is a subtle consequence of the symmetry. Hence our choice of even degree must be fundamentally more important. Also, the evenness of the vector fields shows that corresponding surfaces (See Note \ref{note-steiner} and Section \ref{sec5.2}) are one-sided. 
\end{Poly}

We can similarly define superflows over $\mathbb{C}$; this is the topic of the third part of this work \cite{alkauskas-super3}. Moreover, in real case we can talk about $\Gamma-$superflow$_{\pm}$, or $\Gamma$-superflow$_{+}$, depending whether $\det:\Gamma\mapsto\{\pm 1\}$ is surjective or not. Next, note that for the superflow $\Lambda$, defined in Section \ref{sec-2d}, the series of superflows $\phi_{S_{N+1}}$, $N\geq 2$, defined in Section \ref{symm-N} (the superflow $\phi_{\widehat{\mathbb{T}}}$ in Section \ref{S4} is an example), and the superflow $\phi_{\mathbb{O}}$, all vector fields are solenoidal: $\mathrm{div}\,\mathbf{Q}=0$. This also holds for the icosahedral, $4$-antiprismal, and $3$-prismal superflows in \cite{alkauskas-super2}. Next, in the octahedral example, the function $\mathscr{W}=x^2+y^2+z^2$ is the first integral of the vector field (so is in the icosahedral case). This quadratic form is, obviously, also an invariant of the group $O(3)$. On the other hand, the vector field for the dihedral superflow in Section \ref{dih} is solenoidal only for $d=1$. As one further remark, we  are dealing  with linear groups only, since isomorphic groups - the full tetrahedral group $\widehat{\mathbb{T}}$ in Section \ref{S4} and the octahedral group $\mathbb{O}$ in Section \ref{octahedral} - produce very different superflows. For example, the orbits of these two superflows are unbounded and bounded curves, respectively. Therefore, we refine the problem which was formulated in \cite{alkauskas-un}, as follows.
\begin{prob}
\label{prob-vienas}
Let $n\in\mathbb{N}$. Describe all groups $\Gamma$ inside ${\rm GL}(n,\mathbb{R})$ or ${\rm GL}(n,\mathbb{C})$ for which there exists the (projective or polynomial) $\Gamma$-superflow. Describe the properties of such a flow $\phi(\m{x})=\phi_{\Gamma}(\m{x})$ algebraically and analytically. If there exists the $\Gamma$-superflow with the vector field $\mathbf{Q}$,
\begin{itemize}
\item[i) ]what algebraic property of $\Gamma$ forces the vector field to be solenoidal? That is,
\begin{eqnarray*}
\sum\limits_{i=1}^{n}\frac{\partial}{\partial x_{i}}\,Q_{i}=0.
\end{eqnarray*}
\item[ii)]If $\Gamma$ is inside $U(n)$ or $O(n)$, when $\phi$ is a flow on spheres $\sum_{i=1}^{n}x_{i}^2=\xi$? That is,
\begin{eqnarray*}
\sum\limits_{i=1}^{n}x_{i}\cdot Q_{i}=0.
\end{eqnarray*}
\item[iii)]When it is both?
\end{itemize}
\end{prob}
The superflow which satisify ii) can be called \emph{spherical (polynomial or projective) superflow}. For such superflows we have the following: 
\begin{itemize}
\item[1)] The notion of polynomial and projective superflows essentially coincide;
\item[2)]For such superflows there appears to always hold the triple reduction of the corresponding differential system - see Theorems \ref{thm4}, \ref{thm-fin}, the main Theorem in \cite{alkauskas-super2}, and Problem 1 in \cite{alkauskas-super2} where we significantly strengthen part ii); therefore, the flow itself can be described, though in a very complicated way, in terms of a single abelian function;
\item[3)]For such superflows one can define three canonical real numbers (one of them is algebraic and one is from the extended ring of periods \cite{konzag}) - spherical constants for the superflow $\phi_{\Gamma}$, wired canonically into the group $\Gamma$ itself (see Section \ref{funda-period}); 
\item[4)]Such groups, as far as all examples show, are not just inside $O(n)$, but in fact inside $SO(n)$;
\item[5)]If $\Gamma$ is the orientation-preserving symmetry group of a certain polytope, then this polytope has an origin as a center of symmetry (for example, a tetrahedron does not have one);
\item[6)]The first integrals of the corresponding differential system turn out to be also invariants of the group - see Section \ref{sub-inv} and Problem \ref{prob-trys} (this is not the case for the tetrahedral superflow in Section \ref{S4}).
\end{itemize}
We will make all these statements more precise in \cite{alkauskas-super2}. A related question of investigation of homogeneous polynomial vector fields of degree $2$ on the $2$-sphere is treated in \cite{llibre}. 
 
Let $\mathbf{Q}=Q_{1}\bl\cdots\bl Q_{n}$ be a vector field corresponding to the superflow $\phi_{\Gamma}$. Taking to the common denominator, we write
\begin{eqnarray*}
\mathbf{Q}=\frac{P_{1}(\m{x})}{D(\m{x})}\bl\cdots\bl\frac{P_{n}(\m{x})}{D(\m{x})},
\end{eqnarray*}
and we say that this is in \emph{the reduced form}, if $\mathrm{g.c.d.}(P_{1},\ldots,P_{n},D)=1$. Call such $D$ \emph{the denominator of the superflow}.
\begin{Note}
\label{note-red1}
Over $\mathbb{R}$, there exist reducible superflows in dimension $3$ - the $4$-antiprismal and $3$-prismal superflows mentioned above. And over $\mathbb{C}$ they even exist in dimension $2$. The latter is the topic of Part III of this work. For example, one of the results in \cite{alkauskas-super3} reads as follows.
\begin{prop}
\label{prop-red}
Let $k\in\mathbb{N}\cup\{0\}$, and $\zeta$ be a primitive $(4k+3)$th root of unity. Let 
\begin{eqnarray}
\alpha=\begin{pmatrix}
\zeta & 0\\
0 & -\zeta^{-1}
\end{pmatrix}.\label{alfa}
\end{eqnarray}
Let us define
\begin{eqnarray*}
\phi(\m{x})=U(x,y)\bl V(x,y)=\sqrt[2k+1]{x^{2k+1}+y^{2k+2}}\bl y.
\end{eqnarray*}  
Then $\phi$ has a vector field $\frac{1}{2k+1}\frac{y^{2k+2}}{x^{2k}}\bl 0$, and is a projective superflow for the cyclic group of order $8k+6$ generated by the matrix $\alpha$. 
\end{prop} 
The crucial remark here is that cyclic group of order $8k+6$ generated by $\alpha$ is not the full group of symmetries for this superflow - there are uncountably many symmetries. This full group $\Gamma_{4k+3}$ is abelian in case $k\geq 1$ and non-abelian in case $k=0$. However, in the latter case $\Gamma_{3}$ still does not have finite non-abelian subgroups \cite{alkauskas-super3}. Therefore, for every $k\in\mathbb{N}_{0}$, $\phi$ is still a reducible superflow. See Note \ref{note-red2} to this account.
\end{Note}
\section{Symmetries and Noether's theorem}
In light of Jouanolou theorem, we note that the orthogonal symmetry of the vector field $(y^2,z^2,x^2)$ is non-trivial, it is cyclic of order $3$. This symmetry is not enough to guarantee even a single polynomial first integral. On the other hand, it is very natural to expect that a differential system for a superflow possesses exactly $(n-1)$ independent polynomial first integrals. In other words, we pose
\begin{prob}
\label{orbits-algebro}
Is it true that orbits of the superflow are always algebraic curves? 
\label{prob-two}
\end{prob} 
We emphasize that, on a philosophical level, the last question has a lot in common with Emmy Noether's theorem concerning conservation laws for dynamical systems whose Lagrangians have differentiable symmetries.\\

For example, let us consider a system of material points and bodies with ideal holonomic relations (relations which involve only positions of points, but not velocities), with an assumption that the forces which act are potential (\cite{pet}, Chapter 4, \S 12). Then the motion of such a system can be described in terms of Lagrange equations of the second kind:
\begin{eqnarray*}
\frac{\d}{\d t}\frac{\p L}{\p\dot{q}_{s}}-\frac{\p L}{\p q_{s}}=0,\quad s=1,2,\ldots,\ell,
\end{eqnarray*}   
where $L$ is the Lagrangian of the system. Now, if the Lagrangian is explicitly independent of the time variable, then the generalized \emph{energy} does not change, and it is the integral of the system:
\begin{eqnarray*}
\sum\limits_{s=1}^{\ell}\dot{q}_{s}\frac{\p L}{\p\dot{q}_{s}}-L=\textrm{const.}
\end{eqnarray*} 
This happens if the time is homogeneous. On equal grounds, if the space is homogeneous, then this leads to another integral, the conservation of \emph{impulse}. If the space is isotropic, this leads to the conservation of \emph{kinetic moment}. All three occur in the isolated system of material points of our physical space.\\

These ideas were generalized and greatly expanded by Emmy Noether in a theorem, whose significance in particle mechanics, Hamiltonian mechanics, field theory, electrodynamics, relativity theory, Chern-Simons theory, gauge theories (Yang-Mills), quantum field theories, and many other branches of mathematics and physics, is difficult to overstate.
\begin{thm}[E. Noether, 1918 \cite{pet}]To every infinitesimal transformation which leaves the Lagrangian intact, there corresponds the integral (conserved charge) of the system.
\end{thm} 
This also extends to gauge symmetries. In Hamiltonian mechanics case, where for the Hamiltonian we have the canonical system of equations 
\begin{eqnarray*}
\frac{\d q_{s}}{\d t}=\frac{\p H}{\p p_{s}},\quad \frac{\d p_{s}}{\d t}=-\frac{\p H}{\p q_{s}},\quad s=1,2,\ldots,\ell,
\end{eqnarray*}
the inverse Noether's theorem holds (conserved charges give symmetries), and conserved charges have an addition structure - they form a Lie algebra, the pairing being the Poission bracket. For more on this subject, see a review \cite{banados}. An expository paper \cite{novikov} on symmetries and solitons is even closer in spirit to our current work. The knowledge of enough symmetries allows sometimes to integrate the dynamical system completely. The author finishes the first Section with the following sentence (the italics is by S. Novikov):
\begin{itemize}
\item[] ``How about the yet-undiscovered laws governing the transformation of elementary particles or the evolution of the cosmos in early or later stages at very large size scales? What should we expect - that they will be supersymemtric or arbitrarily chosen? Probably the former; but it would be hard to know in advance just which sort of supersymmetry it will be. \emph{The higher reason is not predictable just from our present knowledge and concepts; we can only make partial prediction. (Luckily for us!)}".
\end{itemize}

These remarks also justify our term ``superflows".\\

The situation with Problem \ref{prob-two} is, at the level of ideas, very close, but, nevertheless, with many differences:

\begin{itemize}
\item[i)]first, we are looking only for algebraic first integrals; as we have seen in Section \ref{jou} with an Jouanolou example, a symmetry of order $3$ does not guarantee a single algebraic first integral;
\item[ii)]for high symmetry, when the superflow does exists, we still need to consider optimal, minimal vector field (of smallest degree) in order the maximal amount of symmetries to exist;
\item[iii)]our symmetries are not infinitesimal, but discrete.
\end{itemize}
Hence, positive resolution of Problem \ref{prob-two} might be considered as a discrete and algebro-geometric analogue to Noether's theorem. This question with be dealt with in \cite{alkauskas-super2}.
\section{Invariants}
\label{sub-inv}
 We briefly recall the topic of invariants, which we will need further, and which is the classical XIX-th century topic in the origins of algebraic geometry \cite{benson,kostrikin,verma}. 

\begin{defin}Let $G\subset\mathrm{GL}(n,\mathbb{C})$ be a finite group. A polynomial function $P:\mathbb{C}^{n}\mapsto\mathbb{C}$ is \emph{an invariant} of the group $G$, if $P\circ\gamma(\m{x})=P(\m{x})$ for each $\m{x}\in\mathbb{C}^{n}$ and each $\gamma\in G$ treated as a linear map. The polynomial function $P$ is called \emph{a relative invariant}, if $P\circ\gamma(\m{x})=\omega_{\gamma}P(\m{x})$ for a certain root of unity $\omega_{\gamma}$. In real case, that is, $G\subset\mathrm{GL}(n,\mathbb{R})$, $\omega_{\gamma}=\pm 1$.
\end{defin}
If we want an action of $G$ on homogeneous polynomials to be associative, we must define 
\begin{eqnarray*}
\gamma(P)(\m{x})=P\circ\gamma^{-1}(\m{x}).
\end{eqnarray*} 
However, we will not use associativity.
\begin{defin}An invertible linear transformation $V:\mathbb{C}^{n}\mapsto\mathbb{C}^{n}$ is called \emph{a pseudoreflection}, if all but one of its eigenvalues are equal to $1$. If this exceptional eigenvalue is equal to $-1$, then a pseudoreflection simply becomes \emph{a reflection}. This is the case if the matrices we consider are in $\mathrm{GL}(n,\mathbb{R})$. 
\end{defin}  
\begin{thmm}Let $G\subset\mathrm{GL}(n,\mathbb{C})$ be a finite group.
\begin{itemize}
\item[1.](Classical \cite{kostrikin}). The ring of polynomial invariants of $G$ has a transcendence degree $n$, and it is generated by $n$ algebraically independent forms $f_{1},f_{2},\ldots,f_{n}$, and (possibly) one additional form $f_{n+1}$.
\item[2.](Chevalley-Shephard-Todd \cite{chevalley, shephard}). The ring of invariants is a polynomial ring (so, no $f_{n+1}$) if and only if the group is generated by pseudoreflections.
\end{itemize}
\end{thmm}

\begin{Example}
Consider the $4$th order cyclic group $\Big{\{}\begin{pmatrix} i & 0\\ 0 & -i\\ \end{pmatrix}^{s},s\in\mathbb{Z}_{4}\Big{\}}\subset SU(2)$. Its invariant ring is generated by $\{xy,x^4,y^{4}\}$. The conjugate of this group, which is inside $SO(2)$ and is generated by the matrix $\begin{pmatrix} 0 & -1\\1 & 0 \\ \end{pmatrix}$, has a ring of invariants generated by $\{x^2+y^2, x^2y^2,x^3y-xy^3\}$.
\end{Example}

The next proposition follows easily from the definition of the superflow. 
\begin{prop}The denominator of the projective superflow is a polynomial relative invariant of the linear group $\Gamma$.
\end{prop}
Of course, since the vector field is a collection of $2$-homogeneous functions, the invariant in question is in fact a form. Note, however, that in all discovered cases the denominator is an invariant (but mind Example \ref{ex7} in Section \ref{sub5.1}).\\

In relation to P\'{o}lya urns we know that the notion of the urn to be \emph{balanced} \cite{fdp,flajolet-g-p} corresponds to the vector field being homogeneous in flow setting \cite{alkauskas-un}. Next, $1$-homogeneity of the function $F(\m{x},0)=\m{x}$ and the series (\ref{expl-taylor}) makes natural the question what happens to the flow if the vector field is $2$-homogeneous. This motivates investigation of projective flows.\\

The notion of the superflow is another huge motivation. Indeed, since a constant function is automatically an invariant of any linear group, any projective superflow with a denominator produces a family of general flows. For example, if we alter the vector field in Section \ref{sub5.1} as follows
\begin{eqnarray*}
\mathbf{C}_{a}(\m{x})=
\frac{y^3z-yz^3}{x^2+y^2+z^2+a}\bl\frac{z^3x-zx^3}{x^2+y^2+z^2+a}\bl\frac{x^3y-xy^3}{x^2+y^2+z^2+a},\quad a\in\mathbb{R,}
\end{eqnarray*}
then this vector field has the same $24$-fold symmetry, and produces a flow for every $a$. The uniqueness property, however, is not preserved. Thus, we can talk about \emph{superflows} only in a projective flow or polynomial flow (homogeneous vector field) setting. \\

We will see that sometimes the first integrals of the corresponding differential system for the superflow (see (\ref{sys-in}) in a $3$-dimenioanl case) are not the invariants of the group, like in Sections \ref{gener-2d} or \ref{dih-5}, but sometimes they are, like in Sections \ref{octahedral} and \ref{hyper-5}. This seems to be very closely related to Problem \ref{prob-vienas} and remarks just succeeding it. If the first integrals are indeed the invariants of the group, we may strengthen Problem \ref{prob-two} as follows.
\begin{prob}
\label{prob-trys}
If there are exactly $(n-1)$ first integrals of the corresponding differential system for the superflow $\phi_{\Gamma}$, and they are all the invariants of the group $\Gamma$, how to distinguish $1$ or $2$ remaining invariants which are not the first integrals?
\end{prob}
See the beginning of Section \ref{sub5.1}, where $x^2+y^2+z^2$ and $x^4+y^4+z^4$ are also the first integrals for the octahedral superflow, while the rest two are not.

%% file: sf-chap2.tex
\chapter{$2$-dimensional case}
\label{gener-2d}
\section{The group $S_{3}\simeq\mathbb{D}_{3}$}
\label{sec-2d}
Let $\lambda(x,y)\bl \lambda(y,x)$ be the flow generated by the vector field $x^2-2xy\bl y^2-2xy$. The formulas are valid over $\mathbb{R}$ as well as over $\mathbb{C}$. In the latter case we even get that the flow is unramified. For the last property, see Definition \ref{defin-unra} below. \\

We remind that \emph{Dixonian elliptic functions} ${\rm sm}$ and ${\rm sm}$ are defined by ${\rm sm}(0)=0$, ${\rm cm}(0)=1$, ${\rm sm}^3+{\rm cm}^3\equiv 1$, ${\rm sm}'={\rm cm}^2$, ${\rm cm}'=-{\rm sm}^2$ \cite{dixon, flajolet-c}.  
\begin{thm}[\cite{alkauskas-un}]
The function $\lambda(x,y)$ can be given the analytic expression
\begin{eqnarray*}
\lambda(x,y)&=&\frac{\v\big{(}c\v^2-sc^2y\v+s^2xy\big{)}^{2}}
{y\big{(}x-c^3y\big{)}\big{(}c^2\v^2-sx\v+s^2cxy\big{)}},\\
\lambda(y,x)&=&\frac{\v\big{(}c^2\v^2-sx\v+s^2cxy\big{)}^{2}}
{x\big{(}x-c^3y\big{)}\big{(}c\v^2-sc^2y\v+s^2xy\big{)}};
\end{eqnarray*}
here $\v=\v(x,y)=[xy(x-y)]^{1/3}$, and $s={\rm sm}(\v),c={\rm cm}(\v)$ are the Dixonian elliptic functions. The function $\lambda(x,y)$ is a single-valued $2$-variable meromorphic function. $\Lambda$ is an unramified flow, also the superflow (projective as well as polynomial), whose group of symmetries is isomorphic to $S_{3}\simeq\mathbb{D}_{3}$, a group of order $6$, generated by two matrices
\begin{eqnarray*}
\left(\begin{array}{cc}0 & 1 \\1 & 0 \\ \end{array}\right),\quad
\left(\begin{array}{cc}1 & -1 \\0 & -1 \\ \end{array}\right).
\end{eqnarray*}
\label{thm2}
\end{thm} 
Note that it is sufficient to give the formula for $\lambda(x,y)$. But swapping $x$ and $y$ changes $\v$ into $-\v$, and in \cite{alkauskas-un} we used the known formulas to rewrite ${\rm sm}(-\v)$ and ${\rm cm}(-\v)$ in terms of ${\rm sm}(\v)$ and ${\rm cm}(\v)$. \\

If we conjugate the group $S_{3}$ so it will end up inside $O(2)$, the corresponding vector field with a $6$-fold symmetry is given by
\begin{eqnarray}
\varpi\bl\varrho=2xy-x^2+y^2\bl 2xy +x^2-y^2.
\label{lambent-6}
\end{eqnarray} 
This is, up to the sign, a formula in Proposition \ref{prop5-def}, $\varpi\bl\varrho=\varpi_{3}\bl\varrho_{3}$ (see below).
\begin{Note}
\label{dihedral-beltrami}
We remark that numerators of the vector fields for the superflows give rise to \emph{lambent flows}, as explained in Notes \ref{tetra-beltrami} and \ref{octa-beltrami}, and also in \cite{alkauskas-super2} (all these relate to dimension $n=3$). In dimension $n\neq 3$ there exists no simple analogue of the curl operator (the dimension of differential $1$-forms, namely, $n$, and $2$-forms, namely, $\binom{n}{2}$, are not equal), but a slightly relaxed requirements can still be applied. In case $n=3$ these two are intricately related. Thus, the scalar multiple of the vector field (\ref{lambent-6}) generates the following vector field $\mathfrak{D}=\mathfrak{a}\bl\mathfrak{b}$ \cite{alkauskas-beltrami}:
\begin{eqnarray*}
\mathfrak{a}&=&-\cos y+\sqrt{3}\sin\Big{(}\frac{x}{2}\Big{)}\sin\Big{(}\frac{\sqrt{3}y}{2}\Big{)}+\cos\Big{(}\frac{\sqrt{3}x}{2}\Big{)}\cos\Big{(}\frac{y}{2}\Big{)},\\
\mathfrak{b}&=&-\cos x+\sqrt{3}\sin\Big{(}\frac{y}{2}\Big{)}\sin\Big{(}\frac{\sqrt{3}x}{2}\Big{)}+\cos\Big{(}\frac{\sqrt{3}y}{2}\Big{)}\cos\Big{(}\frac{x}{2}\Big{)}.
\end{eqnarray*} 
It has such properties.
\begin{itemize}
\item[i)]The vector field $\mathfrak{D}$ has a $6$-fold dihedral symmetry, given by the same matrices as in Proposition \ref{prop5-def}, $2d+1=3$;
\item[ii)]the Taylor series for $\mathfrak{D}$ contains only even compound degrees, and it starts from $\frac{3}{8}\varpi\bl\frac{3}{8}\varrho$;
\item[iii)]it satisfies the vector \emph{Helmholtz equation} $\nabla^{2}\mathfrak{D}=-\mathfrak{D}$, where $\nabla^{2}$ is a vector Laplace operator;
\item[iv)]$\mathrm{div}\,\mathfrak{D}=0$.
\end{itemize}
However, as explained after Notes \ref{tetra-beltrami} and \ref{octa-beltrami}, lambent flows do not have an algebro-geometric side, but rather a complicated dynamical one, as is clear from \cite{etnyre} in case of general Beltrami vector fields.
\end{Note}
\section{Unramified flows}In this section we will make a small step aside, since this complements perfectly results about the flow $\lambda(x,y)\bl\lambda(y,x)$ in the previous Section. \\

The notion of the flow being \emph{unramified}, or with a trivial monodromy, was discovered in \cite{alkauskas-un}. It deserves, to our opinion, a much broader attention from specialists in the fields of algebraic and differential geometry. Hence we present a precise definition. Here we formulate only the case of a general global flow in $\mathbb{R}^{n}$, the case of flows on algebraic varieties being described and investigated in \cite{alkauskas-un2}. The definition applies to all, not just projective, flows.
\begin{defin}
\label{defin-unra}
Suppose $n\in\mathbb{N}$, $\m{x}\in\mathbb{R}^{n}$, and a vector field $X=f_{1}\bl f_{2}\bl\cdots\bl f_{n}=\sum_{i=1}^{n}f_{i}\frac{\p}{\p x_{i}}$ is given by rational functions $f_{i}\in\mathbb{R}(x_{1},x_{2},\ldots,x_{n})$. Let $F(\m{x},t)=F_{1}(\m{x},t)\bl\cdots\bl F_{n}(\m{x},t)$ is the corresponding flow. We say that a flow is \emph{unramified}, if all $F_{j}(\m{x},t)$, $1\leq j\leq n$, are single-valued meromorphic functions in variables $x_{1},\ldots,x_{n},t\in\mathbb{C}$.
\end{defin} 
Let us define $f^{(0)}_{j}=x_{j}$, $f^{(1)}_{j}=f_{j}$, and
\begin{eqnarray*}
f^{(\ell)}_{j}=\sum\limits_{i=1}^{n}f_{i}\cdot\frac{\p}{\p x_{i}}f^{(\ell-1)}_{j}=X\big{(}f_{j}^{(\ell-1)}\big{)},\quad \ell\geq 2,\quad 1\leq j\leq n.
\end{eqnarray*}
Then as easily follows from (\ref{PDE-gen}) in dimension $n\in\mathbb{N}$ (exactly as in dimension $3$), one has
\begin{eqnarray}
F_{j}(\m{x},t)=\sum\limits_{\ell=0}^{\infty}\frac{1}{\ell!}f^{(\ell)}_{j}(\m{x})t^{\ell},\quad 1\leq j\leq n.
\label{taylor-gen}
\end{eqnarray}
The property of the flow being unramified is thus tantamount to a property that this series has a positive radius of convergence in $t$, and analytically extends as a single-valued meromorphic function in $(n+1)$ variables, for all $j$. In a projective case the number of variables reduces to $n$. \\

If a flow $\phi$ is ramified, there are several ways to demonstrate this. For example, we can find a suitable birational transformation $\ell$ ($1$-homogeneous in a projective flow case if we want to end up with a projective flow), such that $\ell^{-1}\circ\phi\circ\ell$, if expanded as 
\begin{eqnarray*}
F_{j}(x_{1},x_{2},\ldots,x_{n},t)=
\sum\limits_{j=0}^{\infty}x_{1}^{j}w_{j}(x_{2},\ldots,x_{n},t)
\end{eqnarray*}
has one of the functions $w_{j}$ certainly ramified. Functions $w_{j}$ are found by solving inductivelly differential equations. This was our method in \cite{alkauskas-un} to find ramification of many flows with a vector field given by a pair of quadratic forms. Some of these turned out to have ramifications of algebraic, some - of logarithmic type. However, if the flow is indeed unramified, it might take some ingenuity to demonstrate this. Of course, this duplicity in the degree of hardship in demonstrating two opposite phenomena pervades many branches of mathematics.
\begin{Example} If $f\in Cr_{n}(\mathbb{R})$ (a birational transformation of the affine space), then for a fixed $\m{a}\in\mathbb{R}^{n}$, sure,
\begin{eqnarray*}
F(\m{x},t)=f^{-1}\big{(}f(\m{x})+\m{a}t\big{)}
\end{eqnarray*} 
is an unramified flow, given by a collection of rational functions. For projective case, Theorem \ref{mthm} describes all projective rational flows.
\end{Example}
\begin{Example}
Consider the $1$-dimensional vector field $\varpi(x)=x(x+1)$. It generates the flow
\begin{eqnarray*}
F(x,t)=\frac{xe^{t}}{x+1-xe^{t}},
\end{eqnarray*}
and so it is unramified.
\end{Example}
\begin{Example} Since the Taylor series for $\sm(u)$ and $\cm(u)$ contain only powers $u^{3m+1}$ and $u^{3m}$, $m\in\mathbb{N}_{0}$, respectively, formulas in Theorem \ref{thm2} show that the vector field $x^2-2xy\bl y^2-2xy$ produces an unramified elliptic projective flow.
\end{Example}
\begin{Example}
 In \cite{alkauskas-ab}, Appendix A, the vector field $x^2-xy\bl y^2-2xy$ is investigated. The orbits are given by the plane curves $x^3y^2(3x-2y)=\mathrm{const}.$ The differential system (\ref{sys-in}), now coupled with the first integral, reads as
 \begin{eqnarray*}
\left\{\begin{array}{l@{\qquad}l}
1\equiv p^{3}q^2(3p-2q),\\
p'=-p^{2}+pq,\\
q'=-q^{2}+2pq.
\end{array}\right.
\end{eqnarray*}
This has a solution
\begin{eqnarray*}
p(u)&=&\frac{2A\sm^{3}(Bu)-A}{2\sm(Bu)\cm(Bu)},\\
q(u)&=&
\frac{4\sqrt{3}\sm^2(Bu)\cm^2(Bu)}{3A^2[1-2\sm^3(Bu)]},\end{eqnarray*} 
where $\sm$ and $\cm$ are the same Dixonian elliptic functions, $A=\frac{\sqrt[3]{4}}{\sqrt{3}}$, $B=-\frac{1}{\sqrt[3]{2}\sqrt{3}}$. Then we proceed in the usual way, write down the explicit solution using addition formulas for the elliptic functions. However, this time the demonstration of unramification involves also algebraic properties of the number field $\mathbb{Q}(\,\sqrt{3}\,)$. To say concisely, the pair $(p,q)$ solves the differential system also if we assume that in their expressions the symbol $``\sqrt{3}"$ stands for a negative value of the square root rather than positive. This observation is eesentially employed int the proof.
\end{Example}
In general, it is tempting to believe that unramified flows have algebraic curves as their orbits, hence the corresponding differential system possesses $(n-1)$ independent polynomial first integrals, and these flows are either rational flows, or arise from flows with orbits being lines, quadratics or elliptic curves, and only some very special ones of the latter flows are indeed unramified; see \cite{alkauskas-un,alkauskas-un2}.
\begin{figure}
\includegraphics[width=80mm,height=80mm,angle=-90]{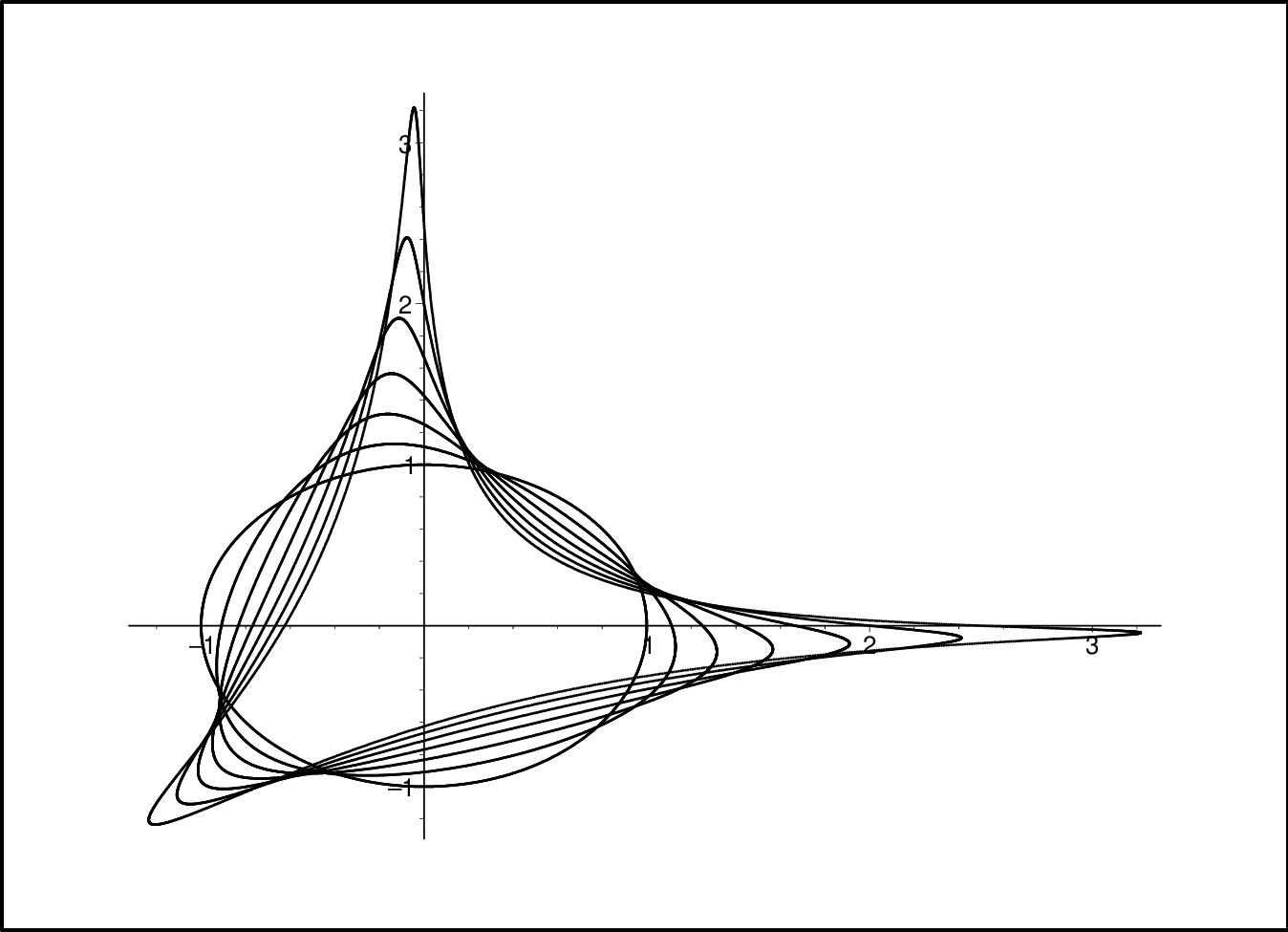}
\caption{Deformation of the unit circle $\mathbf{S}=\{x^{2}+y^{2}=1\}$ under the superflow $\Lambda$. The seven curves $\Lambda^{0.1j}(\mathbf{S})$, $j=0,\ldots 6$, are shown. The vector field is solenoidal, so the area inside each curve is constant and equal to $\pi$.}
\label{figure2}
\end{figure}

\section{Dihedral superflows}
\label{dih}
 Let $2d+1\geq 3$ be an odd integer, and let $\zeta=e^{\frac{2\pi i}{2d+1}}$ be a primitive $(2d+1)-$th root of unity. Consider the following two matrices
\begin{eqnarray*}
\alpha=\begin{pmatrix} \zeta & 0\\ 0 & \zeta^{-1}\\ \end{pmatrix},\quad
\beta=\begin{pmatrix} 0 & 1\\ 1 & 0\\ \end{pmatrix}.
\end{eqnarray*} 
Together they generate a dihedral group $\mathbb{D}_{2d+1}$ of order $4d+2$. The ring of invariants is generated by $\{xy, x^{2d+1}+y^{2d+1}\}$ and thus is a polynomial ring (dihedral group is generated by two reflections). \\

 Let us take $d=2$. The vector field, which is invariant under this group, is given by
\begin{eqnarray*}
\varpi\bl\varrho=\frac{y^{3}}{x}\bl\frac{x^{3}}{y}=\frac{y^4}{xy}\bl\frac{x^4}{xy},
\end{eqnarray*}
and it is the unique (up to conjugation with a homothety) vector field with this property whose denominator is of degree $\leq 2$. Note that the cyclic group of order $5$ which is generated by the matrix $\alpha$ ($2d+1=5$) does not give rise to the superflow (over $\mathbb{C}$), since there exists a family of vector fields which are invariant under this group, and this family is given by 
\begin{eqnarray*}
\varpi\bl\varrho_{a}=\frac{y^{3}}{x}\bl a\frac{x^{3}}{y},\text{ for any }a\in\mathbb{C}.
\end{eqnarray*}
So, this is not compatible with the definition of a superflow. We need to add the matrix $\beta$ to get one.\\

In a general case of the dihedral group $\mathbb{D}_{2d+1}$, the invariant vector field is given by
\begin{eqnarray*}
\varpi\bl\varrho=\frac{y^{d+1}}{x^{d-1}}\bl\frac{x^{d+1}}{y^{d-1}}=\frac{y^{2d}}{(xy)^{d-1}}\bl\frac{x^{2d}}{(xy)^{d-1}}.
\end{eqnarray*}A direct check shows that it is indeed the unique (up to scalar multiple, as always) vector field with this property whose denominator is of degree $\leq 2d-2$.\\

Now, we will conjugate matrices $\alpha$ and $\beta$ so that they will end up in $O(2)$, thus obtaining a superflow over $\mathbb{R}$. Indeed, let
$\gamma=\begin{pmatrix} \frac{1}{\sqrt{2}} & \frac{i}{\sqrt{2}}\\ \frac{i}{\sqrt{2}} &\frac{1}{\sqrt{2}}\\ \end{pmatrix}$. Then
\begin{eqnarray*} 
\gamma^{-1}\circ\beta\circ\gamma&=&\beta,\\
\gamma^{-1}\circ\alpha\circ\gamma&=&\begin{pmatrix}\frac{1}{2}(\zeta+\zeta^{-1}) & \frac{i}{2}(\zeta-\zeta^{-1})\\\frac{i}{2}(\zeta^{-1}-\zeta) & \frac{1}{2}(\zeta^{-1}+\zeta)\\ \end{pmatrix}=
\begin{pmatrix}\cos(\frac{2\pi}{2d+1}) & -\sin(\frac{2\pi}{2d+1})\\
\sin(\frac{2\pi}{2d+1}) & \cos(\frac{2\pi}{2d+1})\\ \end{pmatrix}:=\tilde{\alpha}.
\end{eqnarray*}
In particular, the group $\widetilde{\mathbb{D}}_{2d+1}=\langle \tilde{\alpha},\beta\rangle\subset O(2)$, and the vector field, which is invariant under this group, is given by $\gamma^{-1}\circ(\varpi\bl\varrho)\circ\gamma$. The ring of invariants is generated by $x^2+y^2$ and the form  of degree $2d+1$, given by \cite{verma}
\begin{eqnarray*}
\prod\limits_{p=0}^{2d}\Big{(}x\cos\frac{2\pi p}{2d+1}+y\sin\frac{2\pi p}{2d+1}\Big{)}.
\end{eqnarray*}
 We summarize the findings as follows.
\begin{prop}
\label{prop5-def}
Let $2d+1\geq 3$ be an odd integer, and the dihedral group $\mathbb{D}_{2d+1}\subset O(2)$ is generated by the matrices $\tilde{\alpha}$ and $\beta$. Let
\begin{eqnarray*}
P_{m}(x,y)=\Re\big{(}(x+iy)^{m}\big{)},\quad Q_{m}(x,y)=\Im\big{(}(x+iy)^{m}\big{)}
\end{eqnarray*} 
be the standard harmonic polynomials of order $m\in\mathbb{N}$. Then there exists the superflow $\phi_{\mathbb{D}_{2d+1}}$ with the vector field
\begin{eqnarray*}
\varpi_{2d+1}\bl\varrho_{2d+1}=\frac{P_{2d}(x,y)+(-1)^{d}Q_{2d}(x,y)}{(x^2+y^2)^{d-1}}\bl
\frac{(-1)^{d}P_{2d}(x,y)-Q_{2d}(x,y)}{(x^2+y^2)^{d-1}}.
\end{eqnarray*} 
The orbits of this superflow are curves of genus $d(2d-1)$ and are given by
\begin{eqnarray*}
\mathscr{W}_{2d+1}(x,y)&=&P_{2d+1}(x,y)-(-1)^{d}Q_{2d+1}(x,y)\\
&=&\frac{1+i(-1)^{d}}{2}(x+iy)^{2d+1}+\frac{1-i(-1)^{d}}{2}(x-iy)^{2d+1}=\mathrm{const.}
\end{eqnarray*}
\end{prop}
\begin{proof}We are only left to prove the statement about the orbits. One needs to verify that
\begin{eqnarray*}
\Big{(}\frac{\p P_{2d+1}}{\p x}-(-1)^{d}\frac{\p Q_{2d+1}}{\p x}\Big{)}\big{(}P_{2d}+(-1)^{d}Q_{2d}\big{)}&&
\\
+\Big{(}\frac{\p P_{2d+1}}{\p y}-(-1)^{d}\frac{\p Q_{2d+1}}{\p y}\Big{)}\big{(}(-1)^{d}P_{2d}-Q_{2d}\big{)}&\equiv& 0.
\end{eqnarray*}
This is immediate if we note that $P_{m}=\frac{1}{2}((x+iy)^{m}+(x-iy)^{m})$, $Q_{m}=\frac{1}{2i}((x+iy)^{m}-(x-iy)^{m})$. So,
\begin{eqnarray*}
\frac{\p P_{m+1}}{\p x}=(m+1)P_{m},&\quad& \frac{\p P_{m+1}}{\p y}=-(m+1)Q_{m},\\
\frac{\p Q_{m+1}}{\p x}=(m+1)Q_{m},&\quad& \frac{\p Q_{m+1}}{\p y}=(m+1)P_{m},
\end{eqnarray*}
and the answer follows.
\end{proof}
Note that in case $d=1$, the superflow $\Lambda$ is linearly conjugate to $\phi_{\mathbb{D}_{3}}$. We investigated this superflow in detail in \cite{alkauskas-un}. Now, we are ready to classify all $2-$dimensional real superflows.
\begin{thm}
\label{thm-dih}
For each $d\in\mathbb{N}$, there exists the superflow whose group of symmetries is the dihedral group $\mathbb{D}_{2d+1}$. There does not exist a superflow whose group of symmetries is $\mathbb{D}_{2d}$, $d\in\mathbb{N}$, or a cyclic group $\mathbb{C}_{d}$, $d\in\mathbb{N}$.
\end{thm}
\begin{Note} However, there exists two $3$-dimensional reducible superflows $\phi_{\mathbb{P}_{3}}$ and $\phi_{\mathbb{A}_{4}}$ whose groups of symmetries are isomorphic to dihedral groups $\mathbb{D}_{6}$ of order $12$, and $\mathbb{D}_{8}$ of order $16$, respectively. As representations, these groups are, correspondingly, symmetry groups of a $3$-prism and a $4$-antiprism. 
\end{Note}
\begin{Poly} Exactly the same conclusion as in Theorem \ref{thm-dih} follows for polynomial superflows. They are produced by vector fields 
\begin{eqnarray*}
P_{2d}(x,y)+(-1)^{d}Q_{2d}(x,y)\bl(-1)^{d}P_{2d}(x,y)-Q_{2d}(x,y).
\end{eqnarray*} 
These are solenoidal, whereas for $d\geq 2$ projective superflows are not. This shows that the notion of polynomial superflows is even more fundamental. See Section \ref{hyper-5} to this account.
\end{Poly}

\begin{proof}We know that the only finite subgroups of $O(2)$ are dihedral and cyclic groups. Let $\Gamma\subset O(2)$ be isomorphic to $\mathbb{D}_{2d}$, $d\in\mathbb{N}$. Then $-I\in\Gamma$. If we apply (\ref{kappa}) for $\gamma=-I$, it gives that the vector field is trivial.\\

Let $\Gamma\subset O(2)$ be isomorphic to $\mathbb{C}_{d}$, $d\in\mathbb{N}$, and $\Gamma$ is generated by $\gamma$. Conjugating with an appropriate matrix, we can achieve that inside $U(2)$, $\gamma$ is conjugate to
$\hat{\gamma}=\begin{pmatrix} \zeta & 0\\ 0 & \pm\zeta^{-1}\\ \end{pmatrix}$ for a certain root of unity $\zeta$. \\

Suppose, it is a ``+" sign. Now, if a vector field $\varpi\bl\varrho$ is invariant under conjugation with $\hat{\gamma}$, so is the vector field $a\varpi\bl\varrho$, $a\in\mathbb{C}$, and this contradicts the definition of a superflow. We only should be cautious that $\varpi\neq 0$. In our case it cannot happen, since vector fields with lowest degree denominators, invariant under conjugating  with $\hat{\gamma}$, are always of the form $a\varpi(x,y)\bl\varpi(y,x)$.\\

Suppose, we have a ``-" sign. The trace of the matrix $\hat{\gamma}$ is $2(\zeta-\zeta^{-1})$, and if it is not real, this matrix is not conjugate to any matrix in $\mathrm{GL}(2,\mathbb{R})$. The trace is real only for $\zeta=\pm 1$, and this does not produce a superflow. For example, for $\zeta=1$ a family of invariant vector fields is given by $ay^2\bl xy$. Note, however, that if $k\neq 0\text{ (mod }4)$, the cyclic group generated by $\hat{\gamma}$ produces a superflow over $\mathbb{C}$; see Note \ref{note-red1} and \cite{alkauskas-super3}.
\end{proof}
In the third part of this work \cite{alkauskas-super3} we will classify $2-$dimensional superflows over $\mathbb{C}$. For example, there exists an epimorphism $e:SU(2)\mapsto SO(3)$ with a kernel $\{\pm I\}$. So, $SU(2)$ contains the pre-images $e^{-1}(\mathbb{T})=\widetilde{\mathbb{T}}$ and $e^{-1}(\mathbb{O})=\widetilde{\mathbb{O}}$ of the tetrahedral group $\mathbb{T}$ and octahedral group $\mathbb{O}$, which are given in Sections \ref{S4} (Note \ref{note-red2}) and \ref{octahedral}. Thus, $\widetilde{\mathbb{T}}$ and $\widetilde{\mathbb{O}}$ are of order $24$ and $48$, respectively.  Further, let $\phi=\frac{1+\sqrt{5}}{2}$, and define
\begin{eqnarray}
\alpha=\begin{pmatrix}
-1 & 0 & 0\\
0 & -1 & 0\\
0 & 0 & 1 
\end{pmatrix},\quad
\beta=\begin{pmatrix}
0 & 0 & 1\\
1 & 0 &0\\
0 & 1 & 0 
\end{pmatrix},\quad
\gamma=\begin{pmatrix}
\frac{1}{2} & -\frac{\phi}{2} & \frac{1}{2\phi}\\
\frac{\phi}{2} & \frac{1}{2\phi} & -\frac{1}{2}\\
\frac{1}{2\phi} & \frac{1}{2} & \frac{\phi}{2} 
\end{pmatrix}.
\label{gen-ico}
\end{eqnarray}
These are matrices of orders $\alpha^2=\beta^3=\gamma^{5}=I$, and together they generate the icosahedral group $I\subset SO(3)$ of order $60$. Thus, $SU(2)$ contains $e^{-1}(\mathbb{I})=\widetilde{\mathbb{I}}$, and thus $\widetilde{\mathbb{I}}$ is of order $120$. The groups $\widetilde{\mathbb{T}}$, $\widetilde{\mathbb{O}}$ and $\widetilde{\mathbb{I}}$ are the so called \emph{binary tetrahedral, octahedral and icosahedral groups}. As a matter of fact, there cannot exist $2-$dimensional $\mathbb{C}$-superflows for groups $\widetilde{\mathbb{T}}$, $\widetilde{\mathbb{O}}$ or $\widetilde{\mathbb{I}}$, since $-I$ cannot be an element of the symmetry group for the superflow. Though potentially, there exist $2-$dimensional superflows over $\mathbb{C}$ with symmetries as subgroups of $\widetilde{\mathbb{T}}$, $\widetilde{\mathbb{O}}$ and $\widetilde{\mathbb{I}}$, not to mention other finite subgroups of $U(2)$ which exist in a fascinating variety \cite{coxeter,u2}.

%% file: sf-chap3.tex
\chapter{The dihedral superflow $\phi_{\mathbb{D}_{5}}$}
\label{dih-5}
\section{Basic properties}When $2d+1=5$, the vector field is $\varpi_{5}\bl\varrho_{5}$ is given by\small
\begin{eqnarray}
\varpi\bl\varrho=\frac{x^4+4x^3y-6x^2y^2-4xy^3+y^4}{x^2+y^2}\bl
\frac{x^4-4x^3y-6x^2y^2+4xy^3+y^4}{x^2+y^2}.\label{penki}
\end{eqnarray}\normalsize
\begin{Poly} This is the only section of this paper where results do not transfer automatically \emph{verbatim} (or with inessential alterations) to polynomial superflow case, since we would need to integrate the above vector field without a denominator $x^2+y^2$. However, as can be easily seen, analogous results hold and analysis can be carried out, leading to a different fundamental period (see Section \ref{fundamental-period}).
\end{Poly}
Vector field $\varpi\bl\varrho$ has the order $10$ group $\mathbb{D}_{5}$ as the group of its symmetries, and it gives rise to the superflow $\phi_{\mathbb{D}_{5}}$, which we will now explore in more detail. Cases $2d+1>5$ are investigated similarly.\\

Thus, let $\varpi\bl\varrho$ be given as in (\ref{penki}). The orbits of this superflow $\phi_{\mathbb{D}_{5}}=\gamma(x,y)\bl\gamma(y,x)$ are curves 
\begin{eqnarray*}
\mathscr{W}(x,y):=x^5-5x^4y-10x^3y^2+10x^2y^3+5xy^4-y^5=\mathrm{const}.
\end{eqnarray*}
For $\mathrm{const}.\neq 0$, these are plane curves of genus $6$. Figure \ref{fig3} shows the normalized vector field with one selected orbit, and Figure \ref{fig4} shows the deformation of the unit circle in 7 chosen instances of time. Let also $\mathscr{W}(x)=\mathscr{W}(x,1)$. We have
\begin{eqnarray*}
\mathscr{W}(x)=(x-1)(x^4-4x^3-14x^2-4x+1)=\frac{1+i}{2}(x+i)^{5}+\frac{1-i}{2}(x-i)^{5}.
\end{eqnarray*}
The Galois group of the $4$th degree polynomial in the above expression is a cyclic group of order $4$ ($\mathbb{F}^{*}_{5}$ is cyclic of order $4$). More importantly, the polynomial $\mathscr{W}(x)$ has a $5-$fold symmetry as follows ($\mathbb{F}^{+}_{5}$ is cyclic of order $5$). Let $\kappa_{j}=\frac{2\pi j}{5}$, $j=0,1,2,3,4$. Then
\begin{eqnarray}
\mathscr{W}\Big{(}\frac{x\cos\kappa_{j}+\sin\kappa_{j}}{-x\sin\kappa_{j}+\cos\kappa_{j}}\Big{)}(-x\sin\kappa_{j}+\cos\kappa_{j})^5=\mathscr{W}(x).
\label{cyc}
\end{eqnarray}
Indeed, let $y=\frac{x\cos\kappa_{j}+\sin\kappa_{j}}{-x\sin\kappa_{j}+\cos\kappa_{j}}$. Then
\begin{eqnarray*}
\mathscr{W}(y)=\frac{1+i}{2}(y+i)^5+\frac{1-i}{2}(y-i)^{5},\text{ while }
y\pm i=\frac{(x\pm i)e^{\mp i\kappa_{j}}}{-x\sin\kappa_{j}+\cos\kappa_{j}},
\end{eqnarray*}
and (\ref{cyc}) follows.\\

 The function $\varpi(x,y)$ vanishes when
\begin{eqnarray*}
\frac{x}{y}=-1+\sqrt{2}\pm\sqrt{4-2\sqrt{2}},\quad -1-\sqrt{2}\pm\sqrt{4+2\sqrt{2}},
\end{eqnarray*}
and $\varrho$ vanish at inverses of these algebraic numbers. The vector field itself vanishes only at $(x,y)=(0,0)$. On the unit circle, the flow in outward of inward depending on $\langle(\varpi,\varrho),(x,y)\rangle>0$ or $<0$, the critical values being $\langle(\varpi,\varrho),(x,y)\rangle=0$. Or, as calculations show, it is tantamount to $\mathscr{W}(-\frac{x}{y})=0$. Thus, the flow is tangent to the unit circle at five points
\begin{eqnarray*}
\frac{x}{y}=-1,\quad -1+\sqrt{5}\pm\sqrt{5-2\sqrt{5}},\quad -1-\sqrt{5}\pm\sqrt{5+2\sqrt{5}}.
\end{eqnarray*}
On the unit circle, these are the vertices of the regular pentagon, since these are the values $\tan(-\frac{\pi}{4}+\frac{2\pi j}{5})$, $j=0,1,2,3,4$. On the other hand, the flow is orthogonal to the unit circle at the roots of $\mathscr{W}(\frac{x}{y})$, which are 
\begin{eqnarray}
\xi_{j}=\tan\Big{(}\frac{\pi}{4}+\frac{2\pi j}{5}\Big{)},\quad j=0,1,2,3,4,
\label{xi}
\end{eqnarray}
the negatives of the numbers given just above. Let us arrange these values: 
\begin{eqnarray}
\xi_{1}&=&-1.962_{+}<\xi_{4}=-0.509_{+}\nonumber
\\&<&\xi_{2}=0.158_{+}<\xi_{0}=1.<\xi_{3}=6.313_{+}.
\label{arrange}
\end{eqnarray}

We can give the Taylor series for the function $\gamma$ directly, as before, using an analogue of (\ref{expl-taylor}) in a $2-$dimensional case. This yields
\begin{eqnarray}
\frac{\gamma(x,-x)}{x}&=&1-2x+8x^2+8x^3-16x^4-\frac{768}{5}x^5\label{gamma-first}\\
&+&\frac{2944}{5}x^6+\frac{84352}{35}x^7-\frac{357632}{35}x^8
\cdots,\nonumber\\
\frac{\gamma(x,-x)}{\gamma(-x,x)}
&=&-1+4x-8x^2-32x^3+160x^4\label{gamma-swap}\\
&+&\frac{1216}{5}x^5-\frac{13824}{5}x^6-\frac{55808}{35}x^7+\cdots.
\nonumber
\end{eqnarray}\normalsize
To integrate the vector field $\varpi\bl\varrho$, we use the method developed in \cite{alkauskas-ab}. Thus, consider the ODE
\begin{eqnarray*}
f(x)\varrho(x)+f'(x)(x\varrho(x)-\varpi(x))=1,
\end{eqnarray*}
where in our case $\varpi(x)=\varpi(x,1)$, and $\varrho(x)=\varrho(x,1)$. This ODE has the solution $Y(x)$, where
\begin{eqnarray*}
Y(x)=\frac{1}{\mathscr{W}(x)^{1/5}}\int\limits_{1}^{x}\frac{(t^2+1)\d t}{\mathscr{W}(t)^{4/5}}, \quad x\in\mathbb{R}.
\end{eqnarray*}
Let us now define the function $\mathbf{k}(t)$ implicitly from
\begin{eqnarray*}
\mathscr{W}(\mathbf{k}(t))Y^{5}(\mathbf{k}(t))=t^5.
\end{eqnarray*} 
Let $\alpha(t)$ be the inverse of $\mathbf{k}(t)$: $\mathbf{k}(\alpha(t))=t$. We get $\mathscr{W}(x)Y^5(x)=\alpha^{5}(x)$, and so
\begin{eqnarray}
\alpha(x)=\int\limits_{1}^{x}\frac{(t^2+1)\d t}{\mathscr{W}(t)^{4/5}}.
\label{abel-a}
\end{eqnarray}
So, $\alpha$ is an abelian integral, and $\mathbf{k}(x)$ is an abelian function.

\begin{figure}
\includegraphics[width=80mm,height=80mm,angle=-90]{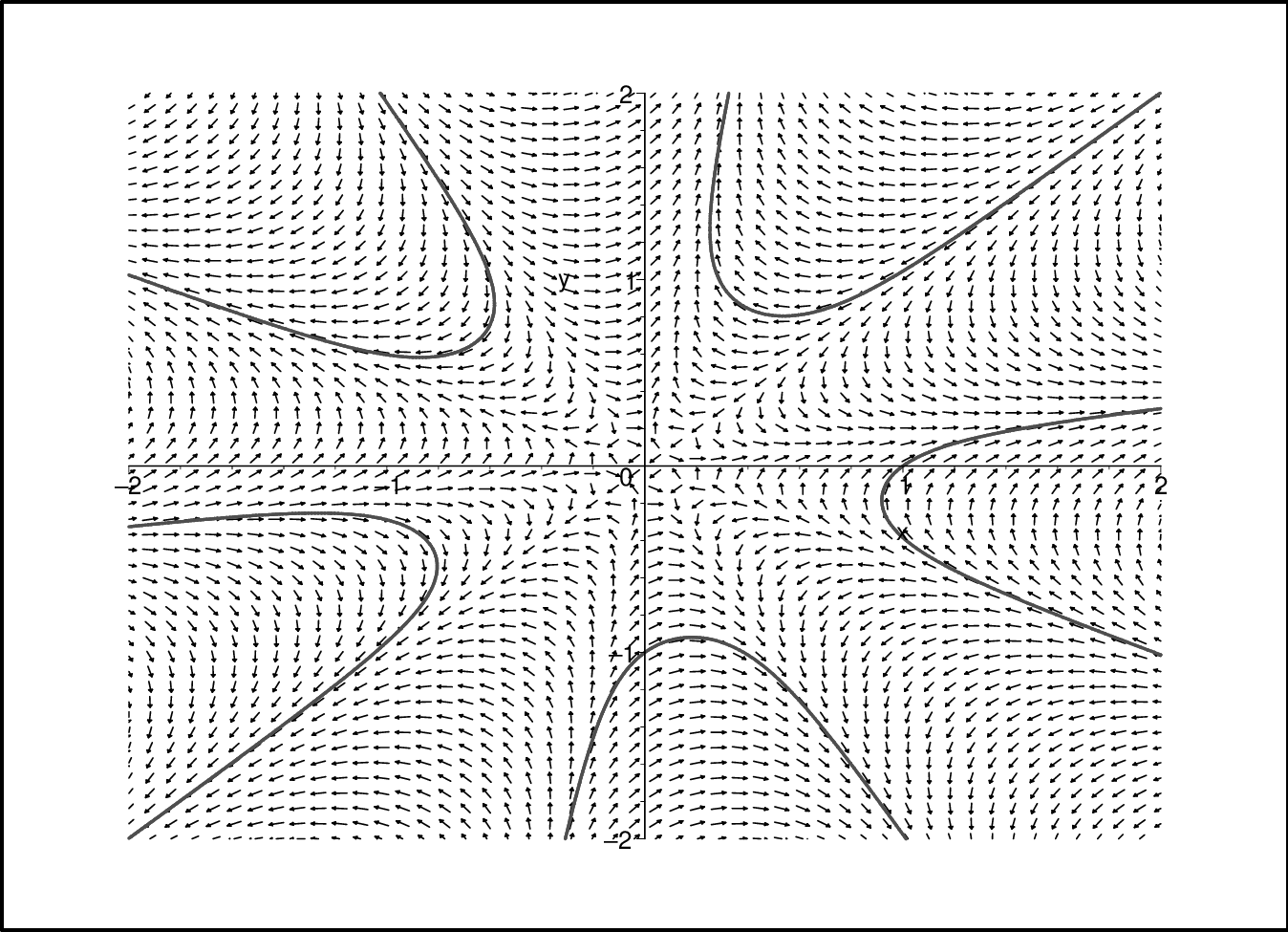}
\caption{The normalized vector field $\varpi\bl\varrho$ (all arrows are adjusted to have the same length), $-2\leq x,y\leq 2$, with a selected  orbit $\mathscr{W}(x,y)=1$.}
\label{fig3}
\end{figure}

\begin{figure}
\includegraphics[width=80mm,height=80mm,angle=-90]{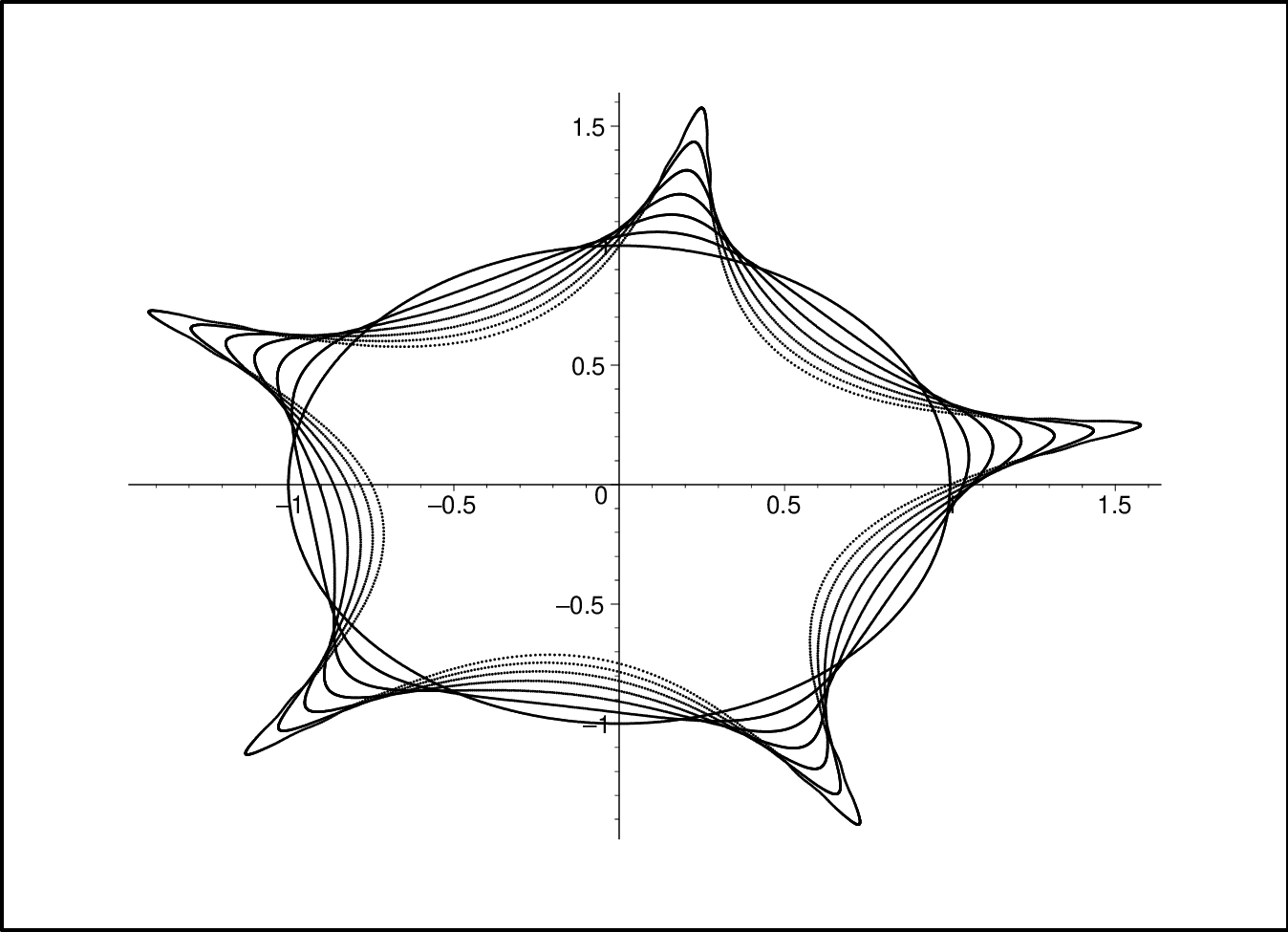}
\caption{Deformation of the unit circle $\mathbf{S}=\{x^{2}+y^{2}=1\}$ under the superflow $\phi_{\mathbb{D}_{5}}$. The seven curves $\phi_{\mathbb{D}_{5}}^{0.044j}(\mathbf{S})$, $j=0,\ldots 6$, are shown.}
\label{fig4}
\end{figure}

\section{Fundamental period and symmetries}
\label{fundamental-period}
Let us introduce the fundamental period (according to the notion introduced in \cite{konzag})
\begin{eqnarray}
\Omega=\frac{1}{5}\int\limits_{-\infty}^{\infty}\frac{(t^2+1)\d t}{\mathscr{W}(t)^{4/5}}=\frac{2}{5}\int\limits_{-1}^{1}\frac{(t^2+1)\d t}{\mathscr{W}(t)^{4/5}}=1.7162590512_{+}.
\label{omega}
\end{eqnarray}
(The equality of integrals will be demonstrated soon). \\

 We will now numerically evaluate $\Omega$. MAPLE cannot handle the integral directly because of the singularities, including at $t=1$. Let $Q(x)=-x^4+4x^3+14x^2+4x-1$. Then $\mathscr{W}(t)=Q(t)(1-t)$. Now, equalities after Proposition \ref{prop6} imply that
\begin{eqnarray*}
\int\limits_{-\xi_{4}}^{1}\frac{(t^2+1)\d t}{\mathscr{W}(t)^{4/5}}=\frac{\Omega}{2}.
\end{eqnarray*}
Assume $x\leq 1$, and let
\begin{eqnarray*}
T(x)=\int\frac{(x^{2}+1)\d x}{(1-x)^{4/5}}=-\frac{5}{11}(1-x)^{11/5}+\frac{5}{3}(1-x)^{6/5}-10(1-x)^{1/5}.
\end{eqnarray*}
Then, using integration by parts,
\begin{eqnarray*}
\int\frac{(x^2+1)\d x}{\mathscr{W}(x)^{4/5}}= T(x)Q(x)^{-4/5}+\frac{4}{5}\int T(x)\frac{Q'(x)}{Q(x)^{9/5}}\d x.
\end{eqnarray*}
The last integral has no singularities for $x\in[-\xi_{4},1]$, and we can calculate the integral to an arbitrary precision with MAPLE.\\

In (\ref{abel-a}), while integrating from the value $t=1$, we can go several times from $-\infty$ to $\infty$, ending up at $t=x$. To rephrase it differently, the function $\alpha$ is only defined modulo $5\Omega$, and so
\begin{eqnarray*}
\mathbf{k}(x+5\Omega)=\mathbf{k}(x).
\end{eqnarray*}
 The crucial property and thus the reason for one of the symmetries of the superflow $\phi_{\mathbb{D}_{5}}$ is the following involutive equality.
\begin{prop}The functions $\alpha$ and $\mathbf{k}$ satisfy
\begin{eqnarray}
 \alpha\Big{(}\frac{1}{x}\Big{)}+\alpha(x)\equiv 0\text{ }\mathrm{( mod}\text{ }5\Omega),\quad \mathbf{k}(x)\cdot\mathbf{k}(-x)=1.
\label{inver-a}
\end{eqnarray}
\label{prop5}
\end{prop}
\begin{proof}
Indeed,
\begin{eqnarray*}
\alpha\Big({}\frac{1}{x}\Big{)}=\int\limits_{1}^{1/x}\frac{(t^2+1)\d t}{\mathscr{W}(t)^{4/5}}=
-\int\limits_{1}^{x}\frac{(t^{-2}+1)t^{-2}\d t}{\mathscr{W}(t^{-1})^{4/5}}=
-\int\limits_{1}^{x}\frac{(t^{2}+1)\d t}{(t^5\mathscr{W}(t^{-1}))^{4/5}}
=-\alpha(x).
\end{eqnarray*}
For example, $2\alpha(-1)\equiv 0\text{ (mod }5\Omega)$, and this gives the second integral in (\ref{omega}). Now,
\begin{eqnarray*}
\mathbf{k}(\alpha(x))\cdot\mathbf{k}\Big{(}\alpha\big{(}\frac{1}{x}\big{)}\Big{)}=x\cdot\frac{1}{x}=1\Rightarrow 
\mathbf{k}(\alpha(x))\cdot\mathbf{k}(-\alpha(x))=1\Rightarrow\mathbf{k}(x)\cdot\mathbf{k}(-x)=1. 
\end{eqnarray*}
\end{proof}
Moreover, it appears that the function $\alpha$ has more symmetries. 
We will employ a deeper structure of the abelian integral. Let $\kappa_{j}=\frac{2\pi j}{5}$, where $j=0,1,2,3,4$ as before. 
\begin{prop}

The abelian integral $\alpha$ possesses the following properties:
\begin{eqnarray}
\alpha(\xi_{3})=\Omega,\quad \alpha(\xi_{2})=-\Omega,\quad \alpha(\xi_{4})=-2\Omega,\quad \alpha(\xi_{1})=-3\Omega,\nonumber\\
\alpha\Big{(}\frac{x\cos\kappa_{j}+\sin\kappa_{j}}{-x\sin\kappa_{j}+\cos\kappa_{j}}\Big{)}-\alpha(x)\equiv 2j\Omega\text{ }\mathrm{( mod}\text{ }5\Omega).\label{diff}
\end{eqnarray}
The abelian function $\mathbf{k}$ has the following properties:
\begin{eqnarray*}
\quad \mathbf{k}(x+2j\Omega)=
\frac{\mathbf{k}(x)\cos\kappa_{j}+\sin\kappa_{j}}{-\mathbf{k}(x)\sin\kappa_{j}+\cos\kappa_{j}}.
\end{eqnarray*}
\label{prop6}
\end{prop}
Note that the Proposition implies (see (\ref{xi}))
\begin{eqnarray*}
\mathbf{k}(0)&=&1,\quad\,\mathbf{k}\Big{(}\frac{\Omega}{2}\Big{)}=-\xi_{1},\\
\mathbf{k}(\Omega)&=&\xi_{3},\quad \mathbf{k}\Big{(}\frac{3\Omega}{2}\Big{)}=-\xi_{3},\\
\mathbf{k}(2\Omega)&=&\xi_{1},\quad  \mathbf{k}\Big{(}\frac{5\Omega}{2}\Big{)}=-1,\\
\mathbf{k}(3\Omega)&=&\xi_{4},\quad \mathbf{k}\Big{(}\frac{7\Omega}{2}\Big{)}=-\xi_{2},\\
\mathbf{k}(4\Omega)&=&\xi_{2},\quad  \mathbf{k}\Big{(}\frac{9\Omega}{2}\Big{)}=-\xi_{4}.
\end{eqnarray*}
\begin{proof}
Let us write $\alpha\Big{(}\frac{x\cos\kappa_{j}+\sin\kappa_{j}}{-x\sin\kappa_{j}+\cos\kappa_{j}}\Big{)}$ by the integral (\ref{abel-a}), and make the change of variables $t=\frac{u\cos\kappa_{j}+\sin\kappa_{j}}{-u\sin\kappa_{j}+\cos\kappa_{j}}$. Using (\ref{cyc}), and also the identities
\begin{eqnarray*}
\d\frac{u\cos\kappa_{j}+\sin\kappa_{j}}{-u\sin\kappa_{j}+\cos\kappa_{j}}=\frac{\d u}{(-u\sin\kappa_{j}+\cos\kappa_{j})^2},\\ (u\cos\kappa_{j}+\sin\kappa_{j})^2+(-u\sin\kappa_{j}+\cos\kappa_{j})^2=u^2+1,
\end{eqnarray*}
we transform the integral in question into the integral (\ref{abel-a}), only with a different lower bound. This shows that the left hand side of (\ref{diff}) is constant (modulo $5\Omega$), and we find it putting $x=1$. At $x=1$, we have:
\begin{eqnarray*}
\alpha\Big{(}\frac{\cos\kappa_{j}+\sin\kappa_{j}}{-\sin\kappa_{j}+\cos\kappa_{j}}\Big{)}=
\alpha\Big{(}\frac{1+\tan\kappa_{j}}{1-\tan\kappa_{j}}\Big{)}=
\alpha\Big{(}\tan\big{(}\frac{\pi}{4}+\kappa_{j}\big{)}\Big{)}=\alpha(\xi_{j}).
\end{eqnarray*}
The ordering (\ref{arrange}) and the definition of the period $\Omega$ gives the correct values. 
Now, the above proof is not literally correct, as M\"{o}bius transformations are transformations of the projective line $P\mathbb{R}^{1}$, so after making a change of variables we sometimes, instead of integrating over $(a,b)$, in fact we end up integrating over $(-\infty,a)\cup(b,\infty)$. The thorough inspection thus shows that the correct statement is the one given in the formulation of the Proposition. Numerical evaluation of integrals with computer confirms this, too. 
\end{proof}
Another period, seemingly not a rational multiple of $\Omega$, is given by 
\begin{eqnarray*}
\Xi=\int\limits_{1}^{\infty}\frac{(t^2+1)\d t}{\mathscr{W}(t)^{4/5}}=\int\limits_{0}^{1}\frac{(t^2+1)\d t}{\mathscr{W}(t)^{4/5}}=2.4439543584_{+}. 
\end{eqnarray*}
This does not play such a crucial r\^{o}le in the arithmetic of $\mathbf{k}$, but nevertheless satisfies
\begin{eqnarray*}
\mathbf{k}(\Xi_{-})=\infty,\quad \mathbf{k}(\Xi_{+})=-\infty,\quad \mathbf{k}(5\Omega-\Xi)=0.
\end{eqnarray*}
\begin{Note} To evaluate $\Xi$ numerically, we use Proposition \ref{prop6}. In particular, it states that $\alpha\big{(}\tan(\frac{2\pi}{5})\big{)}=\alpha(3.077_{+})=2\Omega-\Xi$. So,
\begin{eqnarray*}
\int\limits_{1}^{\tan(\frac{2\pi}{5})}\frac{(t^2+1)\d t}{\mathscr{W}(t)^{4/5}}=2\Omega-\Xi.
\end{eqnarray*}
If we now use the same method as in Note \ref{note1}, only replace $(1-x)$ with $(x-1)$ in the definition of $T(x)$, we get the desired numerical value.
\end{Note}

\section{Analytic formulas}Now, as in (\cite{alkauskas-ab}, Subsection 6.1), we can write the analytic formula for the superflow $\phi_{\mathbb{D}_{5}}=\gamma\bl\xi$ immediately (for a while, we denote the second coordinate by $\xi(x,y)$; after having proved the analytic formulas, we will confirm that $\xi(x,y)=\gamma(y,x)$):
\begin{eqnarray*}
\gamma\Big{(}\frac{\mathbf{k}(a)\v}{a},\frac{\v}{a}\Big{)}\bl
\xi\Big{(}\frac{\mathbf{k}(a)\v}{a},\frac{\v}{a}\Big{)}
=\frac{\mathbf{k}(a-\tilde{\v})\tilde{\v}}{a-\tilde{\v}}Y\big{(}
\mathbf{k}(a-\tilde{\v})\big{)}\bl
\frac{\tilde{\v}}{a-\tilde{\v}}Y\big{(}\mathbf{k}(a-\tilde{\v})\big{)}\\
=\frac{\mathbf{k}(a-\tilde{\v})\tilde{\v}}{\big{(}\mathscr{W}(\mathbf{k}(a-\tilde{\v})\big{)}^{1/5}}\bl
\frac{\tilde{\v}}{\big{(}\mathscr{W}(\mathbf{k}(a-\tilde{\v})\big{)}^{1/5}},
\text{ where }\tilde{\v}=\frac{\v}{Y(\mathbf{k}(a))}.
\end{eqnarray*}
Let
\begin{eqnarray*}
x=\frac{\mathbf{k}(a)\v}{a},\quad y=\frac{\v}{a}.
\end{eqnarray*}
Then, as in \cite{alkauskas-ab}, we have
\begin{eqnarray*}
\mathbf{k}(a)=\frac{x}{y},\quad a=\alpha\Big{(}\frac{x}{y}\Big{)},\quad
\v=y\alpha\Big{(}\frac{x}{y}\Big{)},\quad
\tilde{\v}=\frac{y\alpha\Big{(}\frac{x}{y}\Big{)}}{Y\Big{(}\frac{x}{y}\Big{)}}=
\mathscr{W}(x,y)^{1/5}.
\end{eqnarray*}
So, we finally get
\begin{eqnarray}
\phi_{\mathbb{D}_{5}}(\m{x})=
\frac{\mathbf{k}\big{(}\alpha(\frac{x}{y})-\v\big{)}\v}{\mathscr{W}^{1/5}\Big{(}\mathbf{k}\big{(}\alpha(\frac{x}{y})-\v\big{)}\Big{)}}\bl
\frac{\v}{\mathscr{W}^{1/5}\Big{(}\mathbf{k}\big{(}\alpha(\frac{x}{y})-\v\big{)}\Big{)}},\label{final-fo}
\end{eqnarray}
where $\v=\mathscr{W}(x,y)^{1/5}$. Now, in the second coordinate, let us swap $x$ and $y$. Then $\v$ changes its sign. Using (\ref{inver-a}), and also the reciprocity property
$\mathscr{W}(\frac{1}{t})t^{5}=-\mathscr{W}(t)$, we see that indeed the result is equal to the first coordinate. So, $\xi(x,y)=\gamma(y,x)$.
In particular,
\begin{eqnarray}
\frac{\gamma(x,y)}{\gamma(y,x)}=\mathbf{k}\Big{(}\alpha\Big{(}\frac{x}{y}\Big{)}-\v\Big{)}.
\label{fract}
\end{eqnarray}  
Moreover, using Proposition \ref{prop6} and (\ref{cyc}), we directly verify the pair of functions $\gamma(x,y)\bl\gamma(y,x)$ is invariant under conjugation with the matrix $\tilde{\alpha}$. Thus, Propositions \ref{prop5} and \ref{prop6} are the main reason for the $10-$fold symmetry of our superflow.\\

If we differentiate $\alpha(\mathbf{k}(x))=x$, we obtain $\alpha'(\mathbf{k}(x))\mathbf{k}'(x)=1$, and so the function $\mathbf{k}(x)$ satisfies (see \ref{abel-a})
\begin{eqnarray}
\mathbf{k}'(x)^{5}=\frac{(\mathscr{W}(\mathbf{k}(x))^{4}}{(\mathbf{k}(x)^2+1)^{5}},\quad 
\mathbf{k}(0)=1.
\label{diff-k} 
\end{eqnarray}
Thus, this implies the following.
\begin{thm}
\label{thm-d10}
The superflow $\phi_{\mathbb{D}_{5}}=\gamma(x,y)\bl\gamma(y,x)$ generated by the vector field $\varpi\bl\varrho$, is an abelian flow whose orbits are the curves $\mathscr{W}=\mathrm{const}.$, and the analytic expression of $\phi_{\mathbb{D}_{5}}$ is given by
\begin{eqnarray*}
\phi_{\mathbb{D}_{5}}=\frac{\mathbf{k}(\omega)\v}{\mathbf{k}'(\omega)^{1/4}(\mathbf{k}(\omega)^2+1)^{1/4}}
&\bl&\frac{\v}{\mathbf{k}'(\omega)^{1/4}(\mathbf{k}(\omega)^2+1)^{1/4}},\\
\omega=\alpha\Big{(}\frac{x}{y}\Big{)}-\v,&\quad&\v=\mathscr{W}(x,y)^{1/5}.
\end{eqnarray*}
Here the pair of functions $(\mathbf{k},\mathbf{k}')$ (the abelian function and its derivative) parametrizes 
(locally) the genus $6$ curve $\mathcal{C}$, defined by $(x^2+1)^5y^5=(x-1)^4(x^4-4x^3-14x^2-4x+1)^4$, $\mathbf{k}(0)=1$. The superflow $\phi_{\mathbb{D}_{5}}$ has a dihedral group $\mathbb{D}_{5}$ as the group of its symmetries, and the crucial role in the arithmetic of this superflow is played by the fundamental period (\ref{omega}). 
\end{thm}
If we make a birational change $(x,y)\mapsto(x,\frac{y(x^2+1)}{\mathscr{W}(x)})$, we see that the genus of $\mathcal{C}$ is the same as that of $\mathscr{W}(x,y)=1$. So, it is $6$.\\

Still, this is not the final formula. Indeed, we can write $\mathbf{k}(\omega)$ using addition formulas for abelian functions on the curve $\mathcal{C}$. Similarly as with the superflow $\Lambda$, this leads to the formulas which involve algebraic expressions in $(x,y)$ and the values of abelian functions at $\v$. So, the final conjectural formulas contain only algebraic and abelian functions, not abelian integrals, like $\alpha$. For example, 
\begin{eqnarray*}
\mathbf{k}(\alpha(x))=x,\quad \mathbf{k}'(\alpha(x))=\frac{\mathscr{W}(x)^{4/5}}{x^2+1},
\end{eqnarray*} 
and similarly for other abelian functions. The refinement of Theorem \ref{thm-d10} will be given in the fourth part of this work \cite{alkauskas-super4}. Also, the latter paper will contain analogous results for the icosahedral superflow \cite{alkauskas-super2}. In the icosahedral case the generic orbits are space curves of genus $25$, and with the help of algebraic transformations we can reduce the investigation of the orbits and the superflow itself to investigation of curves of genus $3$. For example, for one specific orbit, this reduces to a pair of functions $(\Delta,\Delta')=(X,Y)$ (the function and its derivative, as in the Weierstrass elliptic function case) which parametrizes the genus $3$ curve \cite{alkauskas-super2,alkauskas-super4}
\begin{eqnarray}
4\Big{(}36Y^2+5(42 X-1)
\big{(}48X^3+12X^2+36X-1\big{)}\Big{)}^2\nonumber\\
=375\big{(}48X^3+12X^2+36 X-1\big{)}^{3}.
\label{curve-ico}
\end{eqnarray} 
For a closely related topics concerning addition formulas for Abelian functions, see \cite{eilbeck}.
\begin{Example} Now, as always, we will verify the obtain formulas with the help of MAPLE. First, we will do it for the formula (\ref{fract}). Let $(x,y)=(x,-x)$. Then $\v=-\sqrt[5]{8}x$. Further, $\mathbf{k}(\alpha(-1))=-1$ gives $\mathbf{k}'(\alpha(-1))=2\sqrt[5]{4}$. Let us define the unknown coefficients $a_{i}$ by 
\begin{eqnarray*}
\mathbf{k}\big{(}\alpha(-1)+\sqrt[5]{8}x\big{)}=-1+\sum\limits_{i=1}^{\infty}a_{i}x^{i}=-1+L.
\end{eqnarray*} 
Then
\begin{eqnarray*}
\mathbf{k}'\big{(}\alpha(-1)+\sqrt[5]{8}x\big{)}=\frac{1}{\sqrt[5]{8}}\sum\limits_{i=1}^{\infty}ia_{i}x^{i-1}.
\end{eqnarray*}
Now,
\begin{eqnarray*}
\mathscr{W}(-1+L)=8-20L+20L^3-10L^4+L^5.
\end{eqnarray*}
Plugging this into (\ref{diff-k}), we obtain
\begin{eqnarray}
\sum\limits_{i=1}^{\infty}ia_{i}x^{i-1}=4\frac{(1-\frac{5}{2}L+\frac{5}{2}L^3-\frac{5}{4}L^4+\frac{1}{8}L^5)^{4/5}}{1-L+\frac{L^2}{2}}.
\label{above}
\end{eqnarray}
On the other hand, we know the values of the coefficients $a_{i}$ in advance, and they are given by (\ref{gamma-swap}). Plugging the known values into the right hand side of (\ref{above}) and expanding as a series of $x$ to an arbitrary power, we readily obtain the left hand side. Further,
\begin{eqnarray*}
\frac{(-1+L)\v}{(\mathscr{W}(-1+L))^{1/5}}=\frac{(1-L)x}{(1-\frac{5}{2}L+\frac{5}{2}L^3-\frac{5}{4}L^4+\frac{1}{8}L^5)^{1/5}}.
\end{eqnarray*}
Plugging the known value of the series $L$ into this formula and dividing by $x$, we obtain exactly the series (\ref{gamma-first}). Thus, the formula (\ref{final-fo}) has been double-verified.
\end{Example}
\section{Higher order dihedral superflows}
Let
\begin{eqnarray*}
\mathscr{W}(x)=\mathscr{W}_{2d+1}(x,1)=P_{2d+1}(x,1)-(-1)^{d}Q_{2d+1}(x,1).
\end{eqnarray*}
 Let us define
\begin{eqnarray*}
\alpha_{d}(x)=\int\limits_{1}^{x}\frac{(t^2+1)^{d-1}\d t}{\mathscr{W}(t)^{\frac{2d}{2d+1}}}.
\end{eqnarray*}
Similarly as in the case $d=2$, this is a fascinating abelian integral having a $(4d+2)-$fold symmetry. The fundamental period is defined by
\begin{eqnarray*}
\Omega_{2d+1}=\frac{1}{2d+1}\int\limits_{-\infty}^{\infty}\frac{(t^2+1)^{d-1}\d t}{\mathscr{W}(t)^{\frac{2d}{2d+1}}}.
\end{eqnarray*} 
All the steps in deriving analytic formulas are similar as in the case $2d+1=5$.\\

In particular, for $d=3$ the vector field $\varpi_{7}\bl\varrho_{7}$ is equal to 
\begin{eqnarray*}
\frac{x^6-6x^5y-15x^4y^2+20x^3y^3+15x^2y^4-6xy^5-y^6}{(x^2+y^2)^2}\bl \\
\frac{-x^6-6x^5y+15x^4y^2+20x^3y^3-15x^2y^4-6xy^5+y^6}{(x^2+y^2)^2}.
\end{eqnarray*}

Figure \ref{figg4} shows the deformation of the unit circle under this superflow.

\begin{figure}
\includegraphics[width=80mm,height=80mm,angle=-90]{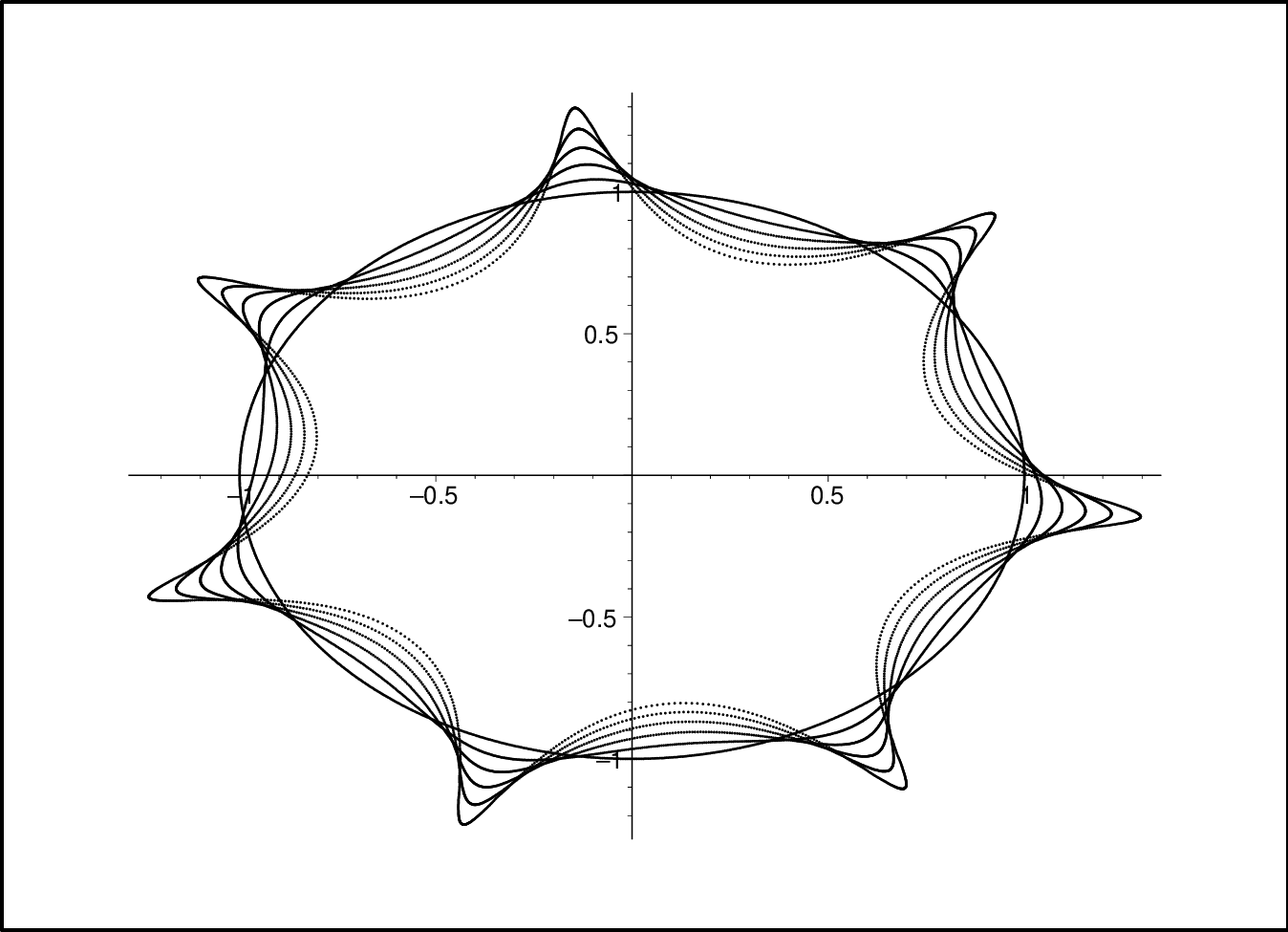}
\caption{Deformation of the unit circle $\mathbf{S}=\{x^{2}+y^{2}=1\}$ under the superflow $\phi_{\mathbb{D}_{7}}$. The six curves $\phi_{\mathbb{D}_{7}}^{0.033j}(\mathbf{S})$, $j=0,\ldots 5$, are shown.}
\label{figg4}
\end{figure}

%% file: sf-chap4.tex
\chapter{The full tetrahedral group $\widehat{\mathbb{T}}$ and the superflow $\phi_{\widehat{\mathbb{T}}}$}
\label{S4}
\begin{Poly} The vector fields of superflows below do not have denominators, so it applies both to projective and polynomial cases without alterations.
\end{Poly}
\section{Higher order symmetric groups $S_{n}$} 
\label{symm-N}
In the previous section it was shown that if we treat the symmetric group $S_{3}$ as a dihedral group $\mathbb{D}_{3}$, and explore the superflow $\Lambda$ given in Section \ref{sec-2d}, then (the linear conjugate of) this flow is a first example in the series of dihedral superflows $\phi_{\mathbb{D}_{2d+1}}$ of dimension $2$.\\

On the other hand, if we treat $S_{3}$ as a symmetric group, then $\Lambda$, as was shown in \cite{alkauskas-un}, is a first example (for $n=2$) in the series of superflows $\phi_{S_{n+1}}$ of dimension $n$ with symmetries given by the symmetric group $S_{n+1}$ of order $(n+1)!$. Namely, we consider the standard $(n+1)-$dimensional permutation representation of $S_{n+1}$, and then investigate its representation on an invariant $n$-dimensional subspace \cite{alkauskas-un,kostrikin}. These computations yield the following.\\

 Let $n\in\mathbb{N}$, $n\geq 2$, and consider the following vector field
\begin{eqnarray}
\mathbf{Q}(\m{x})=Q_{1}(\mathbf{x})\bl Q_{2}(\mathbf{x})\bl \ldots\bl Q_{n}(\mathbf{x}),
\label{Q-vec}
\end{eqnarray}
where
\begin{eqnarray*}
Q_{1}(\m{x})=x_{1}^{2}-\frac{2}{n-1}\cdot x_{1}
\sum\limits_{i=2}^{n}x_{i},
\end{eqnarray*}
and $Q_{i}$ is obtained from the above component by interchanging the r\^{o}les of $x_{1}$ and $x_{i}$. Now, consider the standard permutation representation of the symmetric group $S_{n}$. It is clear that the vector field $\mathbf{Q}$ is invariant under conjugation with elements of $S_{n}$. Moreover, consider the following $n\times n$ matrix of order $2$:
\begin{eqnarray}
\kappa=\begin{pmatrix}
-1 &   &   &   &  \\
-1 & 1 &   &   &  \\
-1 &   & 1 &   &  \\
\,\,\,\vdots & &   & \ddots &  \\
-1 &   &  &  & 1
\end{pmatrix}.
\label{kappa-matrix}
\end{eqnarray} 
Then, as calculations show, (\ref{kappa}) is satisfied. The group $S_{n}$ and the involution $\kappa$ generate the group $S_{n+1}$. The vector field (\ref{Q-vec}) produces a superflow, an $n$-dimensional generalization of $\Lambda$. These superflows are concisely investigated in Section \ref{desimtas}. In the next Section we pass to a special case $n=3$. 
\begin{Note} Let $n\geq 2$. In \cite{alkauskas-super2} the following is shown. Consider the group $S_{n+1}\oplus\mathbb{Z}_{2}$, $\mathbb{Z}_{2}=\{1,-1\}$, and its $(n+1)$-dimensional representation given by
\begin{eqnarray*}
\widehat{\sigma}:S_{n+1}\oplus\mathbb{Z}_{2}\mapsto\mathrm{GL}(n+1,\mathbb{R}),\quad \widehat{\sigma}\big{(}(\tau,\epsilon)\big{)}=\sigma(\tau)\oplus\tilde{\epsilon}.
\end{eqnarray*} 
Here $\tilde{\epsilon}$ is a $1\times 1$ matrix with an entry $\epsilon$, and $\sigma$ is a $n$-dimensional representation of $S_{n+1}$ just described. Then the linear group $\Gamma=\widehat{\sigma}(S_{n+1}\oplus\mathbb{Z}_{2})$ of order $2(n+1)!$ gives rise to a $(n+1)$-dimensional reducible superflow with a vector field
\begin{eqnarray*}
Q_{1}(\m{x})\bl\cdots\bl Q_{n}(\m{x})\bl 0.
\end{eqnarray*}
This is in $(n+1)$ variables $x_{1},x_{2},\ldots,x_{n+1}$, but in fact independent from the last one. In \cite{alkauskas-super2} we conjecture that it is the subgroup of $\mathrm{GL}(n+1,\mathbb{R})$ of the smallest cardinality for which there exists a (projective or polynomial) superflow. For $n=2$ this is exactly (the linear conjugate of) a $3$-prismal superflow.
\end{Note}
\section{Superflow $\phi_{\widehat{\mathbb{T}}}$}
\label{sub4.2}
Let us define 
\begin{eqnarray}
\alpha\mapsto\begin{pmatrix}
1 & 0 & 0\\
0 & -1 & 0\\
0 & 0 & -1 
\end{pmatrix},&\quad&
\beta\mapsto\begin{pmatrix}
-1 & 0 & 0\\
0 & 1 &0\\
0 & 0 & -1 
\end{pmatrix},\nonumber\\
\gamma\mapsto\begin{pmatrix}
0 & 1 & 0\\
1 & 0 & 0\\
0 & 0 & 1 
\end{pmatrix},
&\quad&
\delta\mapsto\begin{pmatrix}
0 & 1 & 0\\
0 & 0 & 1\\
1 & 0 & 0 
\end{pmatrix}.
\label{g-24}
\end{eqnarray}
These are matrices of order $\alpha^{2}=\beta^{2}=\gamma^{2}=\delta^{3}=I$, and together they generate the full tetrahedral group $\widehat{\mathbb{T}}$ of order $24$. In fact, already $\alpha,\gamma$ and $\delta$ generate the whole group, but we will need a matrix $\beta$ later; see Note \ref{note-red2}. Since $\det(\gamma)=-1$, this is  a subgroup of $O(3)$, not $SO(3)$. Calculations with SAGE (or the techniques in \cite{benson, verma}) yield that the ring of invariants is generated by $\{x^2+y^2+z^2,x^4+y^4+z^4, xyz\}$. See \cite{goller} for a rigorous proof. The vector field which is invariant under this group is given by
\begin{eqnarray}
\mathbf{S}(\m{x})=yz\bl xz\bl xy=\varpi\bl\varrho\bl\sigma.
\label{pelican}
\end{eqnarray}
and, up to scalar multiplication, this is the unique vector field without denominators. This vector field is linearly conjugate to the vector field (\ref{Q-vec}) in case $n=3$. \\

 The invariance of the flow $\phi_{\widehat{\mathbb{T}}}=U(x,y,z)\bl U(y,z,x)\bl U(z,x,y)$ under conjugation with all elements of the group $\widehat{\mathbb{T}}$ gives the following. 
\begin{eqnarray}
U(x,y,z)=-U(-x,-y,z)=-U(-x,y,-z)=U(x,z,y).
\label{u-inv}
\end{eqnarray}
(Compare this with (\ref{v-inv})).

\begin{Note}
\label{note-steiner}
Consider the unit sphere $\mathbf{S}^2$, and the map $f:\mathbf{S}^{2}\mapsto\mathbb{R}^{3}$, given by $f(x,y,z)=(yz,xz,xy)$. Since $f(-x,-y,-z)=f(x,y,z)$, this gives a self-intersecting mapping of the real projective plane into a three-dimensional space. If we remove six points, the obtained map is an immersion. The surface is called \emph{Steiner surface} (\cite{conlon}, p. 21), or alternatively \emph{Roman surface}, and it has the tetrahedral symmetry, exactly as the vector field $\mathbf{S}$; see Figure \ref{fig-steiner}. Implicitly, it is given by
\begin{eqnarray*}
x^2y^2+x^2z^2+y^2z^2-xyz=0.
\end{eqnarray*} 
See Section \ref{sec5.2} for an analogous surface in the octahedral case, and \cite{alkauskas-super2} - in the icosahedral. Implicit equation even in the octahedral case is far more complicated.
\end{Note}
\begin{figure}
\includegraphics[scale=0.65]{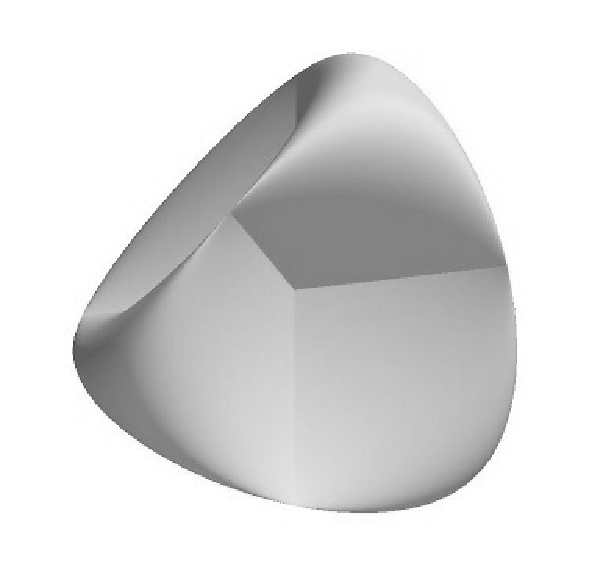}
\caption{The Steiner (Roman) surface}
\label{fig-steiner}
\end{figure}
\begin{Note}
\label{note-red2}
Suppose, a second degree polynomial vector field $A\bl B\bl C$ is invariant under conjugation with $\alpha$ and $\beta$. This gives, respectively,
\begin{eqnarray*}
A(x,-y,-z)=A(x,y,z),\quad -A(-x,y,-z)=A(x,y,z).
\end{eqnarray*}
So, $A=ayz$. Similarly, $B=bxz$, $C=cxy$, and so we have a vector field
\begin{eqnarray*}
ayz\bl bxz\bl cxy,\quad a,b,c\in\mathbb{R}.
\end{eqnarray*}
This is a general expression for the polynomial $2$-homogeneous vector field invariant under group generated by $\alpha$ and $\beta$ (Klein four-group). Note that the corresponding differential system, that is, $x'=A(x,y,z),y'=B(x,y,z),z'=(C(x,y,z)$, was considered in (\cite{hopkins2}, Example 1.3). Further, suppose this vector field is invariant under conjugation with $\delta$. This gives $a=b=c$, and thus we have a superflow.\\

However, three matrices $\alpha$, $\beta$ and $\delta$ generate the group $\mathbb{T}$ of order $12$, orientation preserving symmetries of a tetrahedron. But we know that this vector field is invariant also under conjugation with $\gamma$. Thus,
\begin{cor}
\label{cor1}
For $n=3$, there exists a finite subgroup $\Gamma$ of $\mathrm{GL}(n,\mathbb{R})$, for which there exists a superflow, whose full group of symmetries is equal to $\widetilde{\Gamma}$, a proper extension of $\Gamma$.
\end{cor}
There are no reasons to doubt that this phenomenon of \emph{an extension of the symmetry} does occur also for $n>3$. (See Note \ref{note-pav} where this phenomenon does not occur). We will see in Note \ref{octa-dim} that, strangely enough, the reason behind the fact why this phenomenon occurs in the tetrahedral case is precisely because there are no non-zero $2$-homogeneous vector fields with an octahedral symmetry.\\

The scenario with reducible superflows (over $\mathbb{C}$) is different; see Note \ref{note-red1}. Let $k=0$ in Proposition \ref{prop-red}. The superflow $\phi_{0}(\m{x})=y^2+x\bl y$ is the superflow for the order $6$ cyclic group generated by $\alpha$, which is given by (\ref{alfa}). However, we have the following result \cite{alkauskas-super3}. 
\begin{prop}
\label{interim}
The full group of symmetries of the superflow $\phi_{0}$ is the group
\begin{eqnarray*}
\Gamma=\Bigg{\{}\gamma_{d,b}=\begin{pmatrix}
d^2 & b\\
0 & d
\end{pmatrix}:b\in\mathbb{C},d\in\mathbb{C}^{*}\Bigg{\}}.
\end{eqnarray*}
Matrix $\gamma_{d,b}$ is of finite order only if $d$ is a root of unity, $b$ is arbitrary, except the case $d=1$, $b\neq 0$, when it is of infinite order. All finite subgroups of $\Gamma$ are cyclic. So, formally, $\phi_{0}$ is still a reducible superflow. 
\end{prop}
This dichotomy \emph{finite - infinite group of symmetries} corresponding to precisely \emph{irreducible - reducible superflows} is investigated in \cite{alkauskas-super3}.
\end{Note}

 Therefore, $\mathbf{S}(\m{x})$ gives rise to the $\widehat{\mathbb{T}}$-superflow$_{\pm}$ $\phi_{\widehat{\mathbb{T}}}=U(x,y,z)\bl U(y,z,x)\bl U(z,x,y)$. 
Note that
\begin{eqnarray*}
\mathrm{div}\,\mathbf{S}=0,\quad\mathrm{curl}\,\mathbf{S}=0.
\end{eqnarray*}
So, $\phi_{\widehat{\mathbb{T}}}$ is a gradient flow: $\mathbf{S}=\mathrm{grad}\,(xyz)$.
\begin{Note} 
\label{tetra-beltrami}
Let $\m{S}_{2}=\m{S}$. Since $\mathrm{div}\,\m{S}=\m{0}$, there exists a vector field $\m{R}$ such that $\mathrm{curl}\,\m{R}=\m{S}_{2}$ \cite{ficht3}; this is a special reformulation of the fact that all de Rham cohomology groups of $\mathbb{R}^{3}$ are trivial. Among all possible solutions, we require that, for $\ell=3$, the following hold:
\begin{itemize}
\item[i)]$\mathrm{div}\,\m{R}=0$;
\item[ii)]$\m{R}$ is $\ell$-homogeneous rational vector field;
\item[iii)] $\m{R}$ has at least a tetrahedral symmetry $\mathbb{T}$.
\end{itemize}
A solution exists, and one of them is given by $\m{R}=\m{S}_{3}$, where
\begin{eqnarray*}
\m{S}_{3}=
\frac{1}{4}x(z^2-y^2)\bl\frac{1}{4}y(x^2-z^2)\bl\frac{1}{4}z(y^2-x^2).
\end{eqnarray*}
The full group of symmetries of the latter is $\mathbb{T}\times\{I,-I\}$. The process can be continued for $\ell>3$. Note that, however, sometimes all three requirements are satisfied by a family of vector fields. Thus we construct a sequence $\m{S}_{\ell}$ of vector fields, $\ell\geq 2$, such that $\mathrm{curl}\,\m{S}_{\ell+1}=\m{S}_{\ell}$, and that $\m{S}_{\ell}$ has a tetrahedral symmetry. We then investigate a vector field $\sum\limits_{\ell=2}^{\infty}\m{S}_{\ell}$. For example, in our case one of the solutions is given by $\frac{1}{2}\mathfrak{T}$, where $\mathfrak{T}=(\mathfrak{a},\mathfrak{b},\mathfrak{c})$ is as follows:
\begin{eqnarray*}
\left\{\begin{array}{c@{\qquad}l}
\mathfrak{a}=z\sin y+y\sin z+x\cos y-x\cos z,\\
\mathfrak{b}=x\sin z+z\sin x+y\cos z-y\cos x,\\
\mathfrak{c}=y\sin x+x\sin y+z\cos x-z\cos y.
\end{array}\right.
\end{eqnarray*} 
This resembles to some extent, though structurally differs from, the vector field which gives the so called \emph{Arnold-Beltrami-Childress flow} \cite{fre}. The latter is given by a vector field, or, better, by an autonomous system 
\begin{eqnarray*}
\left\{\begin{array}{c@{\qquad}l}
\dot{x}=A\sin z+C\cos y,\\
\dot{y}=B\sin x+A\cos z,\\
\dot{z}=C\sin y+B\cos x,
\end{array}\right.
\end{eqnarray*}
$A,B,C\in\mathbb{R}$. \\

The Taylor series of $\mathfrak{T}$ indeed starts from $(2yz,2xz,2xy)$. More importantly,
\begin{itemize}
\item[i)]the vector field $\mathfrak{T}$ has a tetrahedral symmetry $\mathbb{T}$;
\item[ii)]and  it satisfies $\mathrm{curl}\,\mathfrak{T}=\mathfrak{T}$.
\end{itemize}
The vector field $B$ is called \emph{a Beltrami vector field}, if $B\times\mathrm{curl}\, B=0$ \cite{blair,etnyre}. So, $\mathfrak{T}$ is a Beltrami field. In physics, vector fields, that equal their own curl, are a special cases of \emph{force-free magnetic fields} \cite{chandr1,chandr2}. Beltrami fields are intricately related to \emph{contact structures}, which are defined as everywhere non-integrable planefields \cite{blair,dahl}.\\

Another solution to the above inverse problem is given by $\mathfrak{T}+a\mathfrak{O}$; see Note \ref{octa-beltrami}. This leads to a notion of \emph{lambent flows}. In dimensions $n\neq 3$ this also has a slightly relaxed analogue, as described by Note \ref{dihedral-beltrami}. The topic of lambent flows is very different from the topic of the current paper and its sequels \cite{alkauskas-super2,alkauskas-super3,alkauskas-super4}. Indeed, in the setting of the current paper, orbits are algebraic curves, and many aspects of the research belong to algebraic geometry, and even to number theory. In the lambent flows case (octahedral case is given by Note \ref{octa-beltrami}, and the icosahedral one in \cite{ico-beltrami}), however, orbits turn out to be chaotic, with the theory of dynamical systems and non-integrable systems manifesting significantly.
\end{Note}
\section{Dimensions of spaces of tetrahedral vector fields}
\label{tetra-sole}
For each even $\ell\in\mathbb{N}$, we can calculate the dimension of the linear space of vector fields of homogeneous degree $\ell$ which have a full tetrahedral symmetry $\widehat{\mathbb{T}}$. For $\ell=2$, this is, of course, equal to $1$, exactly what is needed for the superflow to exist.\\

 Let $\mathcal{P}=\mathcal{P}_{\ell}=\mathbb{R}_{\ell}[x,y,z]$ be a linear space of dimension $\binom{\ell+2}{2}$  of homogeneous polynomials in $x,y,z$. Let $\mathcal{P}_{i}$, $1\leq i\leq 3$, be three copies of $\mathcal{P}$, and $\mathcal{Y}=\mathcal{Y}_{\ell}=\mathcal{P}_{1}\oplus\mathcal{P}_{2}\oplus\mathcal{P}_{3}$. For each $X=(a^{x},a^{y},a^{z})\in\mathcal{Y}$, define $L(X)=L^{\Gamma}_{\ell}(X)=(b^{x},b^{y},b^{z})$, where 
\begin{eqnarray}
b^{x}\bl b^{y}\bl b^{z}=\frac{1}{|\Gamma|}\sum\limits_{\gamma\in\Gamma}\gamma^{-1}\circ(a^{x}\bl a^{y}\bl a^{z})\circ\gamma.
\label{ell-dim}
\end{eqnarray}
In the current case, $\Gamma=\widehat{\mathbb{T}}$, $|\Gamma|=24$.  It is obvious that a vector field $L(X)$ has a tetrahedral symmetry, and if $X$ already has this symmetry, then $L(X)=X$. Let If $\mathcal{U}=\mathcal{U}_{\ell}^{\Gamma}=\mathrm{Im}(L)$, and $\mathcal{V}=\mathcal{V}_{\ell}^{\Gamma}=\mathrm{Ker}(L)$, then we see that $\mathcal{Y}=\mathcal{U}\oplus\mathcal{V}$, and $L$ is a projection onto the first coordinate. So, we are looking for the number $\mathrm{dim}_{\mathbb{R}}\mathcal{U}$.\\

Invariance under matrices $\alpha$, $\beta$, $\gamma$, and $\delta$ (see (\ref{g-24})) of such vector field shows that its first coordinate is of even degree in $x$, of odd degree in each of $y,z$, and symmetric in $y,z$. Thus, it is given by
\begin{eqnarray*}
E_{\ell}(y,z)+x^{2}E_{\ell-2}(y,z)+\cdots+x^{\ell-2}E_{2}(y,z),
\end{eqnarray*} 
where each $E_{2j}(y,z)$ is homogeneous of degree $2j$, of odd degree in each $y,z$, and symmetric in $(y,z)$. So, $E_{2j}(y,z)=yz\widetilde{E}_{j-1}(y^2,z^2)$, where $\widetilde{E}_{j-1}$ is symmetric polynomial of degree $j-1$. The space of symmetric polynomials of degree $s$ in two variables is of dimension $\lfloor\frac{s+2}{2}\rfloor$. Thus, the number we are looking for is
\begin{eqnarray*}
\dim_{\mathbb{R}}\mathcal{U}=\sum\limits_{j=0}^{\frac{\ell-2}{2}}\Big{\lfloor}\frac{j+2}{2}\Big{\rfloor}=\Big{\lfloor}\frac{\ell+2}{4}\Big{\rfloor}\cdot\Big{\lfloor}\frac{\ell+4}{4}\Big{\rfloor}=\Big{\lfloor}\frac{(\ell+2)^2}{16}\Big{\rfloor}.
\end{eqnarray*}
 So, if $\mathcal{U}=\mathcal{U}^{\widehat{\mathbb{T}}}_{\ell}$, the sequence $\{\dim_{\mathbb{R}}\mathcal{U}^{\widehat{\mathbb{T}}}_{\ell}:\ell\in 2\mathbb{N}\}$ starts from
\begin{eqnarray*}
1,2,4,6,9,12,16,20,25,30,36,42,49,56,64,\ldots
\end{eqnarray*}
This is the sequence A002620 in OEIS \cite{oeis}, and is one of the basic sequences. See Section \ref{octa-dim} where we calculate the corresponding dimension for the octahedral group.
\section{Solenoidality}We also would like to know the dimension of the subspace of $\mathcal{U}^{\widehat{\mathbb{T}}}_{\ell}=\mathcal{U}$ of such vector fields that are solenoidal. Let $\ell\geq 4$ be an even integer. Let $X\in\mathcal{U}$. Then
$\mathrm{div} X$ is a symmetric $(\ell-1)$-homogeneous polynomial which is odd in all three variables. It is easy to see that
\begin{eqnarray}
\mathrm{div}:\mathcal{U}\mapsto xyz\cdot\mathbb{R}_{\frac{\ell-4}{2}}[x^2,y^2,z^2]^{S_{3}}
\label{div-surj}
\end{eqnarray} 
is surjective. Here we use the usual notation to denote by $\mathbb{R}_{s}[x,y,z]^{S_{3}}$ a linear space of $s$-homogeneous symmetric polynomials. Indeed, if $a,b,c\in\mathbb{N}$ are odd, then
\begin{eqnarray*}
\sum\limits x^{a}y^{b}z^{c}=\mathrm{div}\Big{(}\frac{x^{a+1}}{a+1}(y^{b}z^{c}+y^{c}z^{b})
\bl\frac{y^{a+1}}{a+1}(x^{b}z^{c}+x^{c}z^{b})\bl\frac{z^{a+1}}{a+1}(x^{b}y^{c}+x^{c}y^{b})\Big{)}.
\end{eqnarray*}
The sum on the right denotes symmetric sum with respect to all $6$ permutations of variables, and the vector field under ``$\mathrm{div}$" has the full tetrahedral symmetry. This proves surjectivity of (\ref{div-surj}).\\

So, we need to know the number of ways an integer $\frac{\ell-4}{2}$ can be partitioned into sum of $3$ non-negative integers (the order is not important), and this number is equal to \cite{andrews}
\begin{eqnarray*}
\Big{\Vert}\frac{(\ell+2)^2}{48}\Big{\Vert}.
\end{eqnarray*}
Here $\Vert\star\Vert$ is the nearest integer function (the tie on $\frac{1}{2}$ cannot occur). Thus, the number (the dimension) we are looking for is equal to
\begin{eqnarray*}
a_{\frac{\ell}{2}}=\Big{\lfloor}\frac{(\ell+2)^2}{16}\Big{\rfloor}
-\Big{\Vert}\frac{(\ell+2)^2}{48}\Big{\Vert}=
\Big{\lfloor}\frac{\ell^2+4\ell+12}{24}\Big{\rfloor},\quad \ell\in2\mathbb{N}.
\end{eqnarray*}
Thus, this sequence $\{a_{s},s\in\mathbb{N}\}$ starts from
\begin{eqnarray*}
1,1,3,4,6,8,11,13,17,20,24,28,33,37,43,\ldots
\end{eqnarray*}
Its number in \cite{oeis} is A156040. This is the dimension of $2s$-homogeneous $3$-dimensional vector fields with full tetrahedral symmetry and which are solenoidal. However, according to \cite{oeis}, this number is equal the amount of compositions (ordered partitions) of an integer $s-1$ into 3 parts (some of which may be zero), where the first summand is at least as great as each of the others.
\section{Explicit formulas}
\label{expl-t}
Next result gives explicit formulas for the tetrahedral superflow.
\begin{thm}
\label{thm-s4}
Assume that $x^2>z^2>y^2$. Then the function $U(x,y,z)$ is given by

\begin{eqnarray}
&&U(x,y,z)=\label{uu}\\
&&\frac{x(x^2-y^2)\sn'\Big{(}\sqrt{x^2-y^2};\sqrt{\frac{x^2-z^2}{x^2-y^2}}\Big{)}+yz\sqrt{x^2-y^2} \sn\Big{(}\sqrt{x^2-y^2};\sqrt{\frac{x^2-z^2}{x^2-y^2}}\Big{)}}{x^2-y^2-x^2\sn^2\Big{(}\sqrt{x^2-y^2};\sqrt{\frac{x^2-z^2}{x^2-y^2}}\Big{)}}.
\nonumber
\end{eqnarray}
\normalsize
Here $\sn$ is the Jacobi elliptic function, and $\sn'$ refers to the derivative with respect to the first variable.\\

Assume now that $x^2>y^2$, $z=x$. Then
\begin{eqnarray}
&&U(x,y,x)=\label{uu-spec}\\
&&\frac{x(x^2-y^2)\cos(\sqrt{x^2-y^2})+xy\sqrt{x^2-y^2}\sin(\sqrt{x^2-y^2})}{x^2\cos^{2}(\sqrt{x^2-y^2})-y^2}.
\nonumber
\end{eqnarray}
\end{thm}
For the basics on Jacobi elliptic functions we refer to \cite{ahiezer,meyer}. 
The formulas in cases other that $x^2>z^2>y^2$ are analogous. It is important to emphasize that in Theorem \ref{thm-s4}, and in general, in all theorems describing explicit solutions of projective flows (like Theorems \ref{thm-d10}, \ref{thm-spec} and \ref{thm4}), the conditions on variables in which the formulas are valid, are always homogeneous. Therefore, these formulas determine not only $\phi(\m{x})=\phi^{1}(\m{x})$, but the whole orbit $\phi^{t}(\m{x})$, $t\in\mathbb{R}$. For example, in Theorem \ref{thm-s4} we find the first coordinate of $\phi^{t}(\m{x})$, which is
\begin{eqnarray*}
\frac{U(xt,yt,zt)}{t}.
\end{eqnarray*} 
\begin{proof}
The corresponding differential system (\ref{sys-in}) in this case is
\begin{eqnarray*}
p'=-qr,\quad q'=-pr,\quad r'=-pq.
\end{eqnarray*}
This implies
\begin{eqnarray*}
pp'=qq'=rr'=-pqr\Rightarrow (p^2)'=(q^2)'=(r^2)'.
\end{eqnarray*}
So, the orbits of the superflow $\phi_{\widehat{\mathbb{T}}}$ are given by the curves
\begin{eqnarray*}
x^2-y^2=\mathrm{const.},\quad x^2-z^2=\mathrm{const.}
\end{eqnarray*}
It is therefore the intersection of two space quadratics, and generically the intersection is an elliptic curve. This shows that the flow is abelian - its orbits are algebraic curves. According to the general setting of Section \ref{sub-1.4}, the superflow $\phi_{\widehat{\mathbb{T}}}$ can be described in terms of functions $p,q,r$ which solve the following system
\begin{eqnarray}
\left\{\begin{array}{l}
p'=-qr,\quad q'=-pr,\quad r'=-pq,\\
p^2-q^2=\xi^2,\quad p^2-r^2=1.
\end{array}
\right.
\label{syys}
\end{eqnarray}
The free variable is tacitly assumed to be $t$. So, $\phi_{\widehat{\mathbb{T}}}$ is an abelian flow  of level $(2,2)$. We now will provide the solution to this differential system explicitly. Singular cases are $\xi=\pm1$, $0$ or $\infty$ (the meaning of the latter is that we have a system $p^2-q^2=1$, $p^2-r^2=0$ instead), where the intersection is of genus $0$. We will give the closed-form formulas for the superflow $\phi_{\widehat{\mathbb{T}}}$ for $\xi^2=1,\infty$ as boundary values for this superflow in case (what we henceforth assume) $\xi^2>1$. The case $0<\xi^2<1$ is investigated completely analogously. \\

Recall that Jacobi elliptic functions $\sn,\cn,\dn$ with modulus $k$, $0<k^2<1$, satisfy \cite{ahiezer}
\begin{eqnarray*}
\sn^2(t)+\cn^2(t)=1,&\quad& k^2\sn^2(t)+\dn^2(t)=1,\\
\sn'(t)=\cn(t)\cdot\dn(t),\quad\cn'(t)&=&-\sn(t)\cdot\dn(t),\quad\dn'(t)=-k^2\sn(t)\cdot\cn(t).
\end{eqnarray*}
So, as the solution to (\ref{syys}), we can choose
\begin{eqnarray*}
p(t)=\sn(\xi t;\xi^{-1}),\quad q(t)=i\xi\dn(\xi t;\xi^{-1}),\quad r(t)=i\cn(\xi t;\xi^{-1}).
\end{eqnarray*}
(We are dealing with flows over $\mathbb{R}$, and soon it will be clear that the final formulas are indeed defined over $\mathbb{R}$). So, the first coordinate of the flow $\phi_{\widehat{\mathbb{T}}}$, that is, $U(x,y,z)$, can be given by
\begin{eqnarray}
U\Big{(}p(t)\v,q(t)\v,r(t)\v\Big{)}=p(t-\v)\v.
\label{u-p}
\end{eqnarray}
Let $x=p(t)\v$, $y=q(t)\v$, $z=r(t)\v$. First, we have
\begin{eqnarray*}
x^2-y^2=(p^2-q^2)\v^2=\xi^2\v^2,\quad x^2-z^2=(p^2-r^2)\v^2=\v^2.
\end{eqnarray*} So,
\begin{eqnarray*}
\v=\sqrt{x^2-z^2},\quad \xi=\sqrt{\frac{x^2-y^2}{x^2-z^2}}\in (1,\infty).
\end{eqnarray*}\\
We now apply the addition formulas for jacobian elliptic functions \cite{ahiezer}. Also, we use the fact that $\sn$ is an odd function in the first argument, while $\cn$ and $\dn$ are both even functions. The formula (\ref{u-p}) gives
\begin{eqnarray*}
U(x,y,z)=\sn(\xi t-\xi\v;\xi^{-1})\v
=\v\frac{\sn(\xi t)\cn(\xi\v)\dn(\xi\v)-\sn(\xi\v)\cn(\xi t)\dn(\xi t)}{1-\xi^{-2}\sn^{2}(\xi t)\sn^{2}(\xi\v)}.
\end{eqnarray*}
(Modulus of all elliptic functions involved is $\xi^{-1}$). In this formula, now let us replace
\begin{eqnarray*}
\sn(\xi t)&=&p(t)=x\v^{-1},\\ 
\cn(\xi t)&=&-ir(t)=-iz\v^{-1},\\
\dn(\xi t)&=&-i\xi^{-1}q(t)=-iy\xi^{-1}\v^{-1}.
\end{eqnarray*}
We obtain
\begin{eqnarray*}
U(x,y,z)=(\xi\v)^2\frac{x\cn(\xi\v)\dn(\xi\v)+yz(\xi\v)^{-1}\sn(\xi\v)}{(\xi\v)^2-x^2\sn^{2}(\xi\v)}.
\end{eqnarray*}
Using again the identity $\sn'=\cn\dn$, we can rewrite this as
\begin{eqnarray*}
U(x,y,z)=\frac{x(\xi\v)^2\sn'(\xi\v)+yz(\xi\v)\sn(\xi\v)}{(\xi\v)^2-x^2\sn^{2}(\xi\v)}.
\end{eqnarray*}

The function $\sn(\xi\v;\xi^{-1})$ is explicitly written as
\begin{eqnarray*}
S(x,y,z)=\sn\Big{(}\sqrt{x^2-y^2};\sqrt{\frac{x^2-z^2}{x^2-y^2}}\Big{)},\text{ where } x^2>z^2>y^2.
\end{eqnarray*}
Since $\xi\v=\sqrt{x^2-y^2}$, this gives the formula (\ref{uu}). The boundary case is obtained by specializing $\xi^{-1}=0$ and using the fact that $\sn(z,0)=\sin(z)$. Equally, if we specialize $\xi=1$, we obtain another formula for $U(x,y,y)$ from the fact $\sn(z,1)=\tanh(z)$ (see \cite{ahiezer}). 
\end{proof}
\begin{Example}
 We will now double-check that the formula (\ref{uu}) is correct. Let $x=3t$, $y=t$, $z=2t$, so that
\begin{eqnarray*}
k^2=\frac{x^2-z^2}{x^2-y^2}=\frac{5}{8},\quad \sqrt{x^2-y^2}=t\sqrt{8}.
\end{eqnarray*}
First, we have the Taylor series \cite{abramowitz,viennot} (the first reference contains coefficients up to $u^{11}$)\small
\begin{eqnarray*}
\sn(u;k)=u-\frac{u^3}{3!}(1+k^2)+\frac{u^5}{5!}
(1+14k^2+k^4)-\frac{u^{7}}{7!}(1+135k^2+135k^4+k^6)+\cdots.
\end{eqnarray*}\normalsize
In particular,
\begin{eqnarray}
\sn\Big{(}u,\sqrt{\frac{5}{8}}\Big{)}&=&u-\frac{13}{48}u^3
+\frac{649}{7680}u^{5}-\frac{70837}{2580480}u^{7}\label{sn-tarp}\\
&+&\frac{13141201}{1486356480}u^{9}-\frac{339204983}{118908518400}u^{11}+\cdots.
\nonumber
\end{eqnarray}
Our formula (\ref{uu}) claims (for the modulus $k^{2}=\frac{5}{8}$) that 
\begin{eqnarray*}
U(3t,t,2t)=\frac{24t\sn'(t\sqrt{8})+2t\sqrt{8}\sn(t\sqrt{8})}
{8-9\sn^2(t\sqrt{8})}.
\end{eqnarray*}
Plugging the series (\ref{sn-tarp}) into the above, we obtain the Taylor series for $U(3t,t,2t)$:
\begin{eqnarray*}
U(3t,t,2t)&=&3t+2t^2+\frac{15}{2}t^3+\frac{41}{3}t^4+\frac{253}{8}t^5+\frac{3349}{60}t^6\\
&+&\frac{5557}{48}t^7+\frac{555509}{2520}t^8+\frac{5934937}{13440}t^9+\cdots.
\end{eqnarray*}
On the other hand, we can obtain directly the Taylor series for $U(x,y,z)$ from (\ref{expl-taylor}), where $\varpi\bl\varrho\bl\sigma=yz\bl xz\bl xy$. This independent calculation yields the same series for $U(3t,t,2t)$ as presented above, so this double-checks the validity of the formula (\ref{uu}). Figure \ref{figure-tet} shows the setup.
\end{Example}
\begin{figure}
\includegraphics[scale=0.52]{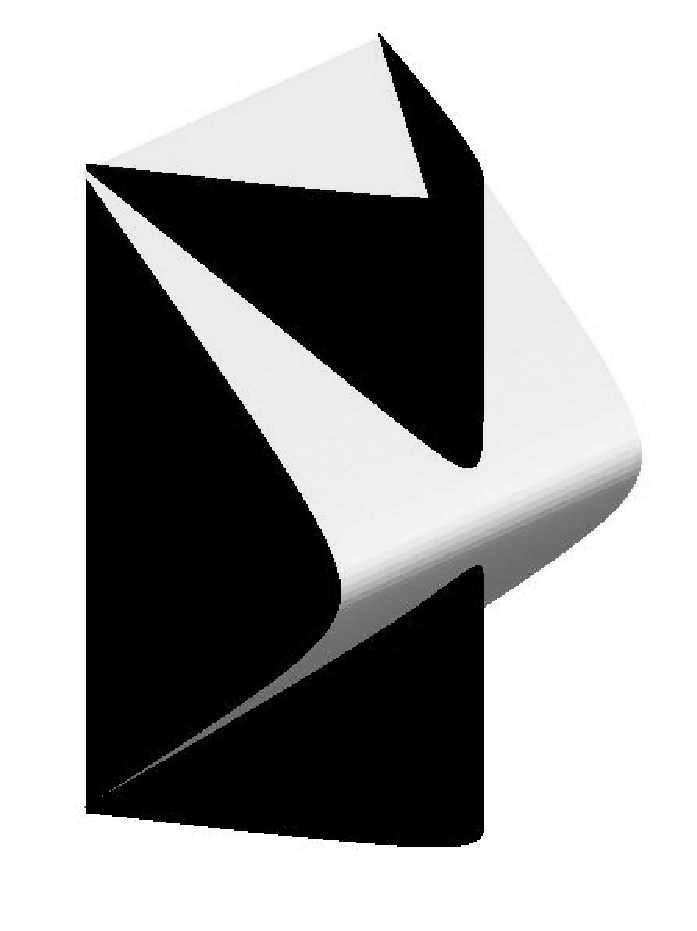}
\caption{The intersection of the surfaces $x^2-y^2=1$ (black) and $x^2-z^2=\frac{5}{8}$ (grey). This intersection consists of $4$ different orbits. Two are shown here, the other part of the picture being the mirror image.}
\label{figure-tet}
\end{figure}
\begin{Example} Let now $(x,y,z)=(5t,4t,5t)$. Then  (\ref{uu-spec}) gives
\begin{eqnarray*}
U(5t,4t,5t)=\frac{45t\cos(3t)+60t\sin(3t)}{25\cos^{2}(3t)-16}.
\end{eqnarray*}
Thus, the Taylor series of the above is 
\begin{eqnarray*}
U(5t,4t,5t)&=&5t+20t^2+\frac{205}{2}t^3+470t^4+\frac{17635}{8}t^5+\frac{20527}{2}t^6\\
&+&\frac{765869}{16}t^7+\frac{6247769}{28}t^8+\frac{932089729}{896}t^9+\cdots.
\end{eqnarray*}
This is exactly what formula (\ref{expl-taylor}) gives, as MAPLE confirms.
\end{Example}
\begin{Note}
\label{PDE-Jacobi}
The function $U(x,y,z)$, as given by (\ref{uu}), satisfies the PDE (\ref{pde}). That is,
\begin{eqnarray}
U_{x}(x-yz)+U_{y}(y-xz)+U_{z}(z-xy)=U.
\label{pde-spec}
\end{eqnarray}
Via the general framework of changing variables (see \cite{ficht3}, Section VI.4, also Note \ref{legendre}), this transforms into the non-linear PDE for the Jacobi $\sn$ function. This is a known result \cite{wolfram}, and thus our approach via the tetrahedral superflow gives a new proof, which we shortly sketch here. \\

Indeed, assume $x^2>z^2>y^2$, and let $v=v(x,y,z)=\sqrt{x^2-y^2}$, $k=k(x,y,z)=\sqrt{\frac{x^2-z^2}{x^2-y^2}}$. Let $\sn^{(a,b)}(v;k)=\frac{\p^{a+b}}{\p v^{a}\p k^{b}}\sn(v;k)$. We have $x^2-v^2=y^2$, $x^2-k^2v^2=z^2$, $yz=\sqrt{(x^2-v^2)(x^2-k^2v^2)}$, and so 
\begin{eqnarray}
U(x,y,z)&=&\frac{xv^2\sn^{(1,0)}(v;k)+yzv\sn(v;k)}{v^2-x^2\sn^{2}(v;k)}.
\label{u-part}
\end{eqnarray}  
Now, we take $v,k,x$ as functions in $x,y,z$. Then
\begin{eqnarray*}
v_{x}&=&\frac{x}{v},\quad v_{y}=-\frac{y}{v}=-\frac{\sqrt{x^2-v^2}}{v},\quad v_{z}=0;\\
k_{x}&=&\frac{x(1-k^2)}{kv^2},\quad k_{y}=\frac{yk}{v^2}=\frac{k\sqrt{x^2-v^2}}{v^2},\quad k_{z}=-\frac{z}{kv^2}=-\frac{\sqrt{x^2-k^2v^2}}{kv^2}.
\end{eqnarray*} 
Now, differentiate (\ref{u-part}) with respect to $x,y,z$, and plug the known values into (\ref{pde-spec}). We obtain the identity of the form
\begin{eqnarray*}
F\Big{(}\sn(v;k),\sn^{(1,0)}(v;k),\sn^{(2,0)}(v;k),\sn^{(0,1)}(v;k)\sn^{(1,1)}(v;k),v,k,x\Big{)}\equiv 0.
\end{eqnarray*} 
This is valid for free variables $v,k,x$. In particular, we can plug any value for $x$. To simplify calculations, we choose $x=v$. This makes $y=0$. 
\end{Note}
\section{First integrals of superflows $\phi_{S_{n+1}}$}
\label{desimtas}
In this section we will present few more properties of vector fields for the superflows $\phi_{S_{n+1}}$ given in Section \ref{symm-N}. If we conjugate a given exact representation of $S_{n+1}$ ending up in the orthogonal group, we would obtain a full symmetry group of order $(n+1)!$ of a $n$-simplex. We will reveal a huge difference with hyper-octahedral superflows described in Chapter \ref{hyper-5}. In particular, since the obtained first integrals are not symmetric with respect to all variables (exactly as in the case of tetrahedral superflow), this implies that a collection of functions describing all coordinates of the superflow, that is, which solves a $n$-dimenional analogue of (\ref{sys-in}), cannot be accommodated within a framework of a single abelian function on a certain curve, and thus the triple reduction mentioned just after Problem \ref{prob-vienas} does not manifest here. This is a wide-reaching consequence of the fact that $\sum\limits_{j=1}^{n}x_{j}^{2}$ is not the first integral.\\

 The special case $n=3$ of the setting described here is linearly conjugate to the tetrahedral superflow treated in this Chapter.  As said in \ref{symm-N}, for the sake of simplicity, we choose not an orthogonal representation, since then all matrices can be made integral, what we indeed have. \\

So, let $n\geq 2$, $n\in\mathbb{N}$, and 
\begin{eqnarray*}
\varpi_{1}(\m{x})=(n-1)x_{1}^{2}-2\cdot x_{1}
\sum\limits_{i=2}^{n}x_{i}=(n+1)x_{1}^{2}-2x_{1}S,\quad
S=\sum_{j=1}^{n}x_{j}.
\end{eqnarray*}
As before, $\varpi_{i}$ are obtained from this with $x_{i}$ instead of $x_{1}$, and
\begin{eqnarray*}
\m{Q}_{n}=\varpi_{1}\bl\varpi_{2}\bl\cdots\bl\varpi_{n}.
\end{eqnarray*}
Thus, $\varpi_{j}(\m{x})=(n-1)Q_{j}(\m{x})$ as given in Section \ref{symm-N}. Yet again we encounter the phenomenon that high symmetry implies solenoidality and also algebraicity of orbits, what we are about to show next.\\

 Indeed, the first integrals of this flow are a bit tricky to find, so we will present a final answer claiming the following.
\begin{prop}Let $n\geq 3$. The form 
\begin{eqnarray*}
\mathscr{Q}=\mathscr{Q}_{1}=x_{1}^{n-2}\prod\limits_{2\leq i<j\leq n}(x_{i}-x_{j})
\end{eqnarray*} 
is the first integral for the vector field $\m{Q}_{n}$. It is of degree $\frac{(n-2)(n+1)}{2}$.
\end{prop}
\begin{proof}We have
\begin{eqnarray*}
\frac{1}{\mathscr{Q}}\cdot\frac{\p\mathscr{Q}}{\p x_{1}}&=&\frac{n-2}{x_{1}},\\
\frac{1}{\mathscr{Q}}\cdot\frac{\p\mathscr{Q}}{\p x_{\ell}}&=&\sum\limits_{j\neq 1,\ell}\frac{1}{x_{\ell}-x_{j}}\text{ for }2\leq \ell\leq n.
\end{eqnarray*}
Careful inspection shows that signs are the correct ones. Thus,
\begin{eqnarray*}
\sum\limits_{\ell=1}^{n}\frac{1}{\mathscr{Q}}\cdot\frac{\p \mathscr{Q}}{\p x_{\ell}}\cdot\varpi_{\ell}=
(n-2)(n+1)x_{1}-2(n-2)S+\sum\limits_{\ell=2}^{n}\sum\limits_{j\neq 1,\ell}
\frac{(n+1)x^{2}_{\ell}-2x_{\ell}S}{x_{\ell}-x_{j}}.
\end{eqnarray*}
For each pair $(i,j)$, $2\leq i<j \leq n$, the denominator $x_{i}-x_{j}$ will appear exactly twice. Considering these two terms separately, we have
\begin{eqnarray*}
\frac{(n+1)x^{2}_{i}-2x_{i}S}{x_{i}-x_{j}}-\frac{(n+1)x^{2}_{j}-2x_{j}S}{x_{i}-x_{j}}=(n+1)(x_{i}+x_{j})-2S.
\end{eqnarray*} 
Thus, the sum in question is equal to
\begin{eqnarray*}
(n-2)(n+1)x_{1}-2(n-2)S+\sum\limits_{2\leq i<j\leq n}\Big{(}(n+1)(x_{i}+x_{j})-2S\Big{)}=0.
\end{eqnarray*}
This shows that $\mathscr{Q}$ is the first integral.
\end{proof} 
Of course, the case $n=2$, described exactly by Theorem \ref{thm2}, is quite exceptional, since the expression for $\mathscr{Q}$ is then void. For $n\geq 3$ other first integrals are given by cyclic permutations. For example, up to the sign, they are
\begin{eqnarray*}
\mathscr{Q}_{\ell}=x_{\ell}^{n-2}\prod\limits_{1\leq i<j\leq n\atop i,j\neq \ell}(x_{i}-x_{j}),\quad 1\leq\ell\leq n.
\end{eqnarray*}
There seem to be $n$ first integrals, though we really have only $(n-1)$ independent, since there exists exactly one linear relation among them. Indeed, let 
\begin{eqnarray*}
P(X)=\prod\limits_{j=1}^{n}(X-x_{j}).
\end{eqnarray*}
Then the equality (\ref{n-2-eq}) coming from Lagrange interpolation (see Section \ref{sec9.2}) states that
\begin{eqnarray*}
\sum\limits_{j=1}^{n}\frac{x_{j}^{n-2}}{P'(x_{j})}=0,
\end{eqnarray*}
showing indeed the linear dependence of all $n$ first integrals. This yet again answers affirmatively to the question posed as Problem \ref{prob-two} in these cases of superflows $\phi_{S_{n+1}}$, $n\geq 2$.  The linear dependence between $n$ first integrals is exactly what happened in Section \ref{expl-t}, where any two of the expressions $x^2-y^2$, $x^2-z^2$, $x^2-z^2$ form a complete collection of independent first integrals, and all three being linearly dependent. \\

The tetrahedral superflow $\phi_{\widehat{\mathbb{T}}}$ and the octahedral superflow $\phi_{\mathbb{O}}$, as we will see in the next Chapter, have a very different arithmetic of orbits. While the first integrals for the octahedral superflow (and also hyper-octahedral ones in Chapter \ref{hyper-5}) are invariants of corresponding groups (see Proposition \ref{inv-fi}), this is not the case with superflow $\phi_{S_{n+1}}$, since $\mathscr{Q}$ is not invariant under permutation of variables (a subgroup $S_{n}<S_{n+1}$), not to mention the whole group $S_{n+1}$, obtained from $S_{n}$ by adding the matrix $\kappa$ given by (\ref{kappa-matrix}). This yet again illustrates the remarkable properties of spherical superflows, described by items 1) through 6) after posing Problem \ref{prob-vienas}.\\

Thus, we have the following differential system
\begin{eqnarray*}
\left\{\begin{array}{l}
\def\arraystretch{2.0}
p_{1}'=(n+1)p_{1}^{2}-2p_{1}\Big{(}\sum\limits_{\ell=1}^{n}p_{\ell}\Big{)},\\
\frac{p_{1}^{n-2}}{H'(p_{1})}\cdot D=\varkappa_{1},\\
\ldots,\\
\frac{p_{n}^{n-2}}{H'(p_{n})}\cdot D=\varkappa_{n}.
\end{array}
\right.
\end{eqnarray*}
Here $\varkappa_{\ell}$ are any real number subject to the condition
\begin{eqnarray*}
\sum\limits_{\ell=1}^{n}\varkappa_{\ell}=0,
\end{eqnarray*}
and
\begin{eqnarray*}
D=\prod\limits_{1\leq i<j\leq n}(p_{i}-p_{j}),\quad H(X)=\prod\limits_{\ell=1}^{n}(X-p_{\ell}).
\end{eqnarray*}
\begin{Example} In case $n=3$, $(p_{1},p_{2},p_{3})=(p,q,r)$, $(\varkappa_{1},\varkappa_{2},\varkappa_{3})=(a,b,c)$, this system looks like
\begin{eqnarray}
\left\{\begin{array}{c}
p'=4p^{2}-2p\Big{(}p+q+r\Big{)},\\
p(q-r)=a,\\
q(r-p)=b,\\
r(p-q)=c,\\
a+b+c=0.
\end{array}
\right.
\label{sys-3}
\end{eqnarray}
Following the same philosophy as in all our text, we need to show that a pair $(p,p')$ parametrizes an algebraic curve whose coefficient depend only on $a,b,c$. Square of the second equation, and difference of the third and fourth, give
\begin{eqnarray*}
\left\{\begin{array}{c}
p^2\big{(}(q+r)^2-4qr\big{)}=a^2,\\
2qr-p(q+r)=b-c.
\end{array}
\right.
\end{eqnarray*}
Expressing $qr$ from the last and plugging into the first, we obtain
\begin{eqnarray*}
p^2\Big{(}(q+r)^2-2p(q+r)+2(c-b)\Big{)}=a^2\Rightarrow (q+r-p)^2=\frac{a^2}{p^2}+2(b-c)+p^2.
\end{eqnarray*} 
Finally, the square of the first equation in (\ref{sys-3}) gives
\begin{eqnarray*}
(p')^2=4p^2(p-q-r)^2.
\end{eqnarray*} 
Thus,
\begin{eqnarray*}
(p')^2=4a^2+8(b-c)p^{2}+4p^4,
\end{eqnarray*}
and we are back to the Jacobi elliptic function setting.
\end{Example}

%% file: sf-chap5.tex
\chapter{The octahedral group $\mathbb{O}$ and the superflow $\phi_{\mathbb{O}}$}
\label{octahedral}
\begin{Poly} Since $x^2+y^2+z^2$ is the first integral of the corresponding system, the results of the next 3 sections can be easily transferred to the polynomial superflow case.
\end{Poly}
\section{Preliminaries}
\label{sub5.1}
Now we pass to a much more complicated superflow than $\phi_{\widehat{\mathbb{T}}}$. Let us define
\begin{eqnarray}
\alpha\mapsto\begin{pmatrix}
0 & 1 & 0\\
1 & 0 & 0\\
0 & 0 & -1 
\end{pmatrix},\quad
\beta\mapsto\begin{pmatrix}
0 & 0 & 1\\
0 & -1 &0\\
1 & 0 & 0 
\end{pmatrix},\quad
\gamma\mapsto\begin{pmatrix}
-1 & 0 & 0\\
0 & 0 & 1\\
0 & 1 & 0 
\end{pmatrix}.
\label{g-24p}
\end{eqnarray}
These are matrices of orders $\alpha^2=\beta^2=\gamma^{2}=I$, and together they generate the group $\mathbb{O}\subset SO(3)$, isomorphic to the octahedral group of order $24$. Direct calculations with SAGE, or using \emph{Molien's series} \cite{benson,smith,verma}, show that the ring of invariants is generated by
\begin{eqnarray}
x^2+y^2+z^2,\quad x^4+y^4+z^4,\quad x^6+y^6+z^6,\nonumber\\
x^5y^3z+y^5z^3x+z^5x^3y-x^5z^3y-y^5x^3z-z^5y^3x.\label{inv-specc}
\end{eqnarray}

The vector field which is invariant under the action of this group was found in \cite{alkauskas-un} and it is given by
\begin{eqnarray}
\m{C}(\m{x})=
\frac{y^3z-yz^3}{x^2+y^2+z^2}\bl\frac{z^3x-zx^3}{x^2+y^2+z^2}\bl\frac{x^3y-xy^3}{x^2+y^2+z^2}=\varpi\bl\varrho\bl\sigma.
\label{cc}
\end{eqnarray}
This vector field gives rise to the $\mathbb{O}$-superflow$_{+}$ $\phi_{\mathbb{O}}=V(x,y,z)\bl V(y,z,x)\bl V(z,x,y)$. We have also $\mathrm{div}\,\m{C}=0$, so the vector field $\m{C}$ is solenoidal.

\begin{Note}
\label{note-pav}
 As in Note \ref{note-red2}, let $\hat{\alpha}$ and $\hat{\beta}$ and $\hat{\delta}$ be now given by (\ref{g-24}) (we use a ``hat" not to confuse these with matrices given by (\ref{g-24p})). As noted, together they generate the group $\mathbb{T}$ of order $12$, which is a subgroup of $\mathbb{O}$ (and also of $\widehat{\mathbb{T}}$). We will find the vector field with a denominator $x^2+y^2+z^2$ invariant under this group. As in Note \ref{note-red2}, if such a vector field $A\bl B\bl C$ is invariant under $\hat{\alpha}$ and $\hat{\beta}$, it is given by
\begin{eqnarray*}
&&\frac{ay^3z+byz^3+cx^2yz}{x^2+y^2+z^2}\bl \frac{dz^3x+ezx^3+fxy^2z}{x^2+y^2+z^2}
\bl\frac{gx^3y+hxy^3+jxyz^2}{x^2+y^2+z^2},\\ 
&&a,b,c,d,e,f,g,h,j\in\mathbb{R}.
\end{eqnarray*} 
Now, invariance under $\hat{\delta}$ gives $a=d=g$, $b=e=h$, $c=f=j$, and so the vector field is
\small
\begin{eqnarray}
\frac{ay^3z+byz^3+cx^2yz}{x^2+y^2+z^2}\bl \frac{az^3x+bzx^3+cxy^2z}{x^2+y^2+z^2}
\bl\frac{ax^3y+bxy^3+cxyz^2}{x^2+y^2+z^2},\label{a-invv}
\end{eqnarray}\normalsize
where $a,b,c\in\mathbb{R}$. This is a general expression for the vector field with a denominator $x^2+y^2+z^2$ which has $\mathbb{T}$ as the group of its symmetries. Now, invariance under $\alpha$ (as given by (\ref{g-24p})) of (\ref{a-invv})  gives a scalar multiple of $\m{C}$, as given by (\ref{cc}), confirming once again that $\m{C}$ is a vector field of the superflow. However, this time the phenomenon $\mathbb{T}<\widehat{\mathbb{T}}$, described in Corollary \ref{cor1}, that is, a group produces a superflow, whose full group of symmetries is a proper extension of an initial one, does not occur.
\end{Note}

The invariance of the flow $\phi_{\mathbb{O}}=V(x,y,z)\bl V(y,z,x)\bl V(z,x,y)$ under conjugation with all elements of the group $\widehat{\mathbb{T}}$ give the following:
\begin{eqnarray}
V(x,y,z)=-V(-x,-y,z)=-V(-x,y,-z)=-V(-x,z,y).
\label{v-inv}
\end{eqnarray}
Note that we get a different set of identities than those in (\ref{u-inv}). Indeed, the fourth Klein group, consisting of matrices $\mathrm{diag}(-1,-1,1)$, $\mathrm{diag}(-1,1,-1)$, $\mathrm{diag}(1,-1,-1)$, $\mathrm{diag}(1,1,1)$ is a subgroup of both $\widehat{\mathbb{T}}$ and $\mathbb{O}$, and this is responsible that the first two identities in (\ref{u-inv}) and (\ref{v-inv}) coincide. However, the last equality is different.\\

In the next few chapters we will explicitly integrate the vector field $\m{C}$ and will find its arithmetic structure. \\
\begin{Example} Consider the vector field whose first coordinate is a $2-$homogeneous function
\begin{eqnarray*}
W(x,y,z)=\frac{(y^{5}z-yz^{5})+a(x^{2}y^{3}z-x^{2}yz^{3})}{(x^2+y^2+z^2)^2+b(x^2y^2+x^2z^2+y^2z^2)},\quad a,b\in\mathbb{R}.
\end{eqnarray*}
and all coordinates are $\widehat{\m{C}}(\m{x})=W(x,y,z)\bl W(y,z,x)\bl W(z,x,y)$. This vector field is also invariant under conjugation with all $\gamma\in\mathbb{O}$. But, first, its denominator has higher degree than that of $\m{C}$. Second, it is not unique but a $2-$parameter family. This example accentuates the exceptional role of the vector field $\m{C}$.
\end{Example}
\begin{Example}
\label{ex7}
 Note that the group $\mathbb{O}$ has a relative invariant of degree $3$; namely, $P(x,y,z)=xyz$. The family of vector fields with a denominator $xyz$, invariant under the group $\mathbb{O}$, is given by \small
\begin{eqnarray*}
\frac{x(z^4-y^4)+ax^3(z^2-y^2)}{xyz}\bl\frac{y(x^4-z^4)+ay^3(x^2-z^2)}{xyz}\bl\frac{z(y^4-x^4)+az^3(y^2-x^2)}{xyz},
\end{eqnarray*}\normalsize
$a\in\mathbb{R}$. This is a $1-$parameter family, and thus is not compatible with the definition of the superflow.
\end{Example}
For any point $(x,y,z)\in\mathbb{R}^{3}\setminus\{0\}$, we say that $(x,y,z)\cdot(x^{2}+y^{2}+z^{2})^{-1/2}$ is \emph{its projection} onto the unit sphere.
Note that on the unit sphere, the vector field $\m{C}(\m{x})$ vanishes at exactly $26$ points; see Section \ref{van-26} for a more transparent explanation. These correspond to $8$ projections of the points $(\pm 1,\pm 1,\pm 1)$ (the signs are independent), $12$ projections of all permutations of $(0,\pm 1,\pm 1)$ (signs are independent again), and $6$ projections of permutations of the points $(0,0,\pm 1)$. If one considers a cube inscribed into the unit sphere, so that its vertices are $(\pm \frac{1}{\sqrt{3}},\pm \frac{1}{\sqrt{3}},\pm \frac{1}{\sqrt{3}})$, the $26$ points correspond to $8$ vertices of this cube itself, $12$ projections of middle points of all edges, and $6$ projections of centres of all faces. 
\section{Orbits} 
\label{orbits-oct}
If $\mathscr{W}(x,y,z)$ is the first integral of the vector field $\varpi\bl\varrho\bl\sigma$, as given by (\ref{cc}), it satisfies (after clearing the common denominator) the identity
\begin{eqnarray*}
\mathscr{W}_{x}(y^3z-yz^3)+\mathscr{W}_{y}(z^3x-zx^3)+\mathscr{W}_{z}(x^3y-xy^3)=0.
\end{eqnarray*}
We see that there exist two independent polynomial first integrals: $\mathscr{W}=x^2+y^2+z^2$, and $\mathscr{W}=x^4+y^4+z^4$. So, the orbits of the superflow $\phi_{\mathbb{O}}$ are algebraic space curves
\begin{eqnarray*}
x^2+y^2+z^2=\mathrm{const}.,\quad x^4+y^4+z^4=\mathrm{const}.
\end{eqnarray*}
Generically, these are space curves of genus $9$ \cite{ha}; see Figure \ref{figure6} in Section \ref{partiv}. Thus, the flow is locked on any sphere with a center at the origin, and if we multiply the sphere by $\lambda$, since the vector field is $2$-homogeneous, the behaviour of the flow on the new-obtained sphere is almost identical, only $\lambda$-times faster.
Stereographic projection of the vector field $\m{C}(\m{x})$ on the unit sphere can be visualised, and is presented, together with an explanation what a stereographic projection of a vector field and a flow is, in Section \ref{sec6.3}.

\section{The surface $\mathscr{O}$}
\label{sec5.2}
Though Steiner surface (see Note \ref{note-steiner}) was only an illustration, since $x^2+y^2+z^2$ is not the first invariant of the tetrahedral superflow, the similar construction for the octahedral case is more important, since it really reveals the geometry of the vector field, so we will carry this in a little more detail. \\

Thus, let us consider the unit sphere $\m{S}^{2}$, and the map $f:\m{S}^{2}\mapsto\mathbb{R}^{3}$, given by
\begin{eqnarray*}
f(x,y,z)=(y^{3}z-yz^{3},z^{3}x-zx^{3},x^{3}y-xy^{3}).
\end{eqnarray*} This yet again gives a self-intersecting map from a real projective space into $\mathbb{R}^{3}$. Figure \ref{alka-octa-fig} shows this surface $\mathscr{O}$. From the point of view of superflows, out of all algebraic surfaces with exactly the orientation-preserving octahedral symmetry, $\mathscr{O}$ is the distinguished one. By \emph{orientation-preserving octahedral symmetry} we mean algebraic surfaces which have a symmetry $\mathbb{O}$, but not $\widehat{\mathbb{O}}$. For example, such are the surfaces\small
\begin{eqnarray*} 
c(x^{2s+4}+y^{2s+4}+z^{2s+4})+x^{2s-1}yz(y^2-z^2)+z^{2s-1}xz(z^2-x^2)+z^{2s-1}xy(x^2-y^2)=1,
\end{eqnarray*}\normalsize
$s\in\mathbb{N}$, $s\geq 3$, $c\in\mathbb{R}_{+}$. For $s=3$, $c=\frac{1}{40}$ it is shown in Figure \ref{octa-symm}. As we saw in Section \ref{orbits-oct}, the vanishing  of
\begin{eqnarray*}
x^{2s-1}yz(y^2-z^2)+z^{2s-1}xz(z^2-x^2)+z^{2s-1}xy(x^2-y^2)
\end{eqnarray*}
for $s=1,2$, gives two first integrals of the octahedral superflow, and for $s=3$ this is exactly the generating invariant (\ref{inv-specc}) of $\mathbb{O}$ which is not an invariant of $\widehat{\mathbb{O}}$.\\
\begin{figure}
\includegraphics[scale=0.62]{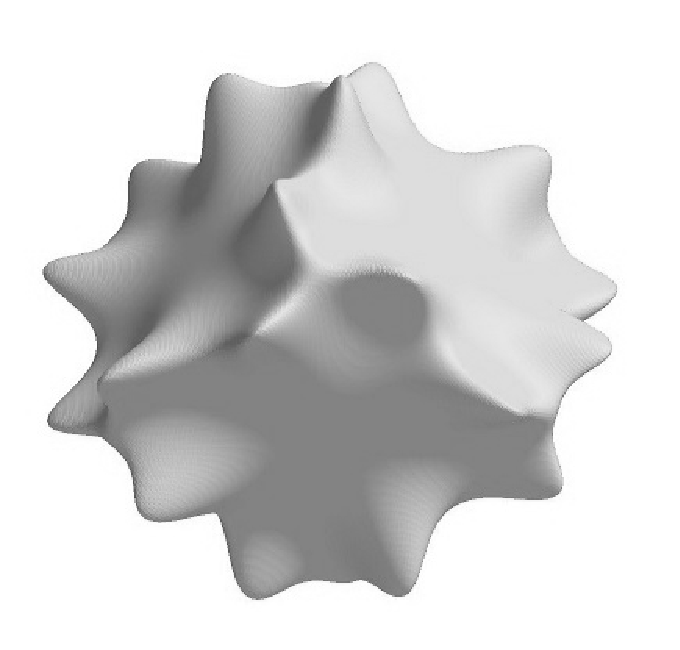}
\caption{The surface with orientation-preserving octahedral symmetry, but not the full octahedral symmetry}
\label{octa-symm}
\end{figure}
\begin{figure}
\includegraphics[scale=0.70]{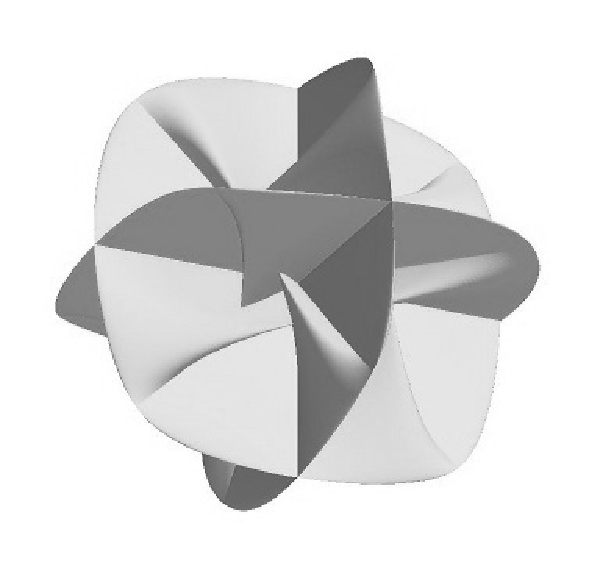}
\caption{The surface $\mathscr{O}$ associated with the octahedral superflow}
\label{alka-octa-fig}
\end{figure}

Of course, $\mathscr{O}$ is a rational surface, since $\m{S}^{2}$ is. To find an implicit algebraic equation for the surface $\mathscr{O}$, let
\begin{eqnarray*}
X=yz(y^2-z^2),\quad Y=xz(z^2-x^2),\quad Z=xy(x^2-y^2). 
\end{eqnarray*}
We want to find polynomial $\mathcal{E}$ in three variables such that
\begin{eqnarray}
\mathcal{E}(X,Y,Z)=(x^2+y^2+z^2-1)\cdot T(x,y,z).
\label{factor-e}
\end{eqnarray}
This is a complicated task, so we first set $\mathfrak{x}=X^2$, $\mathfrak{y}=Y^2$, $\mathfrak{z}=Z^2$, and $a=x^2$, $b=y^2$, $c=z^2$. Thus,
\begin{eqnarray*}
\mathfrak{x}=bc(b-c)^2,\quad\mathfrak{y}=ac(a-c)^2,\quad\mathfrak{z}=ab(a-b)^2.
\end{eqnarray*}  
We want to find a polynomial $\mathcal{F}$ such that
\begin{eqnarray*}
\mathcal{F}(\mathfrak{x},\mathfrak{y},\mathfrak{x})=(a+b+c-1)\cdot L(a,b,c).
\end{eqnarray*}
Let $c=1-a-b$. Then 
\begin{eqnarray*}
\mathfrak{x}&=&b(1-a-b)(2b-1+a)^2,\\
\mathfrak{y}&=&a(1-a-b)(2a-1+b)^2,\\
\mathfrak{z}&=&ab(a-b)^2.
\end{eqnarray*}
The polynomial in three variables which gives an algebraic dependence between three bivariate polynomials $\mathfrak{x},\mathfrak{y},\mathfrak{z}$ is the needed polynomial $\mathcal{F}$. This is a standard task of finding a \emph{relation ideal} in the ring of polynomials $\mathbb{Q}[\mathfrak{x},\mathfrak{y},\mathfrak{z}]$ \cite{cox}. We make computations with MAGMA:
\begin{verbatim}
Q := RationalField(); 
> R<a,b> := PolynomialRing(Q,2); 
> f1 := b*(1-a-b)*(2*b-1+a)^2; 
> f2 := a*(1-a-b)*(2*a-1+b)^2; 
> f3 := a*b*(a-b)^2; 
> L := [f1,f2,f3]; 
> S<x,y,z> := PolynomialRing(Q,3); 
> RelationIdeal(L,S);
\end{verbatim}

 It immediately gives a very complicated single generator $\mathcal{F}(\mathfrak{x},\mathfrak{y},\mathfrak{z})$, which is an irreducible polynomial. However, $\mathcal{F}(X^2,Y^2,Z^2)$ is reducible, splits into two factors, and one of them is exactly the polynomial we are looking for. It is given by
\begin{eqnarray*}
\mathscr{E}(X,Y,Z)&=&\sum\limits_{\textrm{cyclic}}X^{12}(Y^2+Z^2)^2\\
&+&\sum\limits_{\textrm{cyclic}}9X^{11}YZ(Y^2-Z^2)\\
&+&\sum\limits_{\textrm{cyclic}}X^{10}(44Z^2Y^4+44Z^4Y^2+12Z^2Y^2-4Y^6-4Z^6)\\
&-&\sum\limits_{\textrm{cyclic}}X^9YZ(Y^2-Z^2)(18Y^2+18Z^2+1)\\
&+&\sum\limits_{\textrm{cyclic}}6Y^{8}Z^{8}\\
&+&\sum\limits_{\textrm{cyclic}}2X^{8}Y^2Z^2(340Y^2Z^2-23Y^4-23Z^4-18Y^2-18Z^2)\\
&+&\sum\limits_{\textrm{cyclic}}3X^7YZ(Y^2-Z^2)(Y^2+Z^2-90Y^2Z^2)\\
&+&\sum\limits_{\textrm{cyclic}}48X^2Y^6Z^6+1050X^4Z^6Y^6+12X^6Y^4Z^4.
\end{eqnarray*}  
MAPLE indeed confirms (\ref{factor-e}). Thus, $\mathcal{E}(X,Y,Z)=0$ is the implicit equation for the surface given in Figure \ref{alka-octa-fig}. The smallest sphere to contain this surface is of radius (see Chapter \ref{funda-period})
\begin{eqnarray*}
\frac{\sqrt{827+73\sqrt{73}}}{96\sqrt{2}}.
\end{eqnarray*}

\begin{Note}
\label{octa-beltrami}
In the setting of Note \ref{tetra-beltrami}, consider the vector field which is the numerator of $\m{C}$ as given by (\ref{cc}). Namely, let
\begin{eqnarray}
\widetilde{\m{C}}=\m{S}_{4}=y^3z-yz^3\bl z^3x-zx^3\bl x^3y-xy^3.
\label{wide}
\end{eqnarray}
If we put $\m{S}_{3}=\mathrm{curl}\,\m{S}_{4}$, then $\mathrm{curl}\,\m{S}_{3}=\m{0}$. We may try to solve analogous inverse problem, but this time we require octahedral symmetry $\mathbb{O}$.\\

And indeed, one of the solutions is given by $6\mathfrak{O}$ , where $\mathfrak{O}=(\mathfrak{a},\mathfrak{b},\mathfrak{c})$ is given by \cite{alkauskas-beltrami}
\begin{eqnarray*}
\left\{\begin{array}{c@{\qquad}l}
\mathfrak{a}=y\sin z-z\sin y+x\cos y-2\sin x+x\cos z,\\
\mathfrak{b}=z\sin x-x\sin z+y\cos z-2\sin y+y\cos x,\\
\mathfrak{c}=x\sin y-y\sin x+z\cos x-2\sin z+z\cos y.
\end{array}\right.
\end{eqnarray*}
This produces octahedral lambent flow.
\end{Note}

\begin{figure}[htb]
\includegraphics[width=85mm,height=85mm,angle=-90]{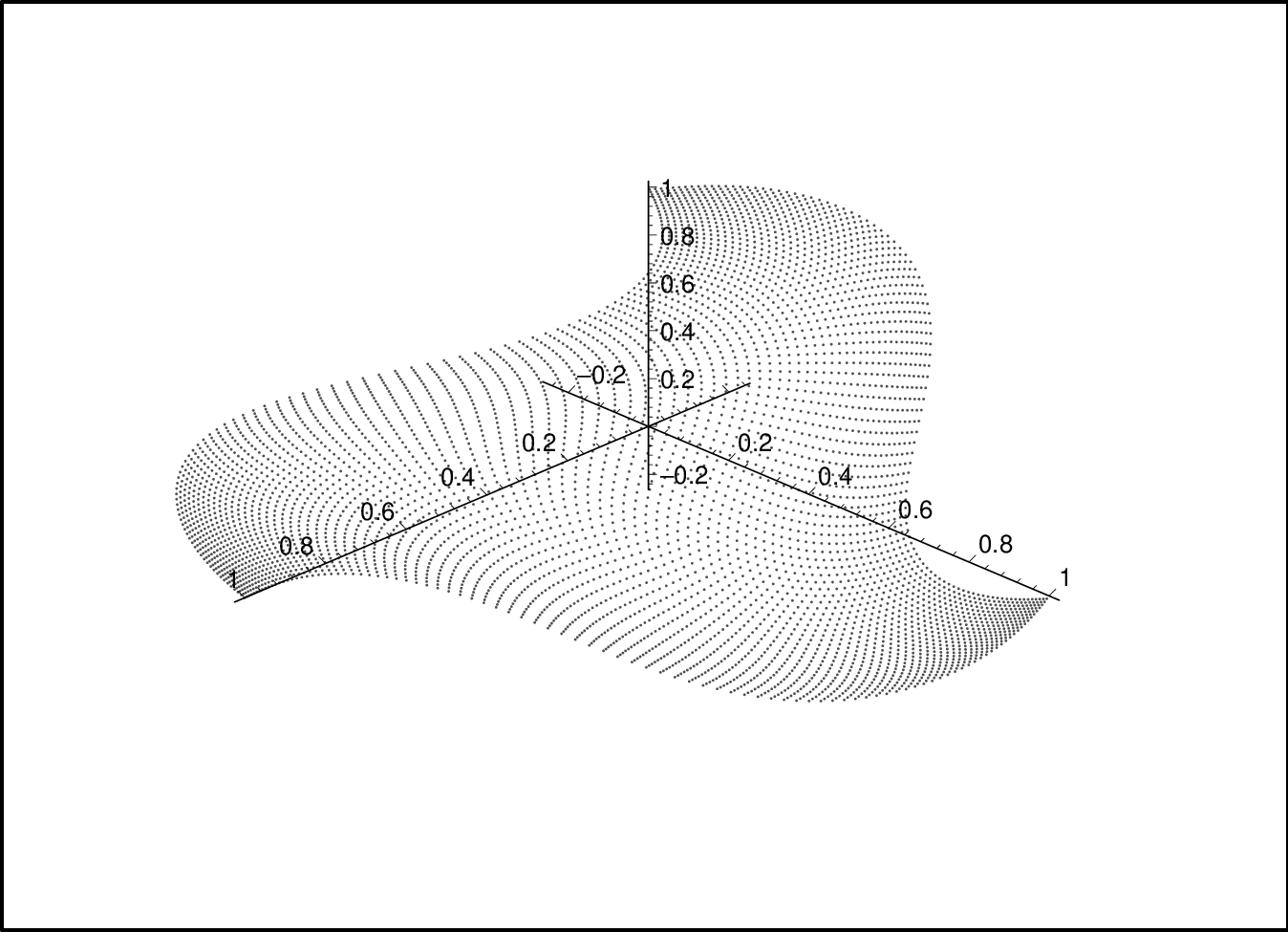}
\caption{The image of the positive octant of the unit sphere $A:=\{x^2+y^2+z^2=1:x,y,z\geq 0\}$ under the superflow $\phi_{\mathbb{O}}$ for $t=1.0$. So, this is the set $\phi_{\mathbb{O}}(A)$.}
\label{figure3}
\end{figure}

\section{The differential system}
A we saw from Section \ref{orbits-oct}, $\phi_{\mathbb{O}}$ is a flow of level $(2,4)$. The corresponding differential system therefore can be taken as
\begin{eqnarray}
\left\{\begin{array}{l}
p'=r^3q-rq^3,\quad q'=p^3r-pr^3,\quad r'=q^3p-qp^3,\\
p^2+q^2+r^2=1,\quad p^4+q^4+r^4=\xi.
\end{array}
\right.
\label{sys2}
\end{eqnarray}
Note that though the vector field $\m{C}$ has a denominator, the first line has no denominators since $p^2+q^2+r^2=1$. Also, if $p,q,r$ is a solution to (\ref{sys2}), so is, for example, $(p,-q,-r)$, in agreement with the fact that $\mathrm{diag}(1,-1,-1)$ (element of a Klein four-group mentioned just after (\ref{v-inv})) is an element of $\mathbb{O}$.\\

 We will now prove the following result.

\begin{prop}Let $P(t)=p^2(t)$, $q^2(t)$ or $r^2(t)$. The pair of functions $(P(t),P'(t))=(X,Y)$ parametrizes the following curve 
\begin{eqnarray*}
Y^2=2X(1-\xi-2X+2X^2)(2\xi-1+2X-3X^2)=f_{\xi}(X).
\end{eqnarray*}
It is of genus $0$ if $\xi=\frac{1}{2}$ or $1$ (quadratic), of genus $1$ if $\xi=\frac{1}{3}$ (elliptic), and of genus $2$ in other cases (hyper-elliptic).
\label{prop8}
\end{prop}
Note that $\xi=0$ is a regular point. 
\begin{proof}
From (\ref{sys2}), we have:
\begin{eqnarray*}
q^2+r^2=1-p^2,\quad q^4+r^4=\xi-p^4\Longrightarrow q^2r^2=\frac{1-\xi}{2}-p^2+p^4.
\end{eqnarray*}
This implies 
\begin{eqnarray*}
(q^2-r^2)^2=q^4+r^4-2q^2r^2=2\xi-1+2p^2-3p^4.
\end{eqnarray*} 
Now, according to the first equation of the system (\ref{sys2}),
\begin{eqnarray*}
(p')^{2}=q^2r^2(r^2-q^2)^2.
\end{eqnarray*}
Therefore,
\begin{eqnarray*}
(p')^2=\Big{(}\frac{1-\xi}{2}-p^2+p^4\Big{)}\Big{(}2\xi-1+2p^2-3p^4\Big{)}.
\end{eqnarray*}
Multiplying this by $4p^2$, we obtain
\begin{eqnarray*}
(2pp')^2=2p^2\Big{(}1-\xi-2p^2+2p^4\Big{)}\Big{(}2\xi-1+2p^2-3p^4\Big{)}.
\end{eqnarray*}
This gives the first statement of the proposition for $P=p^2(t)$. Note that nothing changes if we interchange the r\^{o}les of $p(t)$, $q(t)$ and $r(t)$, so the same conclusion follows if $P=q^{2}(t)$ or $r^{2}(t)$. The discriminant of this $5$th degree polynomial is equal to $8192(2\xi-1)^3(\xi-1)^6(3\xi-1)$. So, except for cases $\xi=\frac{1}{2}$, $\frac{1}{3}$ and $\xi=1$, all five roots of $f_{\xi}(X)$ are distinct. If $\xi=\frac{1}{3}$, the curve becomes
\begin{eqnarray*}
\Big{(}\frac{Y}{3X-1}\Big{)}^2=-\frac{4}{9}X(1-3X+3X^2),
\end{eqnarray*}
and thus it is an elliptic curve. If $\xi=\frac{1}{2}$, the curve is
\begin{eqnarray}
\Big{(}\frac{Y}{X(2X-1)}\Big{)}^2=2-3X,
\label{antroji}
\end{eqnarray}
and so it is birationally equivalent to a parabola, a quadratic of genus $0$. If $\xi=1$, the curve is 
\begin{eqnarray*}
\Big{(}\frac{Y}{X(X-1)}\Big{)}^{2}=-4-12X,
\end{eqnarray*}
and so is also birationally equivalent to a quadratic.
\end{proof}

\section{Elementary geometry of exceptional cases}
\label{elementary}
We will now see that if the superflow $\phi_{\mathbb{O}}$ and its orbits are considered over $\mathbb{R}$, there is a geometric explanation for three exceptional cases $\xi=\frac{1}{3},\frac{1}{2}$, and $1$. Indeed, consider two surfaces $\m{S}^{2}=\{x^2+y^2+z^2=1\}$ and $\m{T}_{\xi}=\{x^4+y^4+z^4=\xi\}$. Since the function $t\mapsto t^2$ is increasing for $t>0$, to understand the geometry of intersection of $\m{S}^{2}$ and $\m{T}$, we can consider the surfaces $\widehat{\m{S}}=\{|x|+|y|+|z|=1\}$ (an octahedron) and $\widehat{\m{T}}=\{x^2+y^2+z^2=\xi\}$ (a sphere). We therefore have the following seven cases:
\begin{itemize}
\item[i)]$0<\xi<\frac{1}{3}$. The surfaces $\widehat{\m{S}}$ and $\widehat{\m{T}}$ do not intersect.
\item[ii)]$\xi=\frac{1}{3}$. The surfaces touch at eight points $(x,y,z)=(\pm\frac{1}{3},\pm\frac{1}{3},\pm\frac{1}{3})$ (signs are independent).
\item[iii)]$\frac{1}{3}<\xi<\frac{1}{2}$. The surfaces intersect at $8$ disjoint curves, each homeomorphic to a circle. Each of $8$ pieces represent a different orbit (over $\mathbb{R}$) of this superflow. 
\item[iv)]$\xi=\frac{1}{2}$. The surfaces intersect at $8$ circles, which are joint in the following fashion. Given a regular octahedron $|x|+|y|+|z|=1$. In each of its $8$ faces a circle is inscribed. Each circle thus touches $3$ other circles inscribed into adjacent faces of the octahedron. There are $12$ touching points: $(\pm \frac{1}{2},\pm\frac{1}{2},0)$, where the signs are independent, and all permutations. If we temporarily return to $\m{S}^{2}$ and $\m{T}$,  these points correspond to fixed points of the superflow, where the vector field vanishes. Thus, the set $\m{S}^{2}\cap\m{T}_{\frac{1}{2}}$ splits into $12$ points (each a full orbit on its own), and $24$ disjoint orbits.
\item[v)]$\frac{1}{2}<\xi<1$. The surfaces intersect at $6$ disjoint curves, each homeomorphic to a circle, and each going round one vertex of the octahedron. Each of $6$ pieces is a separate orbit. Figure \ref{figure6} in Section \ref{partiv} illustrates two surfaces $\m{S}^{2}$ and $\m{T}_{\frac{5}{9}}$ and their intersection; this is also a very important case and will be dealt with in Chapter \ref{partiv} (\emph{a posteriori}, a Weierstrass elliptic function involved in this case has a square period lattice). 
\item[vi)] $\xi=1$. Two surfaces intersect at $6$ points $(\pm 1,0,0)$, and all cyclic permutations.
\item[vii)] $\xi>1$. The surfaces $\widehat{\m{S}}$ and $\widehat{\m{T}}$ do not intersect.
\end{itemize}
Now, note that the intersection of $\m{S}^{2}$ and $\m{T}_{\frac{1}{2}}$ is in fact a union of $4$ large circles on the unit sphere. This is the consequence of an identity
\begin{eqnarray}
&&2(x^4+y^4+z^4)-(x^2+y^2+z^2)^2\nonumber\\
&&=(x+y+z)(x+y-z)(x-y+z)(x-y-z).
\label{iden}
\end{eqnarray}
\begin{Note}
Let, as already defined in the end of Section \ref{dih}, $\phi=\frac{1+\sqrt{5}}{2}$. In \cite{alkauskas-super2} we investigate the icosahedral superflow. The orbits of the latter are singular if $(\phi^2 x^2-y^2)(\phi^2 y^2-z^2)(\phi^2 z^2-x^2)=0$. And so, completely analogously to the $\mathbb{O}-$superflow case, we have a collection of large circles on the unit sphere, $6$ in total.
\end{Note}
As is clear from this Section, the geometry of orbits on the unit sphere better corresponds not to the octahedron, but rather an \emph{cuboctahedron}, see Figure \ref{cub-oct}.

\begin{figure}
\includegraphics[scale=0.22]{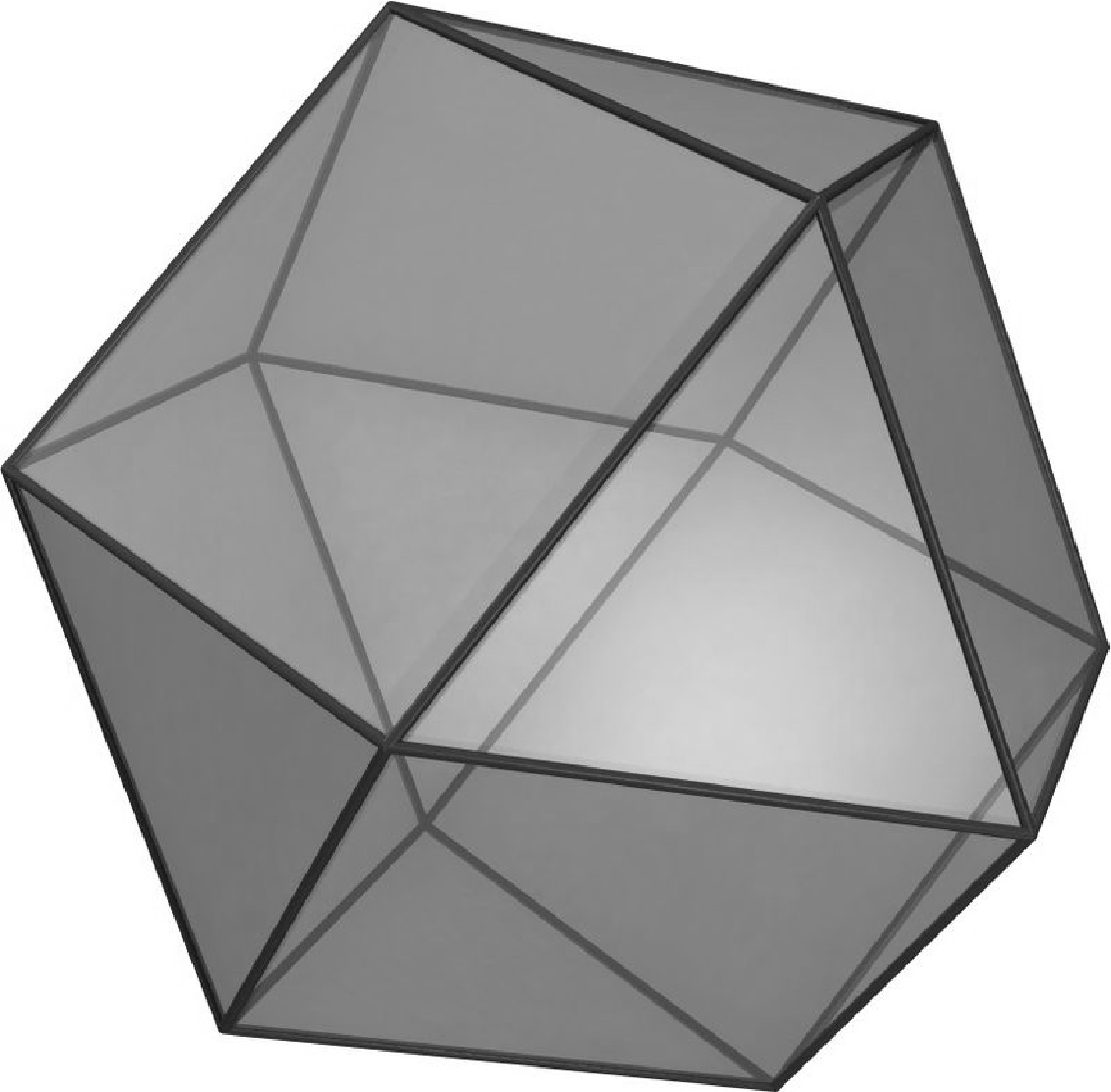}
\caption{The cuboctahedron.}
\label{cub-oct}
\end{figure}

%% file: sf-chap6.tex
\chapter{Octahedral vector fields}
\label{octa-vec}

\section{Dimensions of spaces of octahedral vector fields}
\label{octa-dim}
Let $\ell\in 2\mathbb{N}$. Continuing in the spirit of Section \ref{tetra-sole}, let $\mathcal{U}^{\mathbb{O}}_{\ell}$ be a linear space of $\ell$-homogeneous polynomial vector fields which have an octahedral symmetry. The first coordinate of such a vector field is equal to
\begin{eqnarray*}
E_{\ell}(y,z)+x^{2}E_{\ell-2}(y,z)+\cdots+x^{\ell-2}E_{2}(y,z),
\end{eqnarray*}
where this time each $E_{2j}$ is homogeneous of degree $2j$, of odd degree in each $y,z$ and antisymmetric in $y,z$ (and so $E_{2}=0$). Thus, $E_{2j}(y,z)=yz(y^2-z^2)\widetilde{E}_{j-2}(y^2,z^2)$, where $\widetilde{E}_{j-2}$ is symmetric of degree $j-2$. So, we obtain almost the same sequence of dimensions as in Section \ref{tetra-sole}, only a shifted one:
\begin{eqnarray*}
\dim_{\mathbb{R}}\mathcal{U}^{\mathbb{O}}_{\ell}=\Big{\lfloor}\frac{\ell^2}{16}\Big{\rfloor}.
\end{eqnarray*}  
\section{Two groups}
Next, the only place in the whole study, we deal with both groups $\widehat{\mathbb{T}}$ and $\mathbb{O}$ simultaneously. We use notations of Section \ref{tetra-sole}.\\ 

We know that $\widehat{\mathbb{T}}$ and $\mathbb{O}$ together generate $\widehat{\mathbb{O}}$, and the latter contains $-I$. Next, $\widehat{\mathbb{T}}\cap\mathbb{O}=\mathbb{T}$. The first statement gives
\begin{eqnarray*}
\mathcal{U}^{\widehat{\mathbb{T}}}_{\ell}\cap
\mathcal{U}^{\mathbb{O}}_{\ell}=\emptyset.
\end{eqnarray*}
Using both facts we get that
\begin{eqnarray}
\mathcal{Y}_{\ell}=\mathcal{U}^{\widehat{\mathbb{T}}}_{\ell}
\oplus\mathcal{U}^{\mathbb{O}}_{\ell}\oplus\mathrm{Ker}L_{\mathbb{T}}.
\label{decomp}
\end{eqnarray}
Here the last summand uses the map $L$ given by (\ref{ell-dim}) for the tetrahedral group. Indeed, all we are left to show is that
\begin{eqnarray}
\mathrm{Ker}L_{\mathbb{T}}=\mathrm{Ker}L_{\widehat{\mathbb{T}}}\cap
\mathrm{Ker}L_{\mathbb{O}}.
\label{decomp2}
\end{eqnarray}
Now, the left side is obviously a subspace of the right side. In the other direction, let $a$ be a vector field, $L_{\mathbb{O}}(a)=0$ and $L_{\widehat{\mathbb{T}}}(a)=0$, and put $L_{\mathbb{T}}(a)=c$. 
Note that $\widehat{\mathbb{T}}=\mathbb{T}\cup\gamma\mathbb{T}$, and $\mathbb{O}=\mathbb{T}\cup\alpha\mathbb{T}$, where $\gamma$ and $\alpha$  are given, respectively, by (\ref{g-24}) and (\ref{g-24p}). 
Then this implies
\begin{eqnarray*}
c+\gamma^{-1}\circ c\circ\gamma=0,\quad  c+\alpha^{-1}\circ c\circ\alpha=0.
\end{eqnarray*}
These two imply $c=(\alpha\gamma)^{-1}\circ c\circ (\alpha\gamma)$. But the matrix $\alpha\gamma=\mathrm{diag}(1,1,-1)$. Equally we arrive at the conclusion that $c$ is invariant under conjugation with matrices $\mathrm{diag}(1,-1,1)$ and $\mathrm{diag}(-1,1,1)$. So, $c$ is invariant under conjugation with $-I$, and since it is a vector field of even degree, $c=0$. This proves (\ref{decomp2}).\\

 In particular, for $\ell=2$ the second summand in (\ref{decomp}) is a $0$-space, and so 
\begin{eqnarray*}
\mathrm{Ker}L^{\mathbb{T}}_{2}=\mathrm{Ker}L^{\widehat{\mathbb{T}}}_{2}\Rightarrow
\mathcal{U}^{\mathbb{T}}_{2}=\mathcal{U}^{\widehat{\mathbb{T}}}_{2}.
\end{eqnarray*} 
This tells that if a $2$-homogeneous vector field has a tetrahedral symmetry, it has a full tetrahedral symmetry. But this is exactly the fact \emph{``extension of symmetry"} we discovered in Note \ref{note-red2} and Corollary \ref{cor1}.\\

In words, the decomposition (\ref{decomp}) can be stated as follows:
\begin{itemize}
\item[$\mathbf{\star}$]Any homogeneous vector field of even degree with a tetrahedral symmetry decomposes uniquely into the sum of a vector field with a full tetrahedral symmetry, and a vector field with an octahedral symmetry.
\end{itemize}
\indent Now, the first coordinate of the vector field with a tetrahedral symmetry is of even degree in $x$ and of odd in each $y,z$. And, as with both cases $\mathbb{O}$ and $\widehat{\mathbb{T}}$, the other two coordinates are obtained by a cyclic permutation. So, the claim ``$\star$" becomes ``any homogeneous function in $y^2,z^2$ can be uniquely written as a sum of a symmetric and an antisymmetric function", and thus becomes an obvious fact.
\section{Spheres and solenoidality}
\label{sph-sol}
As a continuation to previous two topics, and in the spirit of Section \ref{tetra-sole}, we finish this discussion with calculation of $\ell$-homogeneous polynomial vector fields with octahedral symmetry, which are both solenoidal and are flows on spheres.\\

Thus, let $\ell\in2\mathbb{N}$, $\ell\geq 4$, and $X=(a,b,c)\in\mathcal{U}^{\mathbb{O}}_{\ell}=\mathcal{U}$. Then both $\mathrm{div}X$ and $xa+yb+zc$ are anti-symmetric polynomials, odd in each $x,y,z$, of total degrees $\ell-1$ and $\ell+1$, respectively. Such polynomials are given by
\begin{eqnarray*}
W\cdot P(x^2,y^2,z^2),\quad W=xyz(x^2-y^2)(y^2-z^2)(z^2-x^2),
\end{eqnarray*}
where $P$ is symmetric polynomial that matches a degree.
Let us consider the map
\begin{eqnarray*}
T:\mathcal{U}\mapsto W\cdot\mathbb{R}_{\frac{\ell-10}{2}}[x^2,y^2,z^2]^{S_{3}}\oplus W\cdot\mathbb{R}_{\frac{\ell-8}{2}}[x^2,y^2,z^2]^{S_{3}},
\end{eqnarray*}
given by $X\mapsto \mathrm{div}X\oplus xa+yb+zc$. For negative values of index (for example, $\ell-10<0$) the space is assumed to be a zero space. The dimension of the first and the second summands on the right are equal to, respectively,
\begin{eqnarray*}
\Big{\Vert}\frac{(\ell-4)^2}{48}\Big{\Vert},\quad \Big{\Vert}\frac{(\ell-2)^2}{48}\Big{\Vert}.
\end{eqnarray*}
  
Yet again, as in Section \ref{tetra-sole}, we see that the map $T$ is surjective.\\

We will now give a proof of this fact. Consider the vector field \begin{eqnarray}
W=(a,b,c)=\varpi(x,y,z)\bl\varpi(y,z,x)\bl\varpi(z,x,y),
\label{w-coord}
 \end{eqnarray} 
where 
\begin{eqnarray}
\varpi=\Big{(}E(x^2,y^2,z^2)+x^{2}F(x^2,y^2,z^2)\Big{)}yz(y^2-z^2).
\label{geras}
\end{eqnarray}
Here $E,F$ are arbitrary symmetric homogeneous polynomials in all three variables $X=x^2$, $Y=y^2$, $Z=z^2$, of degrees $\frac{\ell-4}{2}$ and $\frac{\ell-6}{2}$, respectively. Then it is easy to see that $W$ has an octahedral symmetry ($\varpi$ is of even degree in $x$, of odd in each $y,z$, and antisymmetric in $y,z$), and $xa+yb+zc=0$. This is the same as saying that $x^2+y^2+z^2$ and $x^4+y^4+z^4$ are the first integrals for the octahedral superflow (see Section \ref{orbits-oct}). Now, if the function can be written in the form $E+x^2F$ for $E,F$ symmetric in $x^2,y^2,z^2$, then this can be done in the unique way. Indeed, assuming there is the second way $\widehat{E}+x^2\widehat{F}$, we would get
\begin{eqnarray*}
E-\widehat{E}=x^{2}(\widehat{F}-F).
\end{eqnarray*}
Now, permuting variables does not change the left side, and this gives $\widehat{F}=F$. The construction (\ref{geras}) gives all vector fields on spheres with octahedral symmetry, as can be double-verified from the identity
\begin{eqnarray*}
\Big{\lfloor}\frac{\ell^2}{16}\Big{\rfloor}-\Big{\Vert}\frac{(\ell-2)^2}{48}\Big{\Vert}
=\Big{\Vert}\frac{\ell^2}{48}\Big{\Vert}+\Big{\Vert}\frac{(\ell+2)^2}{48}\Big{\Vert},
\end{eqnarray*} 
which is satisfied for even positive integers $\ell$. Indeed, the first summand on the left is the dimension of $\mathcal{U}$, the minus summand is the dimension of $\mathbb{R}_{\frac{\ell-8}{2}}[x^2,y^2,z^2]^{S_{3}}$, and two summand on the right are the dimensions of symmetric functions in $X=x^2,Y=y^2,Z=z^2$ of degrees $\frac{\ell-4}{2}$ (all functions $E$) and $\frac{\ell-6}{2}$ (all functions $F$), respectively, in $X,Y,Z$.\\

Now, let $W$ be of the form $(a,b,c)$, where, as before, $\varpi$ is given (\ref{geras}). We will prove that $T$ is surjective on the first coordinate, if we limit ourselves to such vector fields. Even more - $T$ is surjectve on a subspace generated by such vector fields with $F\equiv 0$. Indeed, all we need is to prove the following lemma.
\begin{lem}
\label{deriv}
 Let $n\geq 3$, $n\in\mathbb{N}$, $E_{n}$ be a space of symmetric $n$-homogeneous polynomials in variables $X,Y,Z$. Let us define a linear map $\mathscr{L}:E_{n}\mapsto E_{n-3}$ by $\mathscr{L}(F)=H$, where
\begin{eqnarray*}
f(X,Y,Z)(Y-Z)+f(Y,Z,X)(Z-X)+f(Z,X,Y)(X-Y)\\
=(X-Y)(Y-Z)(Z-X)H(X,Y,Z),\quad f=\frac{\p}{\p X}F.
\end{eqnarray*} 
Then $\mathscr{L}$ is surjective.
\end{lem}
But this lemma follows from the calculation 
\begin{eqnarray}
&&\mathscr{L}\Big{(}(X+Y+Z)^p(X^2+Y^2+Z^2)^q(X^3+Y^3+Z^3)^r\Big{)}\nonumber\\
&&=-3r(X+Y+Z)^p(X^2+Y^2+Z^2)^q(X^3+Y^3+Z^3)^{r-1},
\label{calc-a}
\end{eqnarray}
for $p,q,r\in\mathbb{N}_{0}$. Indeed, $\mathcal{L}$ is a derivation:
\begin{eqnarray*}
\mathcal{L}(GH)=\mathcal{L}(G)H+G\mathcal{L}(H).
\end{eqnarray*}
Note that three versions of $\mathcal{L}$ above act on (possibly) different vector spaces. The calculation (\ref{calc-a}) now follows from
\begin{eqnarray*}
\mathcal{L}(X+Y+Z)=\mathcal{L}(X^2+Y^2+Z^2)=0,\quad\mathcal{L}(X^3+Y^3+Z^3)=-3.
\end{eqnarray*}
We are left employ the fact that every symmetric function in $X,Y,Z$ is a linear combination of terms $(X+Y+Z)^p(X^2+Y^2+Z^2)^q(X^3+Y^3+Z^3)^r$ that match a degree. This even gives an explicit kernel of $\mathscr{L}$, consisting of linear combinations of symmetric functions $(X+Y+Z)^p(X^2+Y^2+Z^2)^q$ that match a degree. We will need this lemma once again in Section \ref{lie-brackets}.\\
 
Therefore, the sequence of dimensions we are looking for is equal to
\begin{eqnarray*}
a_{\frac{\ell}{2}}=\Big{\lfloor}\frac{\ell^2}{16}\Big{\rfloor}-\Big{\Vert}\frac{(\ell-4)^2}{48}\Big{\Vert}-\Big{\Vert}\frac{(\ell-2)^2}{48}\Big{\Vert}
=\Big{\lfloor}\frac{\ell^2+12\ell+16}{48}\Big{\rfloor}.
\end{eqnarray*}
This sequence $\{a_{s}:s\in\mathbb{N}\}$, stars from ($s=1$, or $\ell=2$, is also added at the beginning)
\begin{eqnarray*}
0,1,2,3,4,6,7,9,11,13,15,18,20,23,\ldots
\end{eqnarray*}
Its number in \cite{oeis} is A253186. It is equal the amount of ways to partition $s$ into $2$ or $3$ parts, or number of connected unlabeled multigraphs with 3 vertices and $s$ edges. As we have discovered, this number is equal to the dimension of $2s$-homogeneous polynomial vector fields with an octahedral symmetry, which are solenoidal and which are vector fields on spheres with origin at the center. 

\section{Vanishing points}
\label{van-26}
Let $X$ be a homogeneous vector field with an octahedral symmetry on a sphere (not necessarily solenoidal). Let $\m{x}\in\mathbb{R}^{3}$, $\m{x}\neq\m{0}$, be such a point whose stabilizer $G_{\m{x}}<\mathbb{O}$ is non-trivial. Since $X(\m{x})$ is invariant under the action of $G_{\m{x}}$, we immediately get that $X(\m{x})$ and $\m{x}$ are parallel. However, $X(\m{x})$ is tangent to a sphere $\mathbf{S}^{2}$, and so $X(\m{x})\cdot\m{x}=0$. This gives $X(\m{x})=0$. Thus, all such vector fields vanish on $\frac{24}{2}+\frac{24}{3}+\frac{24}{4}=26$ points, mentioned in Section \ref{sub5.1}. In general, there exist a number $v(X)\in\mathbb{N}_{0}$ such that a vector field $X$ vanishes at exactly
\begin{eqnarray*}
26+24v(X)
\end{eqnarray*} 
points, or infinitely many. 
\begin{Example} Such a vector field is given by (\ref{w-coord}), where $\varpi$ is of the form (\ref{geras}). Except for these $26$ points, the vanishing of $W$ implies $E(x^2,y^2,z^2)=0$, $F(x^2,y^2,z^2)=0$. Consider now the point $\m{x}_{0}=(x,y,z)=(\frac{2}{3},\frac{2}{3},\frac{1}{3})$ on $\mathbf{S}^{2}$. For this point, 
\begin{eqnarray*}
\def\arraystretch{1.8}
\left\{\begin{array}{c@{\qquad}l}
x^2+y^2+z^2=1,\\
x^4+y^4+z^4=\frac{11}{27},\\
x^6+y^6+z^6=\frac{43}{243}.
\end{array}\right.
\end{eqnarray*} 
This system has $24$ solutions, given  by all images of $\m{x}_{0}$ under $\mathbb{O}$. Let now $\ell=10$, and let us define
\begin{eqnarray*}
F&=&27(x^4+y^4+z^4)-11(x^2+y^2+z^2)^2,\\
E&=&99(x^6+y^6+z^6)-43(x^2+y^2+z^2)(x^4+y^4+z^4).
\end{eqnarray*}
The the degree $10$ vector field (\ref{w-coord}) has an octahedral symmetry, has $x^2+y^2+z^2$ as its first integral, and exactly $50$ points on the unit sphere where it vanishes.\\

On the other hand, if 
\begin{eqnarray*}
F&=&2(x^4+y^4+z^4)-(x^2+y^2+z^2)^2,\\
E&=&2(x^4+y^4+z^4)(x^2+y^2+z^2)-(x^2+y^2+z^2)^3,
\end{eqnarray*}
then 
\begin{eqnarray*}
\varpi=(x+y+z)(x-y-z)(x-y+z)(x+y-z)(2x^2+y^2+z^2),
\end{eqnarray*}
the degree $10$ vector field vanishes on infinitely many points, on 4 large circles on the unit sphere; see (\ref{iden}) in Chapter \ref{octahedral}.
\end{Example}

\section{Commutative octahedral vector fields}
\label{lie-brackets}

If two $n$-dimensional vector fields $X$ and $Y$ in $\mathbb{R}^{n}$ (or in an open domain of $\mathbb{R}^{n}$) are given in cartesian coordinates by
\begin{eqnarray*}
X=\sum\limits_{i=1}^{n}f_{i}\frac{\p}{\p x_{i}},\quad 
Y=\sum\limits_{i=1}^{n}g_{i}\frac{\p}{\p x_{i}},
\end{eqnarray*}
then their \emph{Lie bracket} is a vector field defined  by \cite{conlon,gadea}
\begin{eqnarray*}
[X,Y]=\sum\limits_{i=1}^{n}\Big{(}X(g_{i})-Y(f_{i})\Big{)}\frac{\p}{\p x_{i}}.
\end{eqnarray*}
Vector fields \emph{commute}, if $[X,Y]=0$. For corresponding flows then one has 
\begin{eqnarray*}
F^{t}\circ G^{s}(\m{x})=G^{s}\circ F^{t}(\m{x}),\text{ for }s,t\text{ small enough}.
\end{eqnarray*}

Let us use notations of Section \ref{sph-sol}. In this section we will find when the vector field $X$, given (\ref{wide}), which we now will denote by
\begin{eqnarray*}
\widetilde{\mathbf{C}}=a\frac{\partial}{\p x}+b\frac{\partial}{\p y}+c\frac{\partial}{\p z},
\end{eqnarray*}
and $Y$, given by  (\ref{w-coord}) and (\ref{geras}),  that is,
\begin{eqnarray*}
W=\varpi(x,y,z)\frac{\partial}{\p x}+\varpi(y,z,x)\frac{\partial}{\p y}+\varpi(z,x,y)\frac{\partial}{\p z},
\end{eqnarray*}
which  is a vector field on the unit sphere with an octahedral symmetry, commute: $[\widetilde{\mathbf{C}},W]=0$.\\

We will prove the following

\begin{prop}All polynomial homogeneous vector fields of even degree, which are vector fields on spheres, have a symmetry group $\mathbb{O}$, and which commute with $\widetilde{\mathbf{C}}$, are given by
\begin{eqnarray*}
W=\varpi(x,y,z)\bl\varpi(y,z,x)\bl\varpi(z,x,y),
\end{eqnarray*} where
\begin{eqnarray*}
\varpi=yz(y^2-z^2)Y(x^2+y^2+z^2,x^4+y^4+z^4).
\end{eqnarray*}
The last factor $Y(\m{x},\m{y})$ is homogeneous in $(\m{x},\m{y})$ with weights $(1,2)$.   
\end{prop}

\begin{proof}We need to find when
\begin{eqnarray}
\widetilde{\mathbf{C}}(\varpi)-W(a)=a\cdot\varpi_{x}+
b\cdot\varpi_{y}+c\cdot\varpi_{z}
&-&\varpi(y,z,x)a_{y}-\varpi(z,x,y)a_{z}
\label{vanishes}
\end{eqnarray}
vanishes (note that $a_{x}=0$). Then the other two coordinates vanish too, since they are just cyclic permutations of the (\ref{vanishes}).\\
 
First, assume that $F\equiv 0$. So, $\varpi=e\cdot a$, $e=E(x^2,y^2,z^2)$, $E$ is symmetric. Then (\ref{vanishes}) is equal to
\begin{eqnarray*}
a(e_{x}a+e_{y}b+e_{z}c)+ea_{y}b+ea_{z}c-ea_{y}b-ea_{z}c=a(e_{x}a+e_{y}b+e_{z}c).
\end{eqnarray*}
Second, assume now that $E\equiv 0$. So, $\varpi=f\cdot x^2\cdot a$, $f=F(x^2,y^2,z^2)$, $F$ is symmetric.
Then (\ref{vanishes}) is equal to
\begin{eqnarray*}
(2xfa+f_{x}x^{2}a)a+(x^2f_{y}a+x^2fa_{y})b+(x^2f_{z}a+x^{2}fa_{z})c-
fy^{2}a_{y}b-fz^{2}a_{z}c\\
=x^2a(f_{x}a+f_{y}b+f_{z}c)+f\Big{(}2xa^{2}+(x^2-y^2)a_{y}b+(x^{2}-z^{2})a_{z}c\Big{)}.
\end{eqnarray*}
Let us denote
\begin{eqnarray*}
2xa^2+(x^2-y^2)a_{y}b+(x^{2}-z^{2})a_{z}c=t(x,y,z),
\end{eqnarray*}
which a is rather complicated polynomial of degree $9$.\\

Now, let us consider the general case, when $E$ and $F$ are arbitrary, so that (\ref{geras}) holds, and degrees do match, so that $\varpi$ is homogeneous of even degree. As we have seen, a vector field $W$ commutes with  $\widetilde{\mathbf{C}}$ if and only if  
\begin{eqnarray}
a\cdot (e_{x}a+e_{y}b+e_{z}c)+x^2a\cdot (f_{x}a+f_{y}b+f_{z}c)+t\cdot f=0.
\label{pirmm}
\end{eqnarray}
All factors $\m{e}=e_{x}a+e_{y}b+e_{z}c$, $\m{f}=f_{x}a+f_{y}b+f_{z}c$, and $f$ are cyclic polynomials in $x,y,z$: the first two are anti-symmetric, and the last one is symmetric. Let us cyclically permute variables, thus obtaining a system of three linear equations for $(\m{e},\m{f},f)$. The determinant
\begin{eqnarray*}
\begin{vmatrix}
a & x^2 a & t(x,y,z)\\
b & y^2 b & t(y,z,x)\\
c & z^2 c & t(z,x,y) 
\end{vmatrix}=0.
\end{eqnarray*} 
Indeed, 
\begin{eqnarray*}
xt(x,y,z)+yt(y,z,x)+zt(z,x,y)=0
\end{eqnarray*}
by a direct calculation, and this gives a linear dependence among the rows of the matrix. The above matrix has therefore rank two, and thus an 
equivalent system can be taken as 

\begin{eqnarray*}
\left\{\begin{array}{c@{\qquad}c}
a\m{e}+x^2 a\m{f}=-t(x,y,z)f,\\
b\m{e}+y^2 b\m{f}=-t(y,z,x)f.
\end{array}\right.
\end{eqnarray*}
This implies
\begin{eqnarray}
ab(x^2-y^2)(f_{x}a+f_{y}b+f_{z}c)=f\Big{(}at(y,z,x)-bt(x,y,z)\Big{)}=fr(x,y,z).
\label{f-r}
\end{eqnarray}
Let us now translate this identity in terms of the variables $(X,Y,Z)=(x^2,y^2,z^2)$. \\

First, since $f(x,y,z)=F(X,Y,Z)$, then $f_{x}(x,y,z)=F_{X}(X,Y,Z)\cdot 2x$, so we can rewrite the above as
 \begin{eqnarray}
2XYZ(X-Y)(Y-Z)(Z-X)\cdot\nonumber\\
\Big{(}F_{X}(Y-Z)+F_{Y}(Z-X)+F_{Z}(X-Y)\Big{)}=FR,
\label{factor}
\end{eqnarray} 
where
\begin{eqnarray*}
R=\sum X^{4}Y^{2}+8\sum X^{3}Y^{2}Z-3\sum X^{4}YZ-\sum X^3Y^3-30X^2Y^2Z^2.
\end{eqnarray*}
The sums are symmetric sums; thus, for example, the coefficient at $X^3Y^3$ is $-2$. We will show that the factorization (\ref{factor}) implies that $F\equiv 0$.  \\

Now, note that $(X-Y)$ and $X$ does not divide $R$, what is clear from calculations. Let $(X-Y)^{r}$ and $X^{s}$ be the highest powers of $X-Y$ and $X$ dividing $F$, respectively. Since $F$ is symmetric, this shows that $r=2\ell$, and
\begin{eqnarray*}
F=P^{\ell}Q^{s}H,\quad s,\ell\in\mathbb{N},\quad P=(X-Y)^2(Y-Z)^2(Z-X)^2,\quad Q=XYZ.
\end{eqnarray*}

 Let us now turn to Lemma \ref{deriv}. Since $\mathcal{L}$ is a derivation,
\begin{eqnarray*}
\mathcal{L}(P^{\ell}Q^{s}H)=\ell P^{\ell-1}Q^{s}H\mathcal{L}(P)
+sP^{\ell}Q^{s-1}H\mathcal{L}(Q)+P^{\ell}Q^{s}\mathcal{L}(H).
\end{eqnarray*}
Let us plug everything into (\ref{factor}). We obtain
\begin{eqnarray}
2\Big{(}\ell QH\mathcal{L}(P)
+sPH\mathcal{L}(Q)+PQ\mathcal{L}(H)\Big{)}=HR.
\label{factor-2}
\end{eqnarray}
By a direct calculation,
\begin{eqnarray*}
\mathcal{L}(P)=-2(Y+Z-2X)(Z+X-2Y)(X+Y-2Z),\quad \mathcal{L}(Q)=-1.
\end{eqnarray*}
The identity (\ref{factor-2}) implies that
\begin{eqnarray*}
H|PQ\mathcal{L}(H)\text{ in }\mathbb{R}[X,Y,Z].
\end{eqnarray*}
 But we know that $(H,PQ)=1$, and $\mathrm{deg}(\mathcal{L}(H))=\mathrm{deg}(H)-3$. This is a contradiction, unless $\mathcal{L}(H)=0$.\\
 
 Hence, suppose that $\mathcal{L}(H)=0$, $H\neq 0$. Then (\ref{factor-2}) gives
 \begin{eqnarray*}
2\ell Q\mathcal{L}(P)
+2sP\mathcal{L}(Q)=R.
\end{eqnarray*}
This does not have a solution in positive integers $(s,\ell)$, but rather $(s,\ell)=(-\frac{1}{2},-\frac{1}{2})$. Therefore $H=0$.\\

Returning to the initial setting, this shows that $f\equiv 0$, and this implies
\begin{eqnarray*}
e_{x}a+e_{y}b+e_{z}c\equiv 0.
\end{eqnarray*}
And so, $e$ is the first integral for the vector field $\widetilde{\mathbf{C}}$. Therefore it is a polynomial in $x^2+y^2+z^2$ and $x^4+y^4+z^4$. The Proposition is proved.
\end{proof}

When the orbit for the superflow with the vector field $\widetilde{\mathbf{C}}$ is fixed, $x^2+y^2+z^2$ and $y^4+z^4+z^4$ are fixed, too. So the vector field $W$ is proportional to $\widetilde{\mathbf{C}}$. Thus, if we are allowed to use the notion of \emph{polynomial homogeneous vector field on a unit sphere} meaning a vector field which is just a restriction on $\mathbf{S}^{2}$ of such a vector field, then 
\begin{cor}
Polynomial homogeneous vector field on a unit sphere, which commutes with $\widetilde{\mathbf{C}}$, is always collinear to the latter on any particular orbit.
\end{cor}
This shows that no non-trivial polynomial examples exist, though the scaling factor differs from orbit to orbit.\\

In Appendix we formulate one extremity problem concerning polynomial homogeneous vector fields on unit spheres $\mathbf{S}^{n-1}$. The restriction of this problem to the octahedral, or any other irreducible finite subgroup of $SO(n)$, is of big interest, too. 
\begin{Note} Let us take any solution to $\mathcal{L}(H)=0$; for example, $H=1$. If we choose $(s,\ell)=(-\frac{1}{2},-\frac{1}{2})$, in terms of $(x,y,z)$, the function $f$ is given by
\begin{eqnarray*}
f(x,y,z)=\frac{1}{xyz(x^2-y^2)(y^2-z^2)(z^2-x^2)}.
\end{eqnarray*}
We directly verify with MAPLE that (\ref{f-r}) is indeed satisfied. We are now on the right track to find not polynomial, but rather rational vector fields which commute with $\widetilde{\mathbf{C}}$. This topic will be continued in \cite{alkauskas-super2}.  
\end{Note}

%% file: sf-chap7.tex
\chapter{Spherical constants and projections}
\section{Three fundamental octahedral spherical constants}
\label{funda-period}
In this Section we will carry out some arithmetic calculations for the octahedral superflow related to average value of the length of the vector field; the same techniques apply generally to a wider class of spherical superflows - see Problem \ref{prob-vienas}, and also Appendix \ref{app}.\\ 

On the unit sphere, the vector field $\widetilde{\m{C}}=y^3z-yz^3\bl z^3x-zx^3\bl x^3y-xy^3$ (see (\ref{wide}), where it coincides with $\m{C}$) is the vector field for the octahedral polynomial superflow, and it is canonically normalized (modulo $\pm1$) in a sense that its coefficients are integers with the g.c.d. equal to one. We can therefore define four real numbers, different average values for the length of $\widetilde{\m{C}}$ on the unit sphere $\m{S}^{2}$. All these numbers belong (except for the first one - but see Problem \ref{prob-period}) to the extended ring of periods \cite{konzag}:
\begin{eqnarray*}
\alpha_{0}&=&\exp\Big{(}\frac{1}{4\pi}\int\limits_{\m{S}^{2}}\ln|\widetilde{\m{C}}|\d S\Big{)},\\
\alpha_{1}&=&\frac{1}{4\pi}\int\limits_{\m{S}^{2}}|\widetilde{\m{C}}|\d S,\\
\alpha_{2}&=&\frac{1}{4\pi}\int\limits_{\m{S}^{2}}|\widetilde{\m{C}}|^{2}\d S,\\
\alpha_{\infty}&=&\sup\limits_{\m{S}^{2}}|\widetilde{\m{C}}|.
\end{eqnarray*}
Since 
\begin{eqnarray*}
|\widetilde{\m{C}}|^2=x^2y^2(x^4+y^4)+x^2z^2(x^4+z^4)+y^2z^2(y^4+z^4)-
2(x^4y^4+x^4z^4+y^4z^4),
\end{eqnarray*}
then by a direct calculation, 
\begin{eqnarray*}
\alpha_{2}=\frac{4}{105},\quad \alpha_{\infty}=\frac{\sqrt{827+73\sqrt{73}}}{96\sqrt{2}}.
\end{eqnarray*}
The maximum for $\alpha_{\infty}$ is attained, for example, at
\begin{eqnarray*}
x\bl y\bl z=\sqrt{\frac{11+\sqrt{73}}{24}}\bl \sqrt{\frac{13-\sqrt{73}}{48}}\bl\sqrt{\frac{13-\sqrt{73}}{48}}.
\end{eqnarray*}
So the average (in the square norm) and maximal length of the vector on the unit sphere is equal to 
\begin{eqnarray*}
\sqrt{\alpha_{2}}=\frac{2}{\sqrt{105}}=0.1951800147_{+},\quad \alpha_{\infty}=0.2805462137_{+}.
\end{eqnarray*} 
However, as can be expected, $\alpha_{0}$ and $\alpha_{1}$ are more complicated. For example, we have
\begin{eqnarray*}
\alpha_{1}=\frac{1}{2\pi}\int\limits_{0}^{2\pi}\int\limits_{0}^{1}
\frac{r \sqrt{T(r,\phi)}}{\sqrt{1-r^2}}\d r\d\phi,
\end{eqnarray*}
where
\begin{eqnarray*}
T(r,\phi)&=&r^{8}(-4\cos^{8}\phi+8\cos^{6}\phi-12\cos^{4}\phi+8\cos^{2}\phi-4)\\
&+&r^{6}(11\cos^{4}\phi-11\cos^{2}\phi+8)+r^{4}(-4\cos^{4}\phi+4\cos^{2}\phi-5)+r^{2}.
\end{eqnarray*}
 
Now, if $\alpha_{0}(t,r)$, $\alpha_{1}(t,r)$, $\alpha_{2}(t,r)$, and $\alpha_{\infty}(t,r)$ are the corresponding averages for the vector field $t\widetilde{\m{C}}$ on the sphere $x^2+y^2+z^2=r^2$ (the constant $(4\pi)^{-1}$ is replaced by $(2\pi r^{2})^{-1}$), then
\begin{eqnarray*}
\begin{tabular}{l c}
$\alpha_{0}(t,r)=tr^{4}\alpha_{0}(1,1)$,& $\alpha_{1}(t,r)=tr^{4}\alpha_{2}(1,1),$\\
$\alpha_{2}(t,r)=t^{2}r^{8}\alpha_{2}(1,1)$,&
$\alpha_{\infty}(t,r)=tr^{4}\alpha_{\infty}(1,1)$.
\end{tabular}
\end{eqnarray*}

Therefore,
\begin{eqnarray*}
\frac{\alpha^{2}_{0}(t,r)}{\alpha_{2}(t,r)}=\frac{\alpha^{2}_{0}(1,1)}{\alpha_{2}(1,1)},\quad\frac{\alpha^{2}_{1}(t,r)}{\alpha_{2}(t,r)}=\frac{\alpha^{2}_{1}(1,1)}{\alpha_{2}(1,1)},\quad
\frac{\alpha^{2}_{\infty}(t,r)}{\alpha_{2}(t,r)}=\frac{\alpha^{2}_{\infty}(1,1)}{\alpha_{2}(1,1)}.
\end{eqnarray*}
Thus, we have discovered three fundamental ratios, the second one (possibly, and the logarithm of the first one) being from the extended ring of periods $\widehat{\mathscr{P}}$ \cite{konzag}, and the third one being an algebraic number, which are independent from the choice of the specific first integral $x^2+y^2+z^2=r^{2}$ of the octahedral superflow, and also independent from the choice of the normalizing constant for the corresponding vector field. These ratios are therefore canonically wired into the octahedral group itself, with a condition that the representation of the octahedral group is orthogonal:
\begin{eqnarray}
\Omega_{0,\mathbb{O}}=\frac{\alpha_{0}^{2}}{\alpha_{2}},\quad
\Omega_{1,\mathbb{O}}=\frac{\alpha_{1}^{2}}{\alpha_{2}},\quad
\Omega_{\infty,\mathbb{O}}=\frac{\alpha_{\infty}^{2}}{\alpha_{2}}=\frac{28945+2555\sqrt{73}}{24576}.
\label{omega-2}
\end{eqnarray}
Note that
\begin{eqnarray*}
\mathcal{N}_{\mathbb{Q}}(\Omega_{\infty,\mathbb{O}})=\frac{5^2\cdot 7^2}{2^{11}}.
\end{eqnarray*}
By the Cauchy-Schwarz inequality,
\begin{eqnarray*}
0<\Omega_{0,\mathbb{O}}<\Omega_{1,\mathbb{O}}<1<\Omega_{\infty,\mathbb{O}}=2.066037173_{+}.
\end{eqnarray*}
Equally, we can perform analogous analysis for the second first integral, the surface $x^4+y^4+z^4=1$. However, this applies only to the octahedral group in dimension $3$, and second, the choice of the first integral of degree $4$ is not canonical, since $x^4+y^4+z^4+t(x^2+y^2+z^2)^2$ is also the first integral. This justifies the following definition.
\begin{defin}Let $\m{Q}$ be the vector field in dimension $n$ for the polynomial spherical superflow (see Problem \ref{prob-vienas} and explanations just succeeding it). Then the 
\emph{spherical constants} $\Omega_{v,\Gamma}$, $v\in\{0,1,\infty\}$, for the superflow $\phi_{\Gamma}$ are defined by 
\begin{eqnarray*}
\Omega_{v,\Gamma}=\frac{\beta_{v}^{2}}{\beta_{2}},
\end{eqnarray*}
 where
\begin{eqnarray*}
\beta_{0}&=&\exp\Big{(}\frac{1}{|\m{S}^{n-1}|}\int\limits_{\m{S}^{n-1}}\ln|\m{Q}|\d S\Big{)},\\
\beta_{1}&=&\frac{1}{|\m{S}^{n-1}|}\int\limits_{\m{S}^{n-1}}|\m{Q}|\d S,\\
\beta_{2}&=&\frac{1}{|\m{S}^{n-1}|}\int\limits_{\m{S}^{n-1}}|\m{Q}|^{2}\d S,\\\quad
\beta_{\infty}&=&\sup\limits_{\m{S}^{n-1}}|\m{Q}|.
\end{eqnarray*}
\label{defin-const}
\end{defin}
Here $\m{S}^{n-1}$ is the $(n-1)$-dimensional unit sphere $\sum_{i=1}^{n}x_{i}^{2}=1$, and $|\m{S}^{n-1}|$ is its $(n-1)$-dimensional volume, given by
\begin{eqnarray*}
|\m{S}^{n-1}|=\frac{2\pi^{n/2}}{\Gamma(n/2)}.
\end{eqnarray*}
 The definition above is independent on exact scaling on the vector field we are using, and remains intact if we change the sphere to $r\m{S}^{n-1}$. So, constants $\Omega_{v,\Gamma}$ are the constants depending only on $\Gamma$, and moreover, they are the same for any other conjugate group $\gamma^{-1}\circ\Gamma\circ\gamma$, $\gamma\in O(n)$. \\
 
Note that in all examples we obtained (octahedral, hyper-octahedral and icosahedral in \cite{alkauskas-super2}), $|\m{Q}|^{2}$ is a function in $x_{i}^{2}$, $1\leq i\leq n$, and if the coefficients of $\m{Q}$ are algebraic numbers (rational numbers in the octahedral and hyper-octahedral case and belonging to $\mathbb{Q}[\,\sqrt{5}\,]$ in the icosahedral case), it is easy to see that $\beta_{2}$ is an algebraic number. In other words, we pose
\begin{prob}
\label{prob-period}
Is it true that for the spherical superflow $\phi_{\Gamma}$, the spherical constant $\Omega_{1,\Gamma}\in\widehat{\mathscr{P}}$? What about $\log\Omega_{0,\Gamma}$? 
\end{prob}
The expression for $\log\Omega_{0,\Gamma}$ is analogous to the one used in defining the (logarithmic) Mahler measure \cite{zudilin}, so it is not excluded that the answer to the second half of the above question is also affirmative.\\

See Appendix \ref{app} where we formulate a related problem of extremal vector fields on spheres.

\section{Stereographic projection} 
\label{sec6.3}
In case $x\varpi+y\varrho+z\sigma=0$, in order to to visualize the vector field $\varpi\bl\varrho\bl\sigma$ on the unit sphere, we use a stereographic projection from the point $\mathbf{p}=(0,0,1)$ (where the vector field vanishes in the octahedral case) to the plane $z=-1$. There is a canonical way to do a projection of a vector field as follows. Let $F(\m{x},t)=t^{-1}\phi(\m{x}t)$ be our flow on the unit sphere, and $\tau$ - the stereographic projection from the point $\mathbf{p}$ to the plane $z=-1$. We call a $2$-dimensional vector field $\mathcal{Y}$ in the plane $z=-1$ a \emph{stereographic projection} of the vector field $\mathcal{X}=\varpi\bl\varrho\bl\sigma$, if for a flow $G$,  associated with the vector field $\mathcal{Y}$, one has
\begin{eqnarray}
G(\tau(\m{x}),t)=\tau(F(\m{x},t))\text{ for all }\m{x}\text{ and }t\text{ small enough}.
\label{proj}
\end{eqnarray}
Let $\mathbf{a}=(x,y,z)$, $\mathbf{a}\neq\mathbf{p}$, $x^2+y^2+z^2=1$. The line $\overline{\mathbf{pa}}$ intersects the plane $z=-1$ at the point $\mathbf{b}$, which we call $\tau(\mathbf{a})$; this is, of course, a standard stereographic projection. Now, consider the line $(x+\varpi t,y+\varrho t, z+\sigma t)$, $t\in\mathbb{R}$. Via a stereographic projection (only one point of it belongs the unit sphere, but nevertheless stereographic projection works in the same way), this line is mapped into a smooth curve (in fact, a quadratic) $\mathcal{Q}(\mathbf{a})$ contained in the plane $z=-1$. Its tangent vector at the point $\mathbf{a}$ is a direction of a new vector in the plane $z=-1$.\\

The point $\mathbf{a}=(x,y,z)\in\mathbf{S}^{2}$ is mapped into a point
\begin{eqnarray}
\tau(\mathbf{a})=(\alpha,\beta)=\Big{(}\frac{2x}{1-z},\frac{2y}{1-z}\Big{)}
\label{tau-simp}
\end{eqnarray}
(where $z=-1$), and the point $(x+\varpi t,y+\varrho t, z+\sigma t)$ - into a point
\begin{eqnarray*}
\alpha(t)\bl\beta(t)=\tau(t,\mathbf{a})=\frac{2x+2t\varpi}{1-z-t\sigma}\bl\frac{2y+2t\varrho}{1-z-t\sigma}.
\end{eqnarray*}
If $\gamma=\gamma_{x,y,z}$ and $\delta=\delta_{x,y,z}$ are the corresponding M\"{o}bius transformations on the right (as a functions in $t$, where $x,y,z$ being fixed, and consequently $\varpi,\varrho,\sigma$ also), so that $\alpha(t)=\gamma(t)$, $\beta(t)=\delta(t)$, then the equation for the curve $\mathcal{Q}(\mathbf{a})$ can be written as $\gamma^{-1}(\alpha)=\delta^{-1}(\beta)$, which shows immediately that it is a quadratic. Of course, we are interested only in its germ at $\tau(\mathbf{a})$. 
The inverse of $\tau$, as given by (\ref{tau-simp}), is
\begin{eqnarray}
\tau^{-1}(\alpha,\beta)=\frac{4\alpha}{\alpha^2+\beta^2+4}
\bl\frac{4\beta}{\alpha^2+\beta^2+4}\bl\frac{\alpha^2+\beta^2-4}{\alpha^2+\beta^2+4}.\label{ab-inv}
\end{eqnarray}
Now,
\begin{eqnarray}
\frac{\d}{\d t}\tau(t,\mathbf{a})\Big{|}_{t=0}=\frac{2\varpi-2\varpi z+2x\sigma}{(1-z)^2}\bl
\frac{2\varrho-2\varrho z+2y\sigma}{(1-z)^2}.
\label{true-vec}
\end{eqnarray}
If we use now a substitution (\ref{ab-inv}), that is, plug $(x,y,z)\mapsto \tau^{-1}(\alpha,\beta)$ into the above, we obtain exactly the stereographic projection of the vector field $\varpi\bl\varrho\bl\sigma$. Indeed, $F(\m{x},t)=x+t\varpi+O(t^2)\bl y+t\varrho+O(t^2)\bl z+t\sigma+O(t^2)$. So, $\tau(F(\m{x},t))=\tau(t,\mathbf{a})+O(t^2)$, and the answer follows.\\

To get a better visualization, we multiply (\ref{true-vec}) by $1-z$; this does not change nor orbits neither directions, but changes the flow (different parametrization of the same curves). We thus arrive at the vector field
\begin{eqnarray}
2\varpi+\frac{2x}{1-z}\sigma\bl 2\varrho+\frac{2y}{1-z}\sigma. 
\label{exprr}
\end{eqnarray}
Or, minding (\ref{ab-inv}), in terms of $(\alpha,\beta)$ this reads as
\begin{eqnarray*}
&&2\varpi\Big{(}\tau^{-1}(\alpha,\beta)\Big{)}+\alpha\sigma\Big{(}\tau^{-1}(\alpha,\beta)\Big{)}
\bl 2\varrho\Big{(}\tau^{-1}(\alpha,\beta)\Big{)}+\beta\sigma\Big{(}\tau^{-1}(\alpha,\beta)\Big{)}\\
&=&\Pi(\alpha,\beta)\bl\Theta(\alpha,\beta).
\end{eqnarray*}
This is a vector field on the $(\alpha,\beta)$ plane. This is inessential for visualisation, but for purposes described in Section \ref{orth-sub}, note that the true projection of the vector field $\mathcal{X}$ is 
\begin{eqnarray}
\frac{1}{8}(\alpha^2+\beta^2+4)\Pi(\alpha,\beta)\bl \frac{1}{8}(\alpha^2+\beta^2+4)\Theta(\alpha,\beta), 
\label{ico-pro}
\end{eqnarray}
and is given by a pair of rational functions, but not $2-$ homogeneous anymore.\\

 If $\varpi\bl\varrho\bl\sigma$ vanishes on the unit sphere at a point $\mathbf{a}$, so does $\Pi\bl\Theta$ at a point $\tau(\mathbf{a})$. In the other direction - suppose $\Pi=0$, $\Theta=0$ for a certain $(\alpha,\beta)$. Then there is a point $(x,y,z)\neq (0,0,1)$ on the unit sphere such that both expressions in (\ref{exprr}) vanish. Recall that we also have $x\varpi+y\varrho+z\sigma=0$. Now, the determinant of the  matrix
\begin{eqnarray*}
\begin{pmatrix}
1-z & 0 & x\\
0 & 1-z & y\\
x & y & z 
\end{pmatrix}
\end{eqnarray*}
is equal to $-(z-1)^2\neq 0$. This implies $\varpi\bl\varrho\bl\sigma=0\bl 0\bl 0$.\\

So, in Figure \ref{figure4} we present a stereographic projection of the octahedral superflow, treated in detail in Sections \ref{octahedral}, \ref{partiii} and \ref{partiv}. This time the singular orbits are curves $\{x^2+y^2+z^2=1,x^4+y^4+z^4=\frac{1}{2}\}$.  This is a reducible case bacause of the identity (\ref{iden}).
So, these singular orbits are intersections of the unit sphere with $4$ planes. A direct calculation shows that stereographically these four circles in the $(\alpha,\beta)$ plane map into $4$ circles with radii $\sqrt{12}$, and centres $(\pm 2,\pm 2)$ (signs are independent). Figure \ref{figure4} shows the setup.

\begin{figure}
\includegraphics[width=85mm,height=85mm,angle=-90]{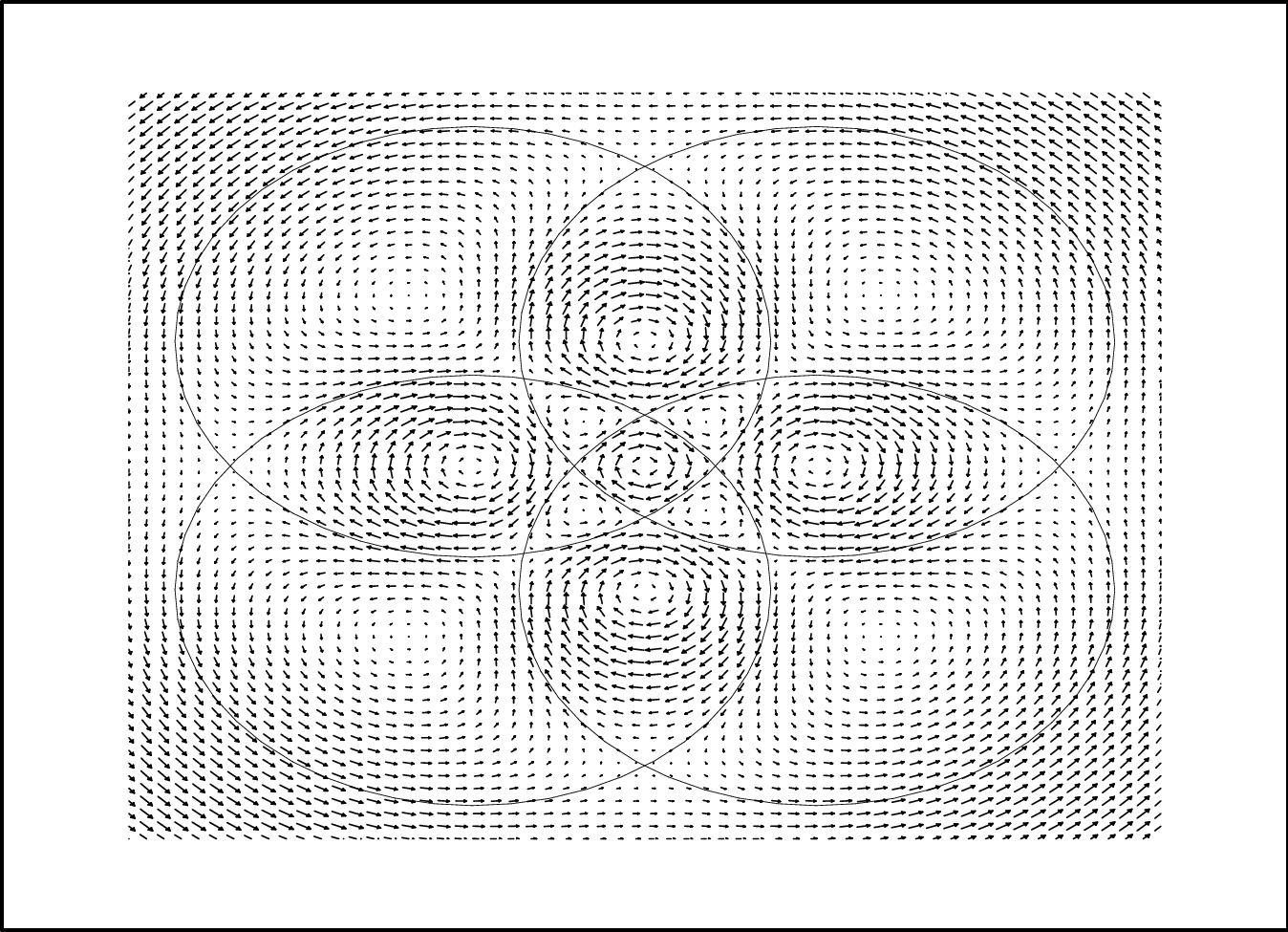}
\caption{Stereographic projection of the octahedral superflow, $-3\leq\alpha,\beta\leq 3$.}
\label{figure4}
\end{figure}

\section{Orthogonal and other birational projections}
\label{orth-sub}
In the same vein, as another, much simpler example how to visualize a $3$-dimensional vector field, we will give an orthogonal projection of the vector field
\begin{eqnarray*}
\mathbf{S}=\varpi\bl\varrho\bl\sigma=yz\bl xz\bl xy,
\end{eqnarray*}
which gives rise to the tetrahedral superflow $\phi_{\widehat{\mathbb{T}}}$, as described in Section \ref{S4}.\\

Consider the surface $\mathscr{Q}=\{(x,y,z)\in\mathbb{R}^{3}:x^2-y^2=1,x>0\}$, which is one of a two connected components of a quadratic. Since $p^2-q^2$ is the first integral of the differential system (\ref{sys-in}), the vector field $\mathbf{S}$ is tangent to $\mathscr{Q}$. We will project now orthogonally $\mathbf{S}$ to the plane $x=0$, exactly as in the previous section. Let $\tau(\m{x})=\tau(x,y,z)=(y,z)$ be this projection. Let $F(\m{x},t)=t^{-1}\phi(\m{x}t)$ be a flow on $\mathscr{Q}$. Equally, we call a $2$-dimensional vector field $\mathcal{Y}$ in the plane $x=0$ an \emph{orthogonal projection} of the vector field $\mathbf{S}$, if for a flow $G$, associated with the vector field $\mathcal{Y}$, one has (\ref{proj}).   \\

So, let $\mathbf{a}=(x,y,z)$ belongs to $\mathscr{Q}$. It is mapped to the point $\tau(\m{a})=(\alpha,\beta)=(y,z)$ (where $x=0$), and the point $(x+\varpi t,y+\varrho t,z+\sigma t)$ on the tangent line - to a point $\tau(t,\m{a})=(y+\varrho t,z+\sigma t)$. The inverse of $\tau$ is given by
\begin{eqnarray*}
\tau^{-1}(\alpha,\beta)=\sqrt{1+\alpha^2}\bl\alpha\bl\beta.
\end{eqnarray*}
Now, obviously, 
\begin{eqnarray*}
\frac{\d}{\d t}\tau(t,\m{a})\Big{|}_{t=0}=\varrho\bl\sigma.
\end{eqnarray*}   
As before, we now plug $\tau^{-1}(\alpha,\beta)\mapsto(x,y,z)$ into the above to obtain an orthogonal projection of our vector field $\mathbf{S}$. Thus, we have the vector field
\begin{eqnarray}
\Pi(\alpha,\beta)\bl\Theta(\alpha,\beta)=\beta\sqrt{1+\alpha^2}\bl \alpha\sqrt{1+\alpha^2}.
\label{pi-theta}
\end{eqnarray} 
On a surface $\mathscr{Q}$, the orbits of the flow with the vector field $\mathbf{S}$ are intersections of $\mathscr{Q}$ with $x^2-z^2=\xi$, $\xi\in\mathbb{R}$. Generically, these  are elliptic curves, except for cases $\xi=0$ and $\xi=1$. These exceptions correspond to intersection of $\mathscr{Q}$ with the planes $x=\pm z$ and $z=\pm y$.  In the $(\alpha,\beta)$ plane, this corresponds to a quadratic $\beta^2-\alpha^2=1$, and two lines $\alpha=\pm \beta$, respectively. In general, the orbits of the flow with the vector field $\Pi\bl \Theta$ are quadratics $\alpha^2-\beta^2=\phi$, $\phi\in\mathbb{R}$. Thus, we get a standard partition of the plane into hyperbolas having two fixed lines as their asymptotes. \\

Note that, differently from a previous section and the vector field (\ref{ico-pro}), the vector field $\Pi\bl\Theta$ in this Section is not rational, since the map $\tau$ is not a birational map. However, the projected flow can be described in rational terms, too, if we use a projection to the plane $x-y=0$ rather that to the plane $x=0$. This time not only $\mathcal{Q}$ is projected injectively, but the whole hyperbolic cylinder $\widehat{\mathcal{Q}}=\{(x,y,z)\in\mathbb{R}^{3}:x^2-y^2=1\}$.\\

 Indeed, suppose, a point $(\alpha,\alpha,\beta)$ in the $x-y=0$ plane  is described by global coordinates $\alpha\bl\beta$. Thus, if $(x,y,z)\in\widehat{\mathcal{Q}}$, then the orthogonal projection to the plane $x-y=0$, and its inverse, are given by
\begin{eqnarray*}
\hat{\tau}(x,y,z)=\frac{x+y}{2}\bl z,\quad
\hat{\tau}^{-1}(\alpha,\beta)=\Big{(}\alpha+\frac{1}{4\alpha}\Big{)}\bl \Big{(}\alpha-\frac{1}{4\alpha}\Big{)}\bl\beta.
\end{eqnarray*}
In the same fashion as before, in the $(\alpha, \beta)$ plane we get a vector field
\begin{eqnarray}
\hat{\Pi}(\alpha,\beta)\bl\hat{\Theta}(\alpha,\beta)=\alpha\beta\bl\alpha^{2}-\frac{1}{16\alpha^2}.
\label{al-be}
\end{eqnarray}
Thus, using the result of Theorem \ref{thm-s4}, we get
\begin{prop}
\label{jac-alt}
The $2$-dimensional vector fields $\Pi\bl\Theta$ and $\hat{\Pi}\bl\hat{\Theta}$, as given by (\ref{pi-theta}) and (\ref{al-be}), respectively, which produce flows, but not projective flows any longer, can be explicitly integrated in terms of Jacobi elliptic functions.  
\end{prop}
However, differently from $\Pi\bl\Theta$, $\hat{\Pi}\bl\hat{\Theta}$ is given by rational functions. So, though the vector field $\hat{\Pi}\bl\hat{\Theta}$ is not given by a pair of $2-$homogeneous functions, it still can be described via a projective flow $\phi_{\hat{\mathbb{T}}}$ investigated in Section \ref{S4}.\\

Further, let $x^2-y^2=1$, $x^2-z^2=\xi$. Thus, we take any point on one particular orbit. Let, as before, $\alpha\bl\beta=\frac{x+y}{2}\bl z$. We know that $\alpha+\frac{1}{4\alpha}=x$. This can be checked directly, using the relation $\frac{1}{x+y}=x-y$. Thus, we obtain the equation
\begin{eqnarray*}
\mathscr{W}(\alpha,\beta)=\Big{(}\alpha+\frac{1}{4\alpha}\Big{)}^2-\beta^2=\xi.
\end{eqnarray*} 
This is the equation for the generic orbit for the flow with the vector field $\hat{\Pi}\bl\hat{\Theta}$, and for $\xi\neq 0,1$ it is an elliptic curve. We can double verify that the following holds:
\begin{eqnarray*}
\mathscr{W}_{\alpha}\hat{\Pi}+\mathscr{W}_{\beta}\hat{\Theta}\equiv 0.
\end{eqnarray*}

To this account, we make one elementary observation. Let
\begin{eqnarray*}
\sum\limits_{k=1}^{n}f_{k}\frac{\p}{\p x_{k}}
\end{eqnarray*}
be any $n$-dimensional vector field in $\mathbb{R}^{n}$. Let us define
\begin{eqnarray*}
\tilde{f}_{k}(x_{1},x_{2},\ldots,x_{n},z)=
z^{2}f_{k}\Big{(}\frac{x_{1}}{z},\frac{x_{2}}{z},\ldots,\frac{x_{n}}{z}\Big{)},\quad k=1,\ldots,n.
\end{eqnarray*}
Then these functions are $2$-homogeneous, and so the vector field
\begin{eqnarray}
\sum\limits_{k=1}^{n}\tilde{f}_{k}\frac{\p}{\p x_{k}}+0 \frac{\p}{\p z}
\label{confine}
\end{eqnarray}
gives rise to an $(n+1)$-dimensional projective flow. Note that $z=1$ is its integral surface. If we confine the projective flow (\ref{confine}) to the plane $z=1$, we recover the initial flow. Hence, we have
\begin{prop}
\label{prop-ner}
Any smooth flow in $\mathbb{R}^{n}$ is a section on a hyperplane of a $(n+1)$-dimensional projective flow. In other words, for any flow $F(\m{x},t):\mathbb{R}^{n}\times\mathbb{R}\mapsto\mathbb{R}^{n}$, there exists a projective flow $\phi:\mathbb{R}^{n+1}\mapsto\mathbb{R}^{n+1}$, such that
\begin{eqnarray*}
F(\m{x},t)=\frac{1}{t}\phi\Big{(}(\m{x},1)t\Big{)}=\frac{1}{t}\phi\big{(}\m{x}t,t\big{)}.
\end{eqnarray*}
\end{prop} 
This holds even if the flow $F$ is defined for $\m{x}$ belonging to an open region of $\mathbb{R}^{n}$, and $t$ small enough (depending on $\m{x}$). On the other hand, the projective flow (\ref{confine}) is of a special kind, and hence the variety of projective flows in dimension $(n+1)$ is bigger than that of general flows in dimension $n$. So, philosophically, we may consider (\ref{funk}) for $\m{x}\in\mathbb{R}^{n}$ as a general flow equation in dimension $n-\frac{1}{2}$. \\

However, the scenario in Proposition \ref{prop-ner} is not the only way to obtain a flow as a projection of a projective flow, as we have seen.\\

Suppose $\varpi\bl\varrho\bl\sigma$ is a triple of $2$-homogeneous rational functions, and let $\mathscr{W}$ be a homogeneous function and the first integral of the flow, as in (\ref{ffirst-int}). Suppose now, an integral surface $\mathcal{S}$, given by $\mathscr{W}(x,y,z)=1$, is of genus $0$. For example, in Proposition \ref{prop-ner} this is just $z=\mathrm{const.}$ \\

 In the most general case, by a \emph{projection of a vector field} we mean the following. Consider a rational parametrization
\begin{eqnarray*} 
\pi:\mathbb{R}^{2}\mapsto\mathcal{S},\quad 
\pi:(\alpha,\beta)\mapsto\big{(}f(\alpha,\beta),g(\alpha,\beta),h(\alpha,\beta)\big{)}.
\end{eqnarray*} 
  Suppose that its inverse $\pi^{-1}:\mathcal{S}\mapsto\mathbb{R}^2$ is also rational, and is given by rational functions $\pi^{-1}:(x,y,z)\mapsto(A(x,y,z),B(x,y,z))$. Thus, the projection of our vector field is a new vector field in the $(\alpha,\beta)$ plane given by
\begin{eqnarray*}
\Big{(}A_{x}\varpi+A_{y}\varrho+A_{z}\sigma, 
B_{x}\varpi+B_{y}\varrho+B_{z}\sigma\Big{)}\Big{|}_{(x,y,z)\mapsto\big{(}f(\alpha,\beta),g(\alpha,\beta),h(\alpha,\beta)\big{)}}.
\end{eqnarray*}
In both cases dealt with in this paper and \cite{alkauskas-super2} (a sphere in icosahedral and octahedral superflow cases, and a hyperbolic cylinder in a tetrahedral superflow case, orthogonal projection to the plane $x-y=0$), surfaces are of genus $0$, and are birationally parametrized.

%% file: sf-chap8.tex
\chapter{The superflow $\phi_{\mathbb{O}}$. Singular orbit}
\label{partiii}
Note that, in the setting of Section \ref{elementary}, if we are dealing with flows over real numbers, the cases $\xi=\frac{1}{3}$ or $\xi=1$ need not be investigated, since, for example, the only real numbers such that
\begin{eqnarray*}
\frac{(x^2+y^2+z^2)^2}{x^4+y^4+z^4}=3
\end{eqnarray*} 
are those of the form $(x,y,z)=(\pm t,\pm t,\pm t)$ (signs are independent), and vector field for these points, as noted above, vanishes. So, $V(x,x,x)=x$, and similarly for other $7$ choices of signs. If we treat $\phi_{\mathbb{O}}$ as the superflow over $\mathbb{C}$ and $\mathbb{O}$ as a subgroup of $U(3)$ (which is not the topic of the current paper), the situation is much more interesting.
\section{The case $\xi=\frac{1}{2}$}
\label{tarpp}
Thus, before presenting explicit formulas for the superflow $\phi_{\mathbb{O}}$, we will derive formulas in the exceptional case $\xi=\frac{1}{2}$, where the situation is boundary and non-trivial. 
\begin{thm}
\label{thm-spec}
Let $x,y,z\in\mathbb{R}$ satisfy the condition
\begin{eqnarray*}
\frac{(x^2+y^2+z^2)^2}{x^4+y^4+z^4}=2.
\end{eqnarray*}
Assume $x,y,z\geq 0$, $x=y+z$; see (\ref{iden}). So, out of $24$ possible orbits, we choose one of them. Let us define the function
\begin{eqnarray*}
J(x,y,z)=\frac{\big{(}\sqrt{3}x+iy-iz\big{)}^3-2\sqrt{2}i(x^2+y^2+z^2)^{3/2}
\tanh\Big{(}\frac{\sqrt{x^2+y^2+z^2}}{2\sqrt{2}}\Big{)}}{2\sqrt{2}(x^2+y^2+z^2)^{3/2}+i\big{(}\sqrt{3}x+iy-iz\big{)}^3\tanh\Big{(}\frac{\sqrt{x^2+y^2+z^2}}{2\sqrt{2}}\Big{)}},
\end{eqnarray*}
where the square root is assumed to be non-negative. Then $|J(x,y,z)|=1$, and
\begin{eqnarray}
V(x,y,z)=\Big{(}J^{1/3}+J^{-1/3}\Big{)}\frac{\sqrt{x^2+y^2+z^2}}{\sqrt{6}},\label{v-spec}
\end{eqnarray}
where we choose $|\arg(J^{1/3})|\leq \frac{\pi}{3}$. Moreover, the function $J(x,y,z)$ satisfies the following symmetry properties:
\begin{eqnarray*}
J(-x,y,z)&=&-\frac{1}{J(x,y,z)},\\ 
J(x,-y,z)=J(x,y,z),\quad J(x,y,-z)&=&J(x,y,z), \quad J(x,y,z)=J(x,z,y).
\end{eqnarray*}
\end{thm}
Note that (this applies to all our results), and this was mentioned after Theorem \ref{thm-s4}, though we give only the formula for $\phi^{t}(\m{x})$ for one specific value $t=1$, but since the conditions on $x,y,z$ are always homogeneous, the obtained formulas are also valid for a point $(xt,yt,zt)$, so this describes the full orbit $\phi^{t}(x,y,z)$, $t\in\mathbb{R}$.  
\begin{Note}
\label{note1}
The presented formulas are valid for other orbits as well, only one should be clear which values for the radicals are being used. For example, in the proof we use the identity $q=\sqrt{2y^2+2z^2-x^2}=y-z$, since this, not $z-y$, is the correct choice, as calculations show. The point which lies in the middle of our particular orbit is $(x,y,z)=\Big{(}t\sqrt{\frac{2}{3}},t\sqrt{\frac{1}{6}}, t\sqrt{\frac{1}{6}}\Big{)}$, and is exactly the point where the expression $q$ vanishes and ramifies. Note also that the symmetry properties for $J(x,y,z)$ as presented above do not pass automatically to $V(x,y,z)$ exactly due to this question of ramification. Correct symmetries, valid in all $\mathbb{R}^{3}$, are given by (\ref{v-inv}).
\end{Note}
\begin{proof}
Let $x,y,z\in\mathbb{R}$ be such that
\begin{eqnarray*}
x^2+y^2+z^2=\v^2,\quad x^4+y^4+z^4=\frac{1}{2}\v^4,\quad \v\neq 0.
\end{eqnarray*}
Now, (\ref{antroji}) implies  (recall that $(X,Y)=(P,P')$)
\begin{eqnarray}
\Big{(}\frac{\d P}{\d t}\Big{)}^2=P^2(2P-1)^2(2-3P).
\label{oho}
\end{eqnarray}
This can be rewritten as 
\begin{eqnarray*}
\frac{\d P}{P(2P-1)\sqrt{2-3P}}=\d t.
\end{eqnarray*}
After integration, we obtain
\begin{eqnarray*}
\frac{2}{27}\frac{2+\sqrt{4-6P}-6P\sqrt{4-6P}}{P(2P-1)^2}-1=e^{t\sqrt{2}}.
\end{eqnarray*}
This can be rewritten as
\begin{eqnarray}
27(e^{t\sqrt{2}}+1)P(2P-1)^2-4=2\sqrt{4-6P}(1-6P).
\label{pqr}
\end{eqnarray}
Now, put $e^{t\sqrt{2}}=W(t)=W$. Squaring the equation (\ref{pqr}), and afterwards factoring, we obtain:
\begin{eqnarray*}
27P(2P-1)^2\Big{(}27(W+1)^2P(2P-1)^2-8W\Big{)}=0.
\end{eqnarray*}
Since we want $P=p^2$ to be non-constant function, we should have
\begin{eqnarray}
P(2P-1)^2=\frac{8W}{27(W+1)^2}.
\label{lygtis}
\end{eqnarray} 
Now, let $P,Q,R$ be the three roots of the above, as functions in $W$. Vieta's formulas imply $P+Q+R=1$, and $P^2+Q^2+R^2=(P+Q+R)^2-2(PQ+PR+QR)=\frac{1}{2}$. Since, as before, from symmetry considerations we get that the functions $q^2$ and $r^2$ must satisfy the same equation (\ref{lygtis}), we can take $P=p^2$, $Q=q^2$, $R=r^2$. This shows that the second line of the system (\ref{sys2}) does hold, and we will double-check that the first identity of (\ref{sys2}) holds, too. In terms of $P,Q,R$, this first identity of (\ref{sys2}) reads as 
\begin{eqnarray*}
(2p'p)^2=4(pqr)^2(r^2-q^2)^2\Rightarrow (P')^2=4PQR(Q-R)^2.
\end{eqnarray*}
First, $Q+R=1-P$, $\frac{1}{4}=PQ+PR+QR=P(Q+R)+QR=P(1-P)+QR$. So, 
\begin{eqnarray*}
(Q-R)^2=(Q+R)^2-4QR=(1-P)^2-(4P^2-4P+1)=2P-3P^2.
\end{eqnarray*}
Further, from Vieta's formulas,
\begin{eqnarray*}
4PQR=\frac{8W}{27(W+1)^2}.
\end{eqnarray*}
Thus, provided (\ref{lygtis}) holds, we therefore need to prove that
\begin{eqnarray*}
(P')^2=\frac{8W(2P-3P^2)}{27(W+1)^2},
\end{eqnarray*}
which, minding (\ref{lygtis}), is just a restatement of (\ref{oho}). \\

Let for simplicity $W=w^2$. So, $w=e^{\frac{t\sqrt{2}}{2}}$. Then the discriminant of the cubic polynomial (\ref{lygtis}), as a polynomial in $P$, is equal to 
\begin{eqnarray*}
\frac{256}{27}\frac{w^2(w^2-1)^2}{(w^2+1)^4}.
\end{eqnarray*}
Thus, solving (\ref{lygtis}) using Cardano's formulas, we obtain
\begin{eqnarray*}
P=\frac{1}{6}\Delta^{1/3}+\frac{1}{6}\Delta^{-1/3}+\frac{1}{3},\quad
\Delta=-\frac{(w-i)^2}{(w+i)^2}.
\end{eqnarray*}
($R$ and $Q$ are obtained using two other values for the cubic root, we will soon specify precisely). Now, 
\begin{eqnarray*}
i\frac{w-i}{w+i}&=&i\frac{e^{\frac{t\sqrt{2}}{2}}-i}{e^{\frac{t\sqrt{2}}{2}}+i}=i\frac{e^{\frac{t\sqrt{2}}{2}-\frac{\pi i}{2}}-1}{e^{\frac{t\sqrt{2}}{2}-\frac{\pi i}{2}}+1}\\
&=&i\frac{e^{\frac{t\sqrt{2}}{4}-\frac{\pi i}{4}}-e^{-\frac{t\sqrt{2}}{4}+\frac{\pi i}{4}}}{e^{\frac{t\sqrt{2}}{4}-\frac{\pi i}{4}}+e^{-\frac{t\sqrt{2}}{4}+\frac{\pi i}{4}}}=
\tan\Big{(}\frac{t\sqrt{2}}{4}i+\frac{\pi}{4}\Big{)}\\
&:=&A.
\end{eqnarray*}
Note that $A^2=\Delta$. Since $w\in\mathbb{R}$, this gives 
$|A|=1$. So, we obtain 
\begin{eqnarray*}
p^2=P=\frac{1}{6}A^{2/3}+\frac{1}{6}A^{-2/3}+\frac{1}{3}=\frac{1}{6}(A^{1/3}+A^{-1/3})^2.
\end{eqnarray*}
This finally gives
\begin{eqnarray}
p=\frac{1}{\sqrt{6}}\tan^{1/3}\Big{(}\frac{t\sqrt{2}}{4}i+\frac{\pi}{4}\Big{)}&+&
\frac{1}{\sqrt{6}}\tan^{-1/3}\Big{(}\frac{t\sqrt{2}}{4}i+\frac{\pi}{4}\Big{)},\label{func-p}\\ 
p(t)\in\mathbb{R},&& |p(t)|\leq \frac{\sqrt{2}}{\sqrt{3}}.\nonumber
\end{eqnarray}
The other two functions $q$ and $r$ are obtained using the same formula, only using two other values for the cubic root. To satisfy the first three identities of the system (\ref{sys2}), we should state explicitly which branches of the cubic root are being used. Let us define the function $f(t)$ in the in neighbourhood of $t=0$ by  
\begin{eqnarray*}
f(t)=\tan^{1/3}\Big{(}\frac{t\sqrt{2}}{4}i+\frac{\pi}{4}\Big{)},\quad |f(t)|=f(0)=1.
\end{eqnarray*}
The Taylor series for the function $f(t)$ at $t=0$ then is given by 
\begin{eqnarray*}
f(t)&=&1+\frac{i\sqrt{2}}{6}t-\frac{1}{36}t^2-\frac{5i\sqrt{2}}{324}t^3+
\frac{37}{7776}t^4+\frac{31 i\sqrt{2}}{14580}t^5\\
&-&\frac{3421}{4199040}t^6-\frac{737 i\sqrt{2}}{2204496}t^7+\frac{605737}{4232632320}t^8+\cdots.
\end{eqnarray*}
Computations show that the correct choice of cubic roots to satisfy the first line of (\ref{sys2}) is the following (recall that $p=q+r$):
\begin{eqnarray*}
p=\frac{f+f^{-1}}{\sqrt{6}},\quad q=-\frac{\omega f+\omega^{2}f^{-1}}{\sqrt{6}},\quad r=-\frac{\omega^{2}f+\omega f^{-1}}{\sqrt{6}},\quad \omega=-\frac{1}{2}+\frac{\sqrt{3}}{2}i.
\end{eqnarray*}
Note also that $f+f^{-1}$ is an even function with rational Taylor coefficients:
\begin{eqnarray*}
f+f^{-1}&=&2-\frac{1}{18}t^2+\frac{37}{3888}t^4-\frac{3421}{2099520}t^6\\
&+&\frac{605737}{2116316160}t^8-\frac{176399641}{3428432179200}t^{10}+\cdots.
\end{eqnarray*}
Indeed, the evenness of $f+f^{-1}$ is clear from the fact that $\tan(-x+\frac{\pi}{4})=\tan^{-1}(x+\frac{\pi}{4})$. Thus $f(-t)=f^{-1}(t)$, and this gives 
\begin{eqnarray*}
p(-t)=p(t),\quad q(-t)=r(t),\quad r(-t)=q(t).
\end{eqnarray*}
The inverse of (\ref{func-p}) is given by
\begin{eqnarray}
\tan\Big{(}\frac{t\sqrt{2}}{4}i+\frac{\pi}{4}\Big{)}=
\frac{\Big{(}\sqrt{3}p(t)+i\sqrt{2-3p^2(t)}\Big{)}^3}{2\sqrt{2}}.
\label{inv-tan}
\end{eqnarray} 
\section{Explicit formulas} We are now ready to integrate the vector field $\m{C}(\m{x})$ explicitly in the special case described by the condition in Theorem \ref{thm-spec}. As we know in general, the solution to the PDE (\ref{pde}) with the first boundary condition (\ref{boundr}) is given by
\begin{eqnarray*}
V\big{(}p(u)\v,q(u)\v,r(u)\v\big{)}=p(u-\v)\v,\quad x=p(u)\v,\quad y=q(u)\v,\quad z=r(u)\v. 
\end{eqnarray*}
Note that $x^2+y^2+z^2=p^2(u)\v^2+q^2(u)\v^2+r^2(u)\v^2=\v^2$. 
Now, using addition formula for the tangent function, we get
\begin{eqnarray*}
\tan\Big{(}\frac{(u-\v)\sqrt{2}}{4}i+\frac{\pi}{4}\Big{)}&=&\tan\Big{(}\frac{u\sqrt{2}}{4}i+\frac{\pi}{4}-\frac{\v\sqrt{2}}{4}i\Big{)}\\
&=&
\frac{\tan\Big{(}\frac{u\sqrt{2}}{4}i+\frac{\pi}{4}\Big{)}-\tan\Big{(}\frac{\v\sqrt{2}}{4}i\Big{)}}{1+\tan\Big{(}\frac{u\sqrt{2}}{4}i+\frac{\pi}{4}\Big{)}\tan\Big{(}\frac{\v\sqrt{2}}{4}i\Big{)}}.
\end{eqnarray*}
This, using (\ref{inv-tan}) and the identity $p(u)=x\v^{-1}$, can be continued (the argument of the function $p$ is assumed to be $u$):
\begin{eqnarray*}
\ldots&=&\frac{\big{(}\sqrt{3}p+i\sqrt{2-3p^2}\big{)}^3-2\sqrt{2}i\tanh\Big{(}\frac{\v}{2\sqrt{2}}\Big{)}}{2\sqrt{2}+i\big{(}\sqrt{3}p+i\sqrt{2-3p^2}\big{)}^3\tanh\Big{(}\frac{\v}{2\sqrt{2}}\Big{)}}\\
&=&\frac{\big{(}\sqrt{3}x\v^{-1}+i\sqrt{2-3x^2\v^{-2}}\big{)}^3-2\sqrt{2}i\tanh\Big{(}\frac{\v}{2\sqrt{2}}\Big{)}}{2\sqrt{2}+i\big{(}\sqrt{3}x\v^{-1}+i\sqrt{2-3x^2\v^{-2}}\big{)}^3\tanh\Big{(}\frac{\v}{2\sqrt{2}}\Big{)}}\cdot\frac{\v^3}{\v^{3}}\\
&=&\frac{\big{(}\sqrt{3}x+i\sqrt{2y^2+2z^2-x^2}\big{)}^3-2\sqrt{2}i(x^2+y^2+z^2)^{3/2}
\tanh\Big{(}\frac{\sqrt{x^2+y^2+z^2}}{2\sqrt{2}}\Big{)}}{2\sqrt{2}(x^2+y^2+z^2)^{3/2}+i\big{(}\sqrt{3}x+i\sqrt{2y^2+2z^2-x^2}\big{)}^3\tanh\Big{(}\frac{\sqrt{x^2+y^2+z^2}}{2\sqrt{2}}\Big{)}}\\
&=&\frac{\big{(}\sqrt{3}x+iy-iz\big{)}^3-2\sqrt{2}i(x^2+y^2+z^2)^{3/2}
\tanh\Big{(}\frac{\sqrt{x^2+y^2+z^2}}{2\sqrt{2}}\Big{)}}{2\sqrt{2}(x^2+y^2+z^2)^{3/2}+i\big{(}\sqrt{3}x+iy-iz\big{)}^3\tanh\Big{(}\frac{\sqrt{x^2+y^2+z^2}}{2\sqrt{2}}\Big{)}}\\
&:=&J(x,y,z).
\end{eqnarray*}
Here we used the identity $2\v^2-3x^2=2y^2+2z^2-x^2=2y^2+2z^2-(y+z)^2=(y-z)^2$ (recall that $x=y+z$). As calculations show, $\sqrt{2y^2+2z^2-x^2}=y-z$ is the correct choice of the square root for our particular orbit. Since $V(x,y,z)=p(u-\v)\v=p(u-\v)\sqrt{x^2+y^2+z^2}$, this gives the main formula of Theorem \ref{thm-spec}. The symmetry we are particularly interested in arises from the fact that
\begin{eqnarray*}
\big{(}\sqrt{3}x+iy-iz\big{)}\cdot
\big{(}\sqrt{3}x-iy+iz\big{)}=2(x^2+y^2+z^2).
\end{eqnarray*}
Now, put $x\mapsto(-x)$ in the above formula for $J(x,y,z)$, and then multiply the numerator and the denominator by $(\sqrt{3}x+iy-iz)^3$. After some computation, we get the first symmetry property of the proposition, and the other three are obvious.   
\end{proof}
In particular, for $(x,y,z)=(xt,yt, zt)$, we get
\begin{eqnarray*}
J(xt,yt,zt)=\frac{\big{(}\sqrt{3}x+iy-iz\big{)}^3-2\sqrt{2}i(x^2+y^2+z^2)^{3/2}
\tanh\Big{(}t\frac{\sqrt{x^2+y^2+z^2}}{2\sqrt{2}}\Big{)}}{2\sqrt{2}(x^2+y^2+z^2)^{3/2}+i\big{(}3\sqrt{3}x+iy-iz\big{)}^3\tanh\Big{(}t\frac{\sqrt{x^2+y^2+z^2}}{2\sqrt{2}}\Big{)}}.
\end{eqnarray*}
At $t=0$, this gives the value
\begin{eqnarray*}
\lim\limits_{t\rightarrow 0}J(xt,yt,zt)&:=&\m{v}=\frac{\big{(}\sqrt{3}x+iy-iz\big{)}^3}{2\sqrt{2}(x^2+y^2+z^2)^{3/2}}\Longrightarrow\\
\m{v}^{1/3}&=&\frac{\sqrt{3}x+iy-iz}{\sqrt{2}\sqrt{x^2+y^2+z^2}},\quad
\m{v}^{-1/3}=\frac{\sqrt{3}x-iy+iz}{\sqrt{2}\sqrt{x^2+y^2+z^2}}.
\end{eqnarray*}
So, 
\begin{eqnarray*}
\lim\limits_{t\rightarrow 0}\frac{V(xt,yt,zt)}{t}=\frac{(\m{v}^{1/3}+\m{v}^{-1/3})\sqrt{x^2+y^2+z^2}}{\sqrt{6}}=x.
\end{eqnarray*}
Thus, the first boundary condition (\ref{boundr}) is indeed satisfied!
\begin{Example}
\label{ex8}
 We know that to derive the Taylor series for the flow we need no explicit formulas, to know the vector field is enough; we should use the recurrence and the series (\ref{expl-taylor}). In particular, we get: 
\begin{eqnarray}
V(3t,2t,t)&=&3t+\frac{3}{7}t^2-\frac{51}{98}t^3-\frac{10}{7}t^4-\frac{671}{2744}t^5\nonumber\\
&+&\frac{2669}{980}t^6+\frac{12969121}{4033680}t^7-\frac{49074611}{19765032}t^8+
\cdots.
\label{v-specc}
\end{eqnarray}
This is exactly the special case described by Theorem \ref{thm-spec}, since $3t=2t+t$.\\

On the other hand, using the explicit formula, we get
\begin{eqnarray*}
J(3t,2t,t)&=&\frac{(3\sqrt{3}t+it)^3-2\sqrt{2}i(14)^{3/2}t^3
\tanh\Big{(}\frac{t\sqrt{7}}{2}\Big{)}}{2\sqrt{2}(14)^{3/2}t^3+i(3\sqrt{3}t+it)^3\tanh\Big{(}\frac{t\sqrt{7}}{2}\Big{)}},\\
&=&\frac{(3\sqrt{3}+i)^3}{8\cdot 7^{3/2}}\cdot
\frac{1-\frac{10+9\sqrt{3}i}{7\sqrt{7}}
\tanh\Big{(}\frac{t\sqrt{7}}{2}\Big{)}}{1-\frac{10-9\sqrt{3}i}{7\sqrt{7}}\tanh\Big{(}\frac{t\sqrt{7}}{2}\Big{)}}\\
&:=&\frac{(3\sqrt{3}+i)^3}{8\cdot 7^{3/2}}\cdot A(t).
\end{eqnarray*}
We used the identity $(9\sqrt{3}+10i)(9\sqrt{3}-10i)=7^3$. Note that $A(0)=1$, so MAPLE can easily handle Taylor series of $A^{\pm 1/3}$. Now, plug this into the formula of Theorem \ref{thm-spec}. That is, 
\begin{eqnarray*}
V(3t,2t,t)&=&\Big{(}J^{1/3}(3t,2t,t)+J^{-1/3}(3t,2t,t)\Big{)}\cdot\frac{\sqrt{7}}{\sqrt{3}}\cdot t\\
&=&A^{1/3}(3t,2t,t)\cdot\Big{(}\frac{3}{2}+\frac{i\sqrt{3}}{6}\Big{)}+A^{-1/3}(3t,2t,t)\cdot\Big{(}\frac{3}{2}-\frac{i\sqrt{3}}{6}\Big{)}.
\end{eqnarray*}
Expand this in Taylor series with the help of MAPLE, and this yields exactly the series (\ref{v-specc}). So we can be double-sure that formula in Theorem \ref{thm-spec} is correct.\\

On the other hand, if we use (\ref{expl-taylor}) for the point $(x,y,z)=(3t,t,2t)$, we get
\begin{eqnarray*}
V(3t,t,2t)&=&3t-\frac{3}{7}t^2-\frac{51}{98}t^3+\frac{10}{7}t^4-\frac{671}{2744}t^5\\
&-&\frac{2669}{980}t^6+\frac{12969121}{4033680}t^7+\frac{49074611}{19765032}t^8+\cdots.
\end{eqnarray*}
This is exactly what the series (\ref{v-specc}) and the last identity in (\ref{v-inv}) would give. 
\end{Example}

%% file: sf-chap9.tex
\chapter{The superflow $\phi_{\mathbb{O}}$. Non-singular orbit}
\label{partiv}

\begin{figure}
\includegraphics[scale=0.5]{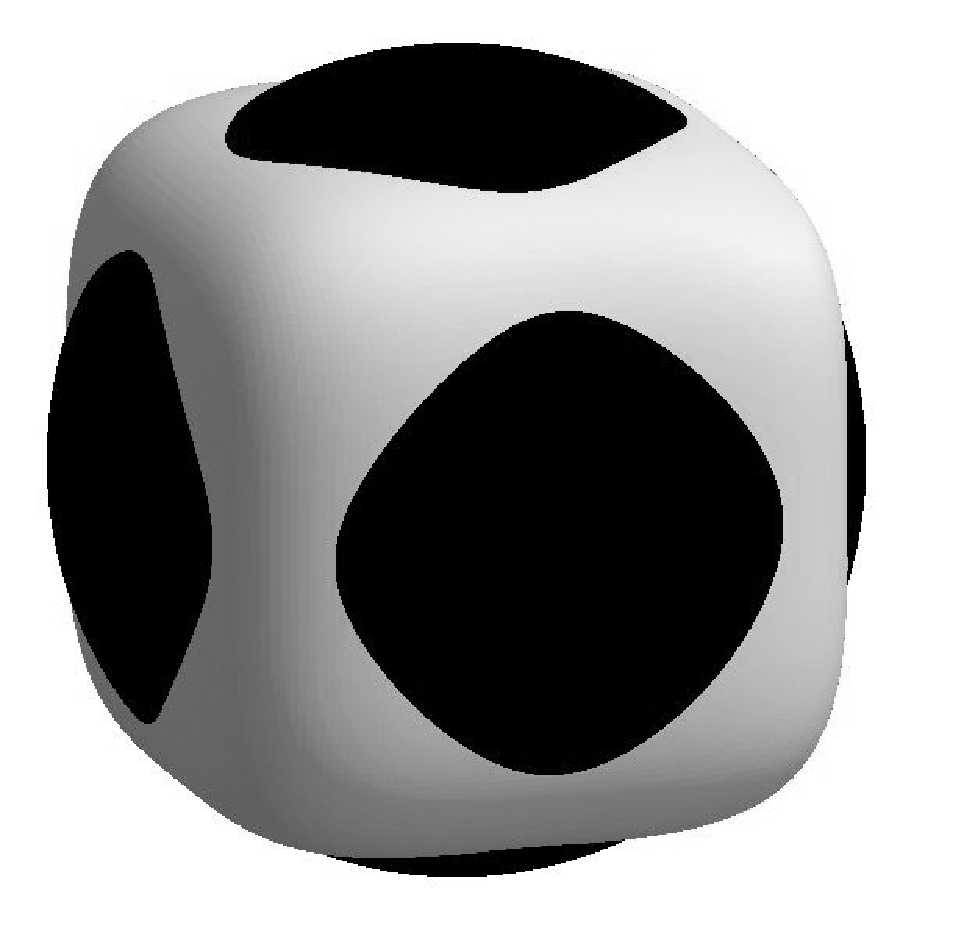}
\caption{The intersection of the surfaces $x^2+y^2+z^2=1$ (black) and $x^4+y^4+z^4=\frac{5}{9}$ (grey). This intersection consists of $6$ different orbits, and the flow $\phi_{\mathbb{O}}$ on these orbits is described in terms of Weierstrass elliptic function with square period lattice.}
\label{figure6}
\end{figure}

\begin{figure}
\includegraphics[scale=0.5]{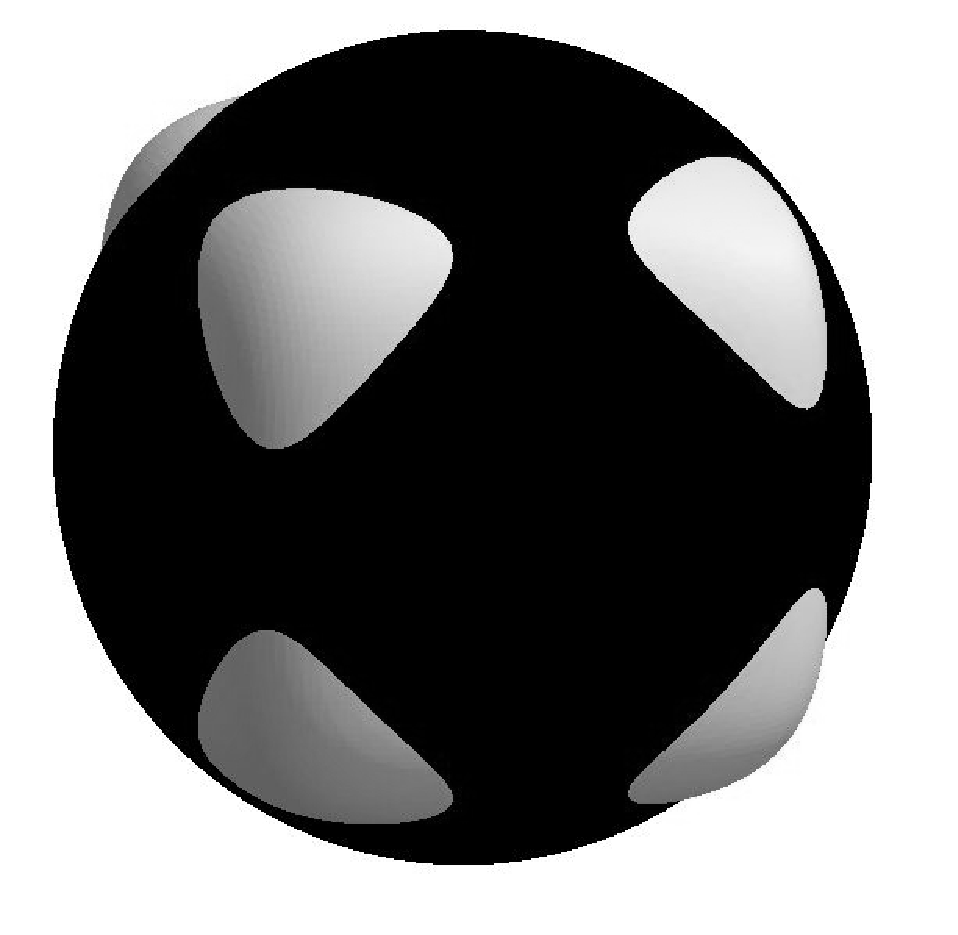}
\caption{The intersection of the surfaces $x^2+y^2+z^2=1$ (black) and $x^4+y^4+z^4=0.4535_{+}$ (grey, see a remark just after Note \ref{papr}). This intersection consists of $8$ different orbits, and the flow $\phi_{\mathbb{O}}$ on these orbits is described again in terms of Weierstrass elliptic function with square period lattice.}
\label{figure7}
\end{figure}

\section{Formulas in a generic case} 
\label{form-generic}
Now, let $\frac{1}{3}<\xi<1$, $\xi\neq\frac{1}{2}$.\\

Let, again, $X$ be any of the functions $P=p^2,Q=q^2$ or $R=r^2$, and let us define
\begin{eqnarray}
-27X^3+27X^2-\frac{27(1-\xi)}{2}X=\Upsilon(t).
\label{cub-up}
\end{eqnarray}
This choice is motivated by the fact that if $\Upsilon$ is fixed and $P,Q,R$ are three solutions of this cubic equation, then $P+Q+R=1$, $P^2+Q^2+R^2=\xi$, exactly as is needed; see (\ref{sys2}). The normalising constant $-27$ is just for convenience.  We will see that $\Upsilon$ is in general an elliptic function. So, this is indeed very fascinating that the superflow $\phi_{\mathbb{O}}$ can be described in terms of Weierstrass elliptic functions, though in a rather complicated way. We have:
\begin{eqnarray*}
\Upsilon'=\Big{(}-81X^2+54X-\frac{27(1-\xi)}{2}\Big{)}X'.
\end{eqnarray*}
Thus, using Proposition \ref{prop8}, we imply that
\begin{eqnarray*}
\Upsilon'^{2}&=&\Big{(}-81X^2+54X-\frac{27(1-\xi)}{2}\Big{)}^{2}X'^2\\
&=&\Big{(}-81X^2+54X-\frac{27(1-\xi)}{2}\Big{)}^{2}\\
&\times& 2X\cdot(1-\xi-2X+2X^2)(2\xi-1+2X-3X^2).
\end{eqnarray*}
But now, using (\ref{cub-up}), we verify by hand (or MAPLE) that the last is $9$th degree polynomial in $X$ and is exactly equal to
\begin{eqnarray}
4\Upsilon^3+(20-36\xi)\Upsilon^2-27(2\xi-1)(\xi-1)^2\Upsilon.
\label{upsilon}
\end{eqnarray}
The discriminant of this cubic polynomial is equal to $23328
(2\xi-1)^2(\xi-1)^4(3\xi-1)^3$. So, $\Upsilon$ is an elliptic function. Moreover, since the discriminant is positive, all three roots are real, and thus the period lattice is rectangular, having primitive periods $2\omega_{1}\in\mathbb{R}$ and $2i\omega_{2}\in i\mathbb{R}$.  If we denote $\widehat{\Upsilon}(t)=\Upsilon(t)+3\xi-\frac{5}{3}$ (thus eliminating the coefficient at $\Upsilon^2$), this leads to a differential equation of exactly Weierstrass type. Since it would be a nuisance to carry on all calculations with an unspecified $\xi$, we choose a particularly elegant case $\xi=\frac{5}{9}$, where the period lattice is a square lattice. \\

\begin{Note}
\label{papr}
 If $\xi=\frac{1}{2}$, from the previous Section and (\ref{lygtis}) we know that in this case $\Upsilon=-\frac{2W}{(W+1)^2}=-(2\cosh^2(\frac{t}{\sqrt{2}}))^{-1}$ (recall different normalizing constants in (\ref{lygtis}) and (\ref{cub-up})), where $W=e^{t\sqrt{2}}$. So, the differential equation $\Upsilon'^2=4\Upsilon^3+2\Upsilon^2$ holds.
\end{Note}
 For general $\xi$, the free term of the equation for $\widehat{\Upsilon}$ vanishes 
(\emph{a posteriori}, the period lattice of the elliptic function then is a square lattice), if  
\begin{eqnarray*}
(9\xi-5)(486\xi^3-567\xi^2+252\xi-43)=0.
\end{eqnarray*}
(See (\ref{g3}) below for $a=1-\xi$, $b=-2$). So, apart from the rational solution, we have one other value
\begin{eqnarray*}
\xi=\frac{1}{18}\sqrt[3]{16\sqrt{2}+13}-\frac{1}{18}\sqrt[3]{16\sqrt{2}-13}+\frac{7}{18}=0.4535087845_{+}.
\end{eqnarray*} 
The orbits for this case are shown in Figure \ref{figure7}.
\begin{Note}
\label{note-hyper}
Consider the hyper-elliptic curve
\begin{eqnarray}
Y^2=2X(a+bX+2X^2)(c+dX-3X^2).
\label{y-eq}
\end{eqnarray}
Here the constants $a,b,c,d\in\mathbb{R}$ are such that the discriminant of the $5$th degree polynomial does not vanish. The constants $2,-3$ at $X^2$ are just for convenience to keep track of our main example $a=1-\xi$, $b=-2$, $c=2\xi-1$, $d=2$ (see Proposition \ref{prop8}).\\

Suppose, the pair $(J(t),J'(t))$ parametrizes the curve (\ref{y-eq}) locally at $t=0$. We want to find a condition on $(a,b,c,d)$ which leads to a similar reduction to elliptic function as in our case of $P,Q,R$ and $\Upsilon$ in the beginning of this section. Let therefore
\begin{eqnarray*}
-27X^3+uX^2+wX=\psi.
\end{eqnarray*}
Again, the normalizing constant $``-27"$ is just for convenience in order to get a differential equation of exactly Weierstrass type, and $u,w$ are to be determined. We have:
\begin{eqnarray*}
\psi'=(-81X^2+2uX+w)X'.
\end{eqnarray*}
And so, for $X=J$,
\begin{eqnarray*}
(\psi')^{2}&=&(-81X^2+2uX+w)^2X'^2\\
&=&2(-81X^2+2uX+w)^2X(a+bX+2X^2)(c+dX-3X^2).
\end{eqnarray*}
We want the last expression to be equal to $4\psi^{3}+p\psi^{2}+q\psi$ for certain $p,q\in\mathbb{R}$ (note that both expressions have no term $X^{0}$). Indeed, both of them are $9$th degree polynomials in $X$, and this puts many restrictions on $a,b,c,d,u,w$. Calculations with MAPLE lead to two cases expressible in rational terms. Note, however, that the involution $(a,b,c,d)\mapsto(-\frac{2c}{3},-\frac{2d}{3},-\frac{3a}{2},-\frac{3b}{2})$ leaves the equation (\ref{y-eq}) intact, so in fact only one of these solutions matters. We formulate the findings as follows.
\begin{prop}Let $a,b\in\mathbb{R}$. Consider the polynomial
\begin{eqnarray*}
2X(a+bX+2X^2)\Big{(}\frac{b^2}{4}-2a-bX-3X^2\Big{)}=f_{a,b}(X).
\end{eqnarray*}
Suppose that $\mathrm{discrim}(f_{a,b}(X))=\frac{1}{4}a^6(b^2-8a)^3(b^2-6a)\neq 0$. Suppose, a pair of functions $(J(t),J'(t))=(X,Y)$ parametrizes the genus $2$ curve $Y^2=f_{a,b}(X)$ locally at $t=0$, $J(0)\neq 0$. Then, if we define
\begin{eqnarray*}
\psi(t)=-27J^3-\frac{27b}{2}J^2-\frac{27a}{2}J,
\end{eqnarray*}
then $\psi(t)=\wp(t+\alpha)-\frac{b^3}{6}+\frac{3ab}{2}$, for a certain $\alpha\in\mathbb{C}$, where $\wp$ is the Weierstarss elliptic function satisfying
\begin{eqnarray*}
\wp'^2=4\wp^3-g_{2}\wp-g_{3}.
\end{eqnarray*} 
Here $g_{2}$ and $g_{3}$ are explicit polynomials in $a,b$:
\begin{eqnarray}
g_{2}&=&\frac{1}{3}b^6-6ab^4+\frac{135}{4}a^2b^2-54a^3,\nonumber\\
g_{3}&=&-\frac{1}{27}b^9+ab^7-\frac{81}{8}a^2b^5+\frac{369}{8}a^3b^3-81a^4b,\label{g3}\\
g_{2}^{3}-27g_{3}^{2}&=&\frac{729}{64}a^4(b^2-8a)^2(b^2-6a)^3.\nonumber
\end{eqnarray}

\end{prop}
Thus, we have discovered an explicit example of hyper-elliptic functions (or integrals) that can be transformed into elliptic functions (or integrals). This topic has a history and stems from the works of A.-M. Legendre in 1825 \cite{legendre,scarpello}.
\end{Note}
\begin{Note} We can explain the situation in the previous Note by a much simpler example. Let $J=\sqrt{\sin x}$. Then $J'=\frac{\cos x}{2\sqrt{\sin x}}$, and thus
\begin{eqnarray*}
(2JJ')^2+J^4=1.
\end{eqnarray*}
So, the pair $(J,J')=(X,Y)$ parametrizes the genus $1$ singular (as a projective curve) quartic
\begin{eqnarray*}
X^4+4X^2Y^2=1.
\end{eqnarray*}
But if we put $\psi=J^2$, then the pair $(\psi,\psi')$ parametrized the genus $0$ quadratic $X^2+Y^2=1$. 
\end{Note}
 We re-iterate that all three functions $P,Q,R$ generically are described in terms of a certain curve of genus $2$, but all three can be accommodated under a uniform curve with a reduced genus; hence, much simpler. This works a bit like a miracle - in Note \ref{note-hyper} there are $6$ degrees of freedom to equate $9$th degree polynomials, but it works! In Section \ref{hyper-5} we will see the reasons behind this phenomenon.\\
 
Exactly the same occurs for the icosahedral superflow \cite{alkauskas-super2}: three analogous functions $P,Q,R$, each with its derivative, parametrizes a genus $7$ curve, but via a third degree polynomial transformation all three can be accommodated under a uniform curve of genus $3$; for example, the one given by (\ref{curve-ico}) for one distinguished orbit. However, this phenomenon did not happen for the tetrahedral superflow in Section \ref{S4}. The reason for this is that $x^2+y^2+z^2$ is not the first integral for the system (\ref{syys}). In relation to this, a strengthened version of  Problem \ref{prob-vienas} is presented in \cite{alkauskas-super2}. 
\section{The case $\xi=\frac{5}{9}$}In this case $3\xi-\frac{5}{3}=0$, and the equation (\ref{upsilon}) reads as
\begin{eqnarray}
\Upsilon'^2=4\Upsilon^3-\frac{16}{27}\Upsilon.
\label{diff-ups}
\end{eqnarray}
Let $\wp$ be the Weierstrass elliptic function satisfying the equation
\begin{eqnarray}
\wp'^2=4\wp^3-\frac{16}{27}\wp.
\label{weier}
\end{eqnarray}
So, the period lattice of this elliptic function is a square lattice, and $\wp(it)=-\wp(t)$. Let $2\omega,2\omega i$ be its primitive periods. We have
\begin{eqnarray*}
\omega=\int\limits_{\frac{2}{3\sqrt{3}}}^{\infty}\frac{\d x}{\sqrt{4x^3-\frac{16}{27}x}}
=\frac{3^{3/4}}{2^{3/2}}\int\limits_{1}^{\infty}\frac{\d x}{\sqrt{x^3-x}}&=&
\frac{3^{3/4}\Gamma(\frac{1}{4})^2}{8 \sqrt{\pi}}=2.1131881555_{+},\\
\wp(\omega)=\frac{2}{3\sqrt{3}},\quad \wp(\omega+i\omega)&=&0,\quad \wp(i\omega)=-\frac{2}{3\sqrt{3}}.
\end{eqnarray*}
The function $\wp(t)$ has a double pole at $t=0$. The orbits of our superflow are bounded curves, thus we want $\Upsilon$ to be bounded and real. Hence, we put 
\begin{eqnarray}
\Upsilon(t)=\wp(t+i\omega),\quad t\in\mathbb{R}.
\label{ups}
\end{eqnarray}
Note that $\wp'(\omega)=\wp'(i\omega)=\wp'(\omega+i\omega)=0$; so $\omega+i\omega$ is a double zero of the order $2$ elliptic function $\wp(t)$, which therefore does not vanish for $t\neq \omega+i\omega$, modulo periods. In particular, since $\Upsilon(0)<0$, this gives
\begin{eqnarray}
\Upsilon(t)\leq 0 \text{ for }t\in\mathbb{R},\quad 4\Upsilon^{2}-\frac{16}{27}\leq 0.
\label{ineq}
\end{eqnarray}
The second inequality follows from the differential equation and the first inequality, since $\Upsilon'^{2}\geq 0$. The differential equation (\ref{weier}) allows us to derive Taylor series for $\wp$, which is a classical computation:
\begin{eqnarray}
\wp(t)&=&\frac{1}{t^2}+\frac{4}{135}t^2+\frac{16}{54675}t^6+\frac{128}{95954625}t^{10}\nonumber\\
&+&\frac{256}{44043172875}t^{14}+\frac{2048}{89187425071875}t^{18}
\cdots.
\label{wp-taylor}
\end{eqnarray}
Now we will derive the addition formula for $\Upsilon$. The classical addition formula for $\wp$ claims that \cite{ahiezer}
\begin{eqnarray*}
\wp(u+v)=\frac{1}{4}\Bigg{[}\frac{\wp'(u)-\wp'(v)}{\wp(u)-\wp(v)}\Bigg{]}^2-\wp(u)-\wp(v).
\end{eqnarray*}
In particular, plugging $u\mapsto u+i\omega$, $v\mapsto -v$, and using the fact that $\wp$ is even, $\wp'$ is odd, we have 
\begin{eqnarray}
\Upsilon(u-v)=\frac{1}{4}\Bigg{[}\frac{\Upsilon'(u)+\wp'(v)}{\Upsilon(u)-\wp(v)}\Bigg{]}^2-\Upsilon(u)-\wp(v).
\label{addition}
\end{eqnarray}
Now, the pair of equations $x^2+y^2+z^2=1$ and $x^4+y^4+z^4=\frac{5}{9}$ define six connected components, as described in Section \ref{elementary}, case v); see Figure \ref{figure6}. The next lemma will come out as the side result of our proof, but we write it explicitly to avoid questions of choosing branches of radicals. This fact is easily verified directly by solving a quadratic equation.
\begin{lem}Let $x^2+y^2+z^2=1$ and $x^4+y^4+z^4=\frac{5}{9}$. From six connected components of this set, choose two which round the points either $(1,0,0)$ or $(-1,0,0)$, as described in Section \ref{elementary}, case v).  Then $\frac{2}{3}\leq x^2\leq\frac{3+2\sqrt{3}}{9}$, and the bounds are exact. For other $4$ components, $x^2\leq\frac{1}{3}$, and the bound is exact.
\label{lemma1}
\end{lem}
\begin{thm}
\label{thm4}
Let $x,y,z\in\mathbb{R}_{+}$, $y\geq z$, satisfy the condition
\begin{eqnarray}
\frac{(x^2+y^2+z^2)^2}{x^4+y^4+z^4}=\frac{9}{5},\quad x^{2}\geq\frac{2}{3}(x^2+y^2+z^2).\label{condi}
\end{eqnarray}
Thus, we choose the orbit which rounds $(1,0,0)$, and $\frac{1}{8}$th of it. Let
\begin{eqnarray*}
T(x,y,z)&=&\frac{1}{4}\Bigg{[}\frac{\sqrt{L(x,y,z)}+\wp'(\sqrt{x^2+y^2+z^2})}{K(x,y,z)-\wp(\sqrt{x^2+y^2+z^2})}\Bigg{]}^2\\
&-&K(x,y,z)-\wp(\sqrt{x^2+y^2+z^2}),
\end{eqnarray*}
where $K$ and $L$ are $0-$homogeneous rational functions, given by (\ref{k-rat}) and (\ref{l-rat}), and $\wp$ is the Weierstrass elliptic function, satisfying (\ref{weier}).  The value $\sqrt{L(x,y,z)}$ is assumed to be non-negative. Let us define
\begin{eqnarray*} 
J(x,y,z)=-108T+12i\sqrt{12-81T^2}.
\end{eqnarray*} 
Then $J(x,y,z)\in\mathbb{C}$, $|J(x,y,z)|=24\sqrt{3}$, and
\begin{eqnarray}
V(x,y,z)&=&\Big{(}\frac{J(x,y,z)^{1/3}}{18}+\frac{2}{3J(x,y,z)^{1/3}}+\frac{1}{3}\Big{)}^{1/2}\nonumber\\
&\cdot&\sqrt{x^2+y^2+z^2}\in\mathbb{R}_{+},
\label{form-v}
\end{eqnarray}
where we choose $|\arg(J^{1/3})|\leq \frac{\pi}{3}$.
\end{thm}
Compare with Theorem \ref{thm-spec}, where the presented formulas, due to ramification, are valid for $\frac{1}{48}$th of the complete intersection of two surfaces $\m{S}$ and $\m{T}_{\xi}$, as defined in Section \ref{elementary}. For other parts of the intersection, we should use other values for quadratic and cubic radicals.
\begin{proof}The equation (\ref{cub-up}) in case $\xi=\frac{5}{9}$ factors as
\begin{eqnarray}
3X(3X-1)(3X-2)=-\Upsilon.
\label{factors}
\end{eqnarray}
This solves as
\begin{eqnarray}
P=\frac{\Delta^{1/3}}{18}+\frac{2}{3\Delta^{1/3}}+\frac{1}{3},\quad \Delta=-108\Upsilon+12i\sqrt{12-81\Upsilon^2}.
\label{P-ups}
\end{eqnarray}
$Q$ and $R$ are obtained from the same equality for a different third root of unity. According to (\ref{ineq}), $12-81\Upsilon^2\geq 0$, so we have $|\Delta|=24\sqrt{3}$, $|\Delta|^{1/3}=2\sqrt{3}$. So, $\frac{\Delta^{1/3}}{18}$ and $\frac{2}{3\Delta^{1/3}}$ are conjugate complex numbers. Since $-\Upsilon\geq 0$, for $P$ we choose non-negative value of the square root and $0\leq \mathrm{Arg}(\Delta^{1/3})\leq\frac{\pi}{6}$. Thus,
\begin{eqnarray}
P\in\mathbb{R}_{+},\quad \frac{2}{3}=P(\omega)\leq P\leq P(0)=\frac{3+2\sqrt{3}}{9}.
\label{p-bound}
\end{eqnarray}
 Since $P=p^2$, the side result of our calculations, the identity (\ref{factors}) and the bound (\ref{p-bound}), is Lemma \ref{lemma1}. In the other direction, (\ref{factors}) and Lemma show again that $\Upsilon$ is always non-positive.
\section{Explicit formulas}
Again, the explicit formulas come from the equality
\begin{eqnarray*}
V\big{(}p(u)\v,q(u)\v,r(u)\v\big{)}=p(u-\v)\v,\quad x=p(u)\v,\quad y=q(u)\v,\quad z=r(u)\v. 
\end{eqnarray*}
We have: $\v^2=x^2+y^2+z^2$. The strategy is the same as for all flows: we write $p$ in terms of $\Upsilon$, and then apply the addition formula (\ref{addition}).\\

First, from (\ref{factors}) and the identity $P=p^2$ it follows that ($X=P$)
\begin{eqnarray}
\Upsilon(u)&=&-3p^2(u)(3p^2(u)-1)(3p^2(u)-2)=-3x^2(3x^2-\v^2)(3x^2-2\v^2)\v^{-6}\nonumber\\
&=&-\frac{3x^2(2x^2-y^2-z^2)(x^2-2y^2-2z^2)}{(x^2+y^2+z^2)^3}:=K(x,y,z).\label{k-rat}
\end{eqnarray}
Next, from the differential equation (\ref{diff-ups}),
\begin{eqnarray}
\Upsilon'(u)^2&=&4K^{3}(x,y,z)-\frac{16}{27}K(x,y,z)
=\frac{x^2(2x^2-y^2-z^2)(x^2-2y^2-2z^2)}{9(x^2+y^2+z^2)^9}\nonumber\\
&&\cdot\Big{(}(y^2+z^2+10x^2)^2-108x^4\Big{)}\cdot\Big{(}(2y^2+2z^2-7x^2)^2-27x^4\Big{)}^2\nonumber\\
:&=&L(x,y,z).
\label{l-rat}
\end{eqnarray}
Let us examine $L(x,y,z)$ more carefully. First, factor
\begin{eqnarray*} 
&&(y^2+z^2+10x^2)^2-108x^4\\
&=&\Big{(}
y^2+z^2+(6\sqrt{3}+10)x^2\Big{)}\cdot\Big{(}y^2+z^2-(6\sqrt{3}-10)x^2\Big{)}.
\end{eqnarray*}
Note that with the condition (\ref{condi}) satisfied,
\begin{eqnarray*}
2x^2-y^2-z^2>0,\quad (7x^2-2z^2-2y^2)^2\geq(6x^2)^2>27x^4.
\end{eqnarray*}
However (see Lemma \ref{lemma1}),
\begin{eqnarray*} x^2-2y^2-2z^2\geq0,\quad 
y^2+z^2-(6\sqrt{3}-10)x^2\geq0,
\end{eqnarray*}  
with equalities achieved at 
\begin{eqnarray*}
(x,y,z)&=&(t\sqrt{2},t,0),\text{ and}\\
(x,y,z)&=&\Big{(}\sqrt{\frac{3+2\sqrt{3}}{9}}t,\sqrt{\frac{3-\sqrt{3}}{9}}t,\sqrt{\frac{3-\sqrt{3}}{9}}t\Big{)},
\end{eqnarray*}
respectively. All these remarks should be taken into account while choosing the correct sign for $\Upsilon'(u)=\pm\sqrt{L(x,y,z)}$ (see Note \ref{note1} in Section \ref{tarpp}).\\
    
Finally, from the addition formula (\ref{addition}), we imply
\begin{eqnarray}
\Upsilon(u-\v)&=&\frac{1}{4}\Bigg{[}\frac{\sqrt{L(x,y,z)}+\wp'(\sqrt{x^2+y^2+z^2})}{K(x,y,z)-\wp(\sqrt{x^2+y^2+z^2})}\Bigg{]}^2\nonumber\\
&-&K(x,y,z)-\wp(\sqrt{x^2+y^2+z^2}).
\label{add-fin}
\end{eqnarray}
This finishes the proof of Theorem \ref{thm4}.
\end{proof}
\begin{Example} Let $(x,y,z)=(t\sqrt{2},t,0)$.
Then the conditions of Theorem \ref{thm4} are satisfied. The series (\ref{expl-taylor}) in this case gives
\begin{eqnarray}
\frac{V(t\sqrt{2},t,0)}{t\sqrt{2}}&=&1+\frac{1}{18}t^2-\frac{13}{648}t^4-\frac{53}{19440}t^6\nonumber\\
&+&\frac{7663}{4199040}t^8+\frac{76183}{377913600}t^{10}-\cdots.
\label{v-ser}
\end{eqnarray}
Further, in this case $K(t\sqrt{2},t,0)=L(t\sqrt{2},t,0)=0$. So,
\begin{eqnarray*}
T(t\sqrt{2},t,0)=\frac{1}{4}\frac{\wp'^2(t\sqrt{3})}{\wp^2(t\sqrt{3})}-\wp(t\sqrt{3})=
-\frac{4}{27\wp(t\sqrt{3})}.
\end{eqnarray*}
Further, the differential equation then gives
\begin{eqnarray*}
\sqrt{12-81T^2}=-\frac{\sqrt{3}\wp'(t\sqrt{3})}{\wp(t\sqrt{3})^{3/2}}.
\end{eqnarray*}
(We choose the ``minus" sign for the right hand side in order it to be positive for $t$ small). Thus, 
\begin{eqnarray}
J(t\sqrt{2},t,0)&=&\frac{16}{\wp(t\sqrt{3})}-i\frac{12\sqrt{3}\wp'(t\sqrt{3})}{\wp(t\sqrt{3})^{3/2}}\nonumber\\
&=&24\sqrt{3}i\Big{(}-\frac{2i}{3\sqrt{3}\wp(t\sqrt{3})}-\frac{\wp'(t\sqrt{3})}{2\wp(t\sqrt{3})^{3/2}}\Big{)}\\
&:=&24\sqrt{3}iA.
\label{ex-j}
\end{eqnarray}
Note that, according to (\ref{wp-taylor}), the Taylor series at $t=0$ of the function $A=A(t)$ stars from $1$.
The equality $|A(t)|=1$ now is clear, because of the differential equation for $\wp$ again.
So, (\ref{form-v}) gives
\begin{eqnarray}
V(t\sqrt{2},t,0)&=&\Big{(}\frac{(3+i\sqrt{3})A^{1/3}}{18}+\frac{(3-i\sqrt{3})}{18A^{1/3}}+\frac{1}{3}\Big{)}^{1/2}\cdot t\sqrt{3}\nonumber\\
&=&\Big{(}\frac{(3+i\sqrt{3})A^{1/3}}{12}+\frac{(3-i\sqrt{3})}{12A^{1/3}}+\frac{1}{2}\Big{)}^{1/2}\cdot t\sqrt{2}.
\label{ex-v}
\end{eqnarray}
Now, plug the series (\ref{wp-taylor}) into the expression for $A$ to obtain
\begin{eqnarray*}\\
A&=&1-\frac{2i\sqrt{3}}{3}t^2-\frac{2}{3}t^4+\frac{8i\sqrt{3}}{45}t^6+\frac{2}{15}t^8\\
&-&\frac{64i\sqrt{3}}{2025}t^{10}-\frac{44}{2025}t^{12}+\frac{128i\sqrt{3}}{26325}t^{14}+\cdots.
\end{eqnarray*}
Lastly, plug this series for $A$ into (\ref{ex-v}). Note that the series in the brackets of (\ref{ex-v}) starts from the term $1$, and MAPLE easily handles this using the binomial formula. We get exactly the series (\ref{v-ser}), so the formula in Theorem \ref{thm4} is verified in this special case. 
\end{Example}

%% file: sf-chap10.tex
\chapter{Hyper-octahedral group in odd dimension $n\geq 3$}
\label{hyper-5}
Let $n\geq 3$ be an odd integer. In this section we (in a much more concise way) investigate the superflow obtained from a symmetry group of an odd-dimensional regular polytope, namely, \emph{hypercube}, defined by
\begin{eqnarray*}
-1\leq x_{i}\leq 1,\quad 1\leq i\leq n.
\end{eqnarray*}
This will generalize all results from the last four Sections, and this will provide more evidence for the main Conjecture in \cite{alkauskas-super2}.
\section{Dimension $n=5$}
 But first, start from a $5$-dimensional cube, the $5$-\emph{cube} \cite{coxeter}. We will see here that the notions of a projective superflow and a polynomial superflow give different answers - while we obtain the polynomial superflow, the degree of its vector field is to high to guarantee the uniqueness of the denominator; hence projective superflow does not exist. \\ 

The $5$-cube has an orientation-preserving symmetry group of order $1920$. Call it $\mathbb{O}_{5}$. This is the so called \emph{orientation-preserving hyperoctahedral group in dimension $5$}. The full hyperoctahedral group $\widehat{\mathbb{O}}_{5}$ of order $3840=120\cdot 32=5!\cdot 2^{5}$ can be constructed as follows. Let  $S_{5}\subset O(5)$ be a group of order $120$, a standard permutation representation of a group $S_{5}$. For example, even permutations $\alpha=(12345)$ and $\beta=(123)$ are represented as matrices, respectively,  
\begin{eqnarray*}
\alpha\mapsto\begin{pmatrix}
0 & 1 & 0 & 0 & 0\\
0 & 0 & 1 & 0 & 0\\
0 & 0 & 0 & 1 & 0\\
0 & 0 & 0 & 0 & 1\\
1 & 0 & 0 & 0 & 0
\end{pmatrix},
\beta\mapsto\begin{pmatrix}
0 & 1 & 0 & 0 & 0\\
0 & 0 & 1 & 0 & 0\\
1 & 0 & 0 & 0 & 0\\
0 & 0 & 0 & 1 & 0\\
0 & 0 & 0 & 0 & 1
\end{pmatrix}.
\end{eqnarray*}
Let $\Delta_{5}\subset O(5)$ is a group $\mathbb{Z}_{2}^{5}$ of order $32$ given by $\mathrm{diag}(\pm 1, \pm 1, \pm 1, \pm 1, \pm 1)$ (signs are independent). Then $\widehat{\mathbb{O}}_{5}$ is the group $\Delta_{5}\rtimes S_{5}$ (semi-direct product), and $\mathbb{O}_{5}$ is a subgroup of index $2$ consisting of all matrices from the latter with determinant $+1$ \cite{coxeter, field}. In fact, adding $\gamma$, given by 
\begin{eqnarray*}
\gamma\mapsto\begin{pmatrix}
0 & 1 & 0 & 0 & 0\\
1 & 0 & 0 & 0 & 0\\
0 & 0 & -1 & 0 & 0\\
0 & 0 & 0 & 1 & 0\\
0 & 0 & 0 & 0 & 1
\end{pmatrix},
\end{eqnarray*}
to $\alpha$ and $\beta$ (we identify permutations with their matrix representations), generate the whole group $\mathbb{O}_{5}$ of order $1920$. Figure \ref{hyperc-5} show the graph of vertices and edges of the $5$-cube. 
\begin{figure}
\includegraphics[scale=0.75]{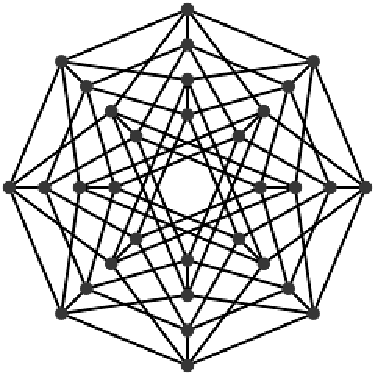}
\caption{$32$ vertices and $80$ edges of a \emph{penteract}, a $5$-cube}
\label{hyperc-5}
\end{figure}

Recall now that $yz(y^2-z^2)$, the numerator for the first coordinate of the octahedral superflow (see Section \ref{sub5.1}), possesses these properties:
\begin{itemize}
\item[1)] it is anti-symmetric in $y,z$,
\item[2)] it is of odd degree in each of $y,z$, and even in $x$ (in this case, of degree $0$).
\end{itemize}
Let $g(x_{1},\ldots,x_{5})$ be homogeneous polynomial of even degree, the first coordinate of the vector field which is invariant under $\mathbb{O}_{5}$. Now, all matrices of the type $\mathrm{diag}(\pm 1,\ldots,\pm 1)$, with even number of $``-1"$s, belong to $\mathbb{O}_{5}$. Thus, for example,
\begin{eqnarray*}
g(x_{1},-x_{2},-x_{3},x_{4},x_{5})=g(x_{1},x_{2},x_{3},x_{4},x_{5})=
-g(-x_{1},-x_{2},x_{3},x_{4},x_{5}).
\end{eqnarray*}
Since the compound degree is even, this shows that 
\begin{eqnarray}
\text{$g$ is even in $x_{1}$, and odd in each $x_{2},\ldots, x_{5}$}.
\label{property-1}
\end{eqnarray} 
Next, all matrices of the same type as $\gamma$, that is, those which have exactly two $``1"$'s and one $"-1"$ on the main diagonal, but permute two remaining variables, belong to $\mathbb{O}_{5}$. Thus, for example,
\begin{eqnarray*}
-g(-x_{1},x_{3},x_{2},x_{4},x_{5})=g(x_{1},x_{2},x_{3},x_{4},x_{5}).
\end{eqnarray*}
Let $g=h_{0}(x_{2},\ldots,x_{5})+x_{1}^{2}h_{1}(x_{2},\ldots,x_{5})+\cdots.$ Then 
\begin{eqnarray}
\text{each $h_{i}$ is antisymmetric with respect to $x_{i}$, $2\leq i\leq 5$}. 
\label{property-2}
\end{eqnarray}
The construction of antisymmetric polynomials is known \cite{kostrikin}: they are given by 
\begin{eqnarray*}
f(x_{2},\ldots,x_{5})\prod\limits_{2\leq i<j\leq 5}(x_{i}-x_{j}),
\end{eqnarray*}
where $f$ is any symmetric polynomial. Therefore, the unique, up to the scalar multiple, polynomial of the smallest degree with the two properties (\ref{property-1}) and (\ref{property-2}) we need, is given by 
\begin{eqnarray*}
h(x_{2},x_{3},x_{4},x_{5})=x_{2}x_{3}x_{4}x_{5}\prod\limits_{2\leq i<j\leq 5}(x_{i}^{2}-x_{j}^{2}).
\end{eqnarray*}
It is of degree $16$ in total and of degree $7$ in each of the variables $x_{i}$ except for $x_{1}$. Hence the solution to our problem is given by the vector field whose first coordinate is $\varpi_{1}(x_{1},x_{2},\ldots, x_{5})=h(x_{2},\ldots,x_{5})$, and other coordinates $\varpi_{j}$, $2\leq j\leq 5$, are obtained cyclically. With these clarifying remarks, we  can pass to polynomial hyperoctahedral superflows in all odd dimensions $n\geq 3$.  
\section{Hyperoctahedral polynomial superflows}
\label{sec9.2}
 Let $n\geq 3$ be an odd integer. The orientation-preserving hyperoctahedral group $\mathbb{O}_{n}\subset SO(n)$ of order $n!\cdot 2^{n-1}$ is given by all matrices $n\times n$ which have exactly one non-zero element, necessarily equal to  $1$ or $-1$, in each row or column, with determinant equal to $+1$. Without the last condition, we get a full hyperoctahedral group $\widehat{\mathbb{O}}_{n}$ of order $n!\cdot 2^{n}$. Then similar calculations as above show that the first coordinate of the unique polynomial homogeneous vector field of the smallest degree, invariant under $\mathbb{O}_{n}$, is given by
\begin{eqnarray*}
\varpi_{1}=\prod\limits_{j=2}^{n}x_{j}\prod\limits_{2\leq i<j\leq n}(x_{i}^{2}-x_{j}^{2}).
\end{eqnarray*}
It is of degree $(n-1)^2$ in total, and of degree $2n-3$ in each of variables $x_{j}$, $2\leq j\leq n$, and does not depend on $x_{1}$. This vector field gives another infinite series of polynomial superflows. The first series of superflows (both polynomial and projective, since there are no denominators) was constructed in (\cite{alkauskas-un}, Section 5.1), and briefly described in the beginning of Section \ref{symm-N}. The symmetry group of the latter is of order $(n+1)!$, as compared to $n!\cdot 2^{n-1}$ in the hyperoctahedral case.\\

 Let us define the vector field $\m{O}_{n}=\varpi_{1}\bl\cdots\bl\varpi_{n}$, whose first coordinate is $\varpi_{1}$, and other coordinates are corresponding cyclic permutations; for example, $\varpi_{2}=\varpi_{1}(x_{2},x_{3},\ldots,x_{n},x_{1})$. The group $\mathbb{O}_{n}$ (as well as $\widehat{\mathbb{O}}_{n}$) has $(n+1)$ invariants: they are given by 
\begin{eqnarray*}
\mathscr{Q}^{s}=\sum\limits_{j=1}^{n}x_{j}^{2s},\quad 1\leq s\leq n,
\end{eqnarray*} 
and one additional invariant of degree $n^2$ (not an invariant of $\widehat{\mathbb{O}}_{n}$), given by
\begin{eqnarray*}
\sum\limits_{\sigma\in S_{n}}\mathrm{sign}(\sigma)x^{1}_{\sigma(1)}\cdot x^{3}_{\sigma(2)}\cdots x^{2n-1}_{\sigma(n)}.
\end{eqnarray*}
See (\ref{inv-specc}) for example in case $n=3$. All the claims follow from writing down a Molien's series \cite{smith,verma}. 
\begin{prop}
\label{inv-fi}
The vector field $\m{O}_{n}$ is solenoidal, the corresponding differential system  has $(n-1)$ independent polynomial first integrals given by $\mathscr{Q}^{s}$, $1\leq s\leq n-1$.
\end{prop}
 Thus, the exception is $\mathscr{Q}^{n}$ which is the invariant of the group but not the  first integral. This is exactly what happened to the octahedral superflow, where orbits are space curves $\{x^2+y^2+z^2=\mathrm{const}.$, $x^4+y^4+z^4=\mathrm{const}.\}$; see Section \ref{orbits-oct}. Also, the last Proposition answers affirmatively to the Problem \ref{prob-two} in this series (infinite number) of cases.  
\begin{proof}Since $\varpi_{j}$ does not depend on $x_{j}$, then $\frac{\p}{\p x_{j}}\varpi_{j}=0$, and the vector field is obviously solenoidal.\\

To prove the second part, let
\begin{eqnarray*}
P(X)=\prod\limits_{j=1}^{n}(X-x_{j}^2),\quad p(X)=\prod\limits_{j=1}^{n}(X-y_{j}). 
\end{eqnarray*}
Then $\varpi_{j}$, $1\leq j\leq n$, can be written as
\begin{eqnarray*}
\varpi_{j}=\frac{\prod\limits_{i=1}^{n}x_{i}}{x_{j}}\frac{D}{P'(x_{j}^2)},\quad D=\prod\limits_{1\leq i<j\leq n}(x_{i}^{2}-x_{j}^{2}).
\end{eqnarray*} 
We should be careful about the sign, but calculations show that this is the correct value. Thus, we need to show that for any for $1\leq s\leq n-1$, we have
\begin{eqnarray*}
\sum\limits_{j=1}^{n}\frac{x_{j}^{2s-1}}{x_{j}P'(x_{j}^{2})}=0.
\end{eqnarray*}
This is equivalent to the following. Let $y_{j}=x_{j}^{2}$. Then
we should prove that for any $0\leq s\leq n-2$, one has
\begin{eqnarray}
\sum\limits_{j=1}^{n}\frac{y_{j}^{s}}{p'(y_{j})}=0.
\label{n-2-eq}
\end{eqnarray}
But this is a standard fact from Lagrange interpolation: just consider polynomial $T_{s}(X)$ of degree at most $n-1$ such that $T(y_{j})=y_{j}^{s}$, $1\leq j\leq n$. It is unique, and is given, of course, by $T(X)=X^{s}$. Then writing the interpolating polynomial and comparing the leading coefficient at $X^{n-1}$ (which is equal to $0$) gives the desired identity. If $s=n-1$, the same calculation shows that
\begin{eqnarray*}
\sum\limits_{j=1}^{n}\frac{y_{j}^{n-1}}{p'(y_{j})}=1.
\end{eqnarray*}
\end{proof}
\section{The differential system}Thus, we need to solve the following system of differential equations, a direct generalization of (\ref{sys2}): 
\begin{eqnarray}
\left\{\begin{array}{l}
p_{1}'=\prod\limits_{j\neq 1}p_{j}\prod\limits_{2\leq i<j\leq n}(p_{i}^{2}-p_{j}^2),\\
\quad\\
p_{2}'=\prod\limits_{j\neq 2}p_{j}\prod\limits_{\text{cyclically}}(p_{i}^{2}-p_{j}^2),\\
\ldots,\\
\quad\\
\sigma_{\ell}(p_{1}^2,\ldots, p_{n}^{2})=\xi_{\ell},\quad 1\leq \ell\leq n-1.
\end{array}
\right.
\label{sys-ind}
\end{eqnarray}
(Signs ``-1", as in (\ref{sys2}), are inessential). In fact, only the first equation and $(n-1)$ first integrals are needed (since we already have found them), but we write for clarity all the equations. 
Here $\sigma_{j}$ are standard symmetric polynomials of degree $\ell$. As in case $n=3$, we can see immediately that if $\{p_{1},\ldots,p_{n}\}$ solves this system, then changing exactly two of them say, $\{p_{i},p_{j}\}$, to $\{-p_{i},-p_{j}\}$, solves the system, too.  \\

In a $3$-dimensional case, we had an inequality $\frac{1}{3}\leq\xi\leq 1$, and all the possibilities do occur. Note that in the setting of this Section, $\xi=1-2\xi_{2}$.\\

 Equally, we have many restrictions on the set $\{\xi_{\ell}:1\leq \ell\leq n-1\}$. First note that $p_{i}^{2}\geq 0$. Not all collections of $\xi$'s produce a genuine orbit. For this to be the case, many inequalities should be satisfied. Let
\begin{eqnarray*}
E_{k}=\frac{\xi_{k}}{\binom{n}{k}},\quad 1\leq k\leq n.
\end{eqnarray*}
Note that we do not have the last one, $\xi_{n}$, though it will play a crucial role further; see Theorem \ref{thm-fin}.
The first series of inequalities follows easily from the \emph{Muirhead's inequalities} \cite{hardy}:
\begin{eqnarray*}
E_{1}\geq E_{2}^{\frac{1}{2}}\geq\cdots\geq E_{n}^{\frac{1}{n}}.
\end{eqnarray*}
Another example is the \emph{Newton-Maclaurin inequalities}
\begin{eqnarray*}
E_{k}^{2}\geq E_{k-1}E_{k+1},\quad 1\leq k\leq n-1.
\end{eqnarray*}
Yet another series of inequalities was given by Roset:\small
\begin{eqnarray*}
4(E_{k+1}E_{k+3}-E_{k+2}^{2})(E_{k}E_{k+2}-E_{k+1}^{2})
\geq(E_{k+1}E_{k+2}-E_{k}E_{k+3})^2,\quad 0\leq k\leq n-3.
\end{eqnarray*}\normalsize
See \cite{hardy} for more on this topic. Thus, we must be careful in choosing a collection of $\xi$'s to make sure that it corresponds to a genuine orbit. We call such a collection of $\xi$'s \emph{admissible}. The easiest way is to start from a fixed point $\{x_{1},\ldots, x_{n}\}$ and consider symmetric expressions in its coordinates.\\ 

Now, multiply the first equation in (\ref{sys-ind}) by $p_{1}$, and raise to the power $2$. We have 
\begin{eqnarray*}
(p_{1}')^{2}p^{2}_{1}=\prod\limits_{j=1}^{n}p^{2}_{j}\prod\limits_{2\leq i<j\leq n}(p_{i}^{2}-p_{j}^2)^{2}.
\end{eqnarray*}
Let $P_{j}=p_{j}^{2}$, $1\leq j\leq n$, and  $P_{1}=P$. Since $P'=2p'_{1}p_{1}$, this gives
\begin{eqnarray}
(P')^2=4P\prod\limits_{j=2}^{n}P_{j}\prod\limits_{2\leq i<j\leq n}(P_{i}-P_{j})^{2}.
\label{P-Pder}
\end{eqnarray}
Next, we have 
\begin{eqnarray*}
\sigma_{\ell}(P,P_{2},\ldots,P_{n})=\xi_{\ell},\quad 1\leq \ell\leq n-1.
\end{eqnarray*}
Let $\omega_{\ell}$, $1\leq \ell\leq n-1$, be the standard symmetric polynomials in variables $P_{2},\ldots, P_{n}$, and put $\omega_{0}=1$. Then:
\begin{eqnarray*}
\xi_{\ell}=\sigma_{\ell}(P,P_{2},\ldots,P_{n})=\omega_{\ell}+P\omega_{\ell-1},\quad 1\leq \ell\leq n-1.
\end{eqnarray*}
This allows to express $\omega_{\ell}$ it terms of $\xi_{\ell}$ and $P$. Indeed,
\begin{eqnarray*}
\omega_{1}&=&\xi_{1}-P,\\ 
\omega_{2}&=&\xi_{2}-P\xi_{1}+P^{2},\\
\omega_{3}&=&\xi_{3}-P\xi_{2}+P^{2}\xi_{1}-P^{3},\\
&\ldots&\\
\omega_{n-1}&=&\xi_{n-1}-P\xi_{n-2}+\cdots+P^{n-1}.
\end{eqnarray*}
Let us define
\begin{eqnarray*}
H(X)=X^{n}-\xi_{1}X^{n-1}+\xi_{2}X^{n-2}+\cdots+\xi_{n-1}X.
\end{eqnarray*}
For a moment, we do not need the free term. Consider now the polynomial
\begin{eqnarray}
\mathscr{L}(Z)&=&\prod\limits_{j=2}^{n}(Z-P_{j})=
\sum\limits_{s=0}^{n-1}Z^{n-s-1}\omega_{s}(-1)^{s}=\nonumber\\
&&\sum\limits_{s=0}^{n-1}Z^{n-s-1}\Big{(}P^{s}-\xi_{1}P^{s-1}+\cdots+(-1)^{s}\xi_{s}\Big{)}\label{discrim}\\
&=&\frac{Z^{n}-P^{n}}{Z-P}-\xi_{1}\frac{Z^{n-1}-P^{n-1}}{Z-P}+\cdots
+\xi_{n-1}\frac{Z-P}{Z-P}\nonumber\\
&=&\frac{H(Z)-H(P)}{Z-P}.\nonumber
\
\end{eqnarray}
Each coefficient of the polynomial (\ref{discrim}) is in its own right a polynomial in $P$ with fixed coefficients which depend only on $\xi$'s. As is clear from (\ref{P-Pder}), we get that
\begin{eqnarray*}
(P')^2=4P\cdot\omega_{n-1}\cdot\mathrm{discrim}\,\mathscr{L}_{Z}=4H(P)\cdot\mathrm{discrim}\,\mathscr{L}_{Z}.
\end{eqnarray*}
The highest degree of $P$ in $\mathrm{discrim}_{Z}\,\mathscr{L}$ can be calculated by looking at (\ref{discrim}) and noticing that this degree coincides with the highest degree of $P$ in the discriminant of the a polynomial
\begin{eqnarray*}
\sum\limits_{s=0}^{n-1}Z^{n-s-1}P^{s}=\frac{Z^{n}-P^{n}}{Z-P}.
\end{eqnarray*}
The discriminant of the latter is $\prod\limits_{1\leq a<b\leq n-1}(P\zeta^{a}-P\zeta^{b})^{2}$, $\zeta=\exp(\frac{2\pi i }{n})$, and so is  non-zero and of degree $(n-1)(n-2)$ in $P$. 
Thus,
\begin{eqnarray*}
\mathrm{discrim}\,\mathscr{L}_{Z}=\mathscr{F}(P),
\end{eqnarray*}
where $\mathscr{F}(X)=\mathscr{F}_{\xi_{1},\ldots,\xi_{\ell-1}}(X)$ is of degree $(n-1)(n-2)$ in $X$. Finally, any of $p_{j}$ can be in place of $p_{1}$. Thus, we have proved an $n$-dimensional analogue of 
Proposition \ref{prop8}.
\begin{prop}
\label{prop-red2}
Suppose a collection of functions $\{p_{1},p_{2},\ldots, p_{n}\}$ satisfies the differential system (\ref{sys-ind}). Let $p^2_{j}(t)=P(t)$ for any choice of $j$, $1\leq j\leq n$. Then the pair of functions $(P,P')=(X,Y)$ parametrizes (generically, hyper-elliptic) curve
\begin{eqnarray*}
Y^2=4H(X)\cdot\mathscr{F}(X).
\end{eqnarray*}
The right hand side  is of degree $n^2-2n+2$ in $X$.
\end{prop} 
When $n=3$, this degree is $5$, and this is exactly what was obtained in Proposition \ref{prop8}. 

\section{Third reduction}Now we will show that the reduction of hyper-elliptic curve into the elliptic, the phenomenon which happened for the octahedral superflow in Section \ref{form-generic}, occurs yet again.\\

Thus, let $X$ be any of the functions $P_{j}(t)$, and consider the function $\Upsilon(t)$ defined by
\begin{eqnarray*}
H(X)=\Upsilon(t).
\end{eqnarray*} 
We will show that $(\Upsilon,\Upsilon')$ parametrizes hyper-elliptic curve of lower degree. \\

Indeed, according to Proposition \ref{prop-red2}, we have
\begin{eqnarray*}
\Upsilon'^2=H'(X)^2(X')^2=H'(X)^2\cdot 4H(X)\cdot\mathscr{F}(X).
\end{eqnarray*} 
The right hand side is of degree $n^2$ in $X$. We claim that it is a polynomial $\mathscr{D}$ of degree $n$ in $H(X)$:
\begin{eqnarray}
H'(X)^2\cdot 4H(X)\cdot\mathscr{F}(X)=\mathscr{D}\big{(}H(X)\big{)}.
\label{above-2}
\end{eqnarray}
But
\begin{eqnarray*}
\mathscr{F}(X)=\mathrm{discrim}_{Z}\,\frac{H(Z)-H(X)}{Z-X}.
\end{eqnarray*}
In reality, the fact (\ref{above-2}) does not depend on $P_{j}$, $2\leq j\leq n$ (roots of $H(Z)=0$), it is purely algebraic and is valid for any monic polynomial $H(Z)$. Indeed,
\begin{eqnarray*}
H'(X)^2\cdot\mathrm{discrim}_{Z}\,\frac{H(Z)-H(X)}{Z-X}=\mathrm{discrim}_{Z}\,\big{(}H(Z)-H(X)\big{)}.
\end{eqnarray*}
Let $H(X)=\mathfrak{x}$. The right hand side of the above is obviously a $n$th degree polynomial $\mathfrak{x}$. Thus we have proved the following result. 
\begin{thm}[Triple reduction of the differential system]
\label{thm-fin}
Consider the differential system (\ref{sys-ind}), with $\xi_{\ell}$, $1\leq \ell\leq n-1$, being fixed and corresponding to an admissible collection. Define the polynomial $H(Z)$ by
\begin{eqnarray*}
H(Z)=Z^{n}-\xi_{1}Z^{n-1}+\xi_{2}Z^{n-2}+\cdots+\xi_{n-1}Z.
\end{eqnarray*}
Define the $n$th degree polynomial $\mathscr{D}_{n}(\mathfrak{x})$ by  
\begin{eqnarray*}
\mathscr{D}_{n}(\mathfrak{x})=4\mathfrak{x}\cdot\mathrm{discrim}_{Z}\,\big{(}H(Z)-\mathfrak{x}\big{)}.
\end{eqnarray*}
Let us define the hyper-elliptic \cite{mumford} function $\Upsilon(t)$ such that
\begin{eqnarray*}
\Upsilon'^2=\mathscr{D}_{n}(\Upsilon),
\end{eqnarray*} 
which is of generic genus $\frac{n-1}{2}$. Let $P_{j}$, $1\leq j\leq n$ be $n$ solutions of the algebraic equation $H(X)=\Upsilon(t)$. Then $p_{j}$, given by $p_{j}^{2}=P_{j}$ with the correct choices of signs, solve the system (\ref{sys-ind}). 
\end{thm}
We thus can clarify the penultimate paragraph of Section \ref{form-generic}. Indeed, the collection $\{p_{1},\ldots,p_{n}\}$ is replaced by the collection  $\{P_{1},\ldots,P_{n}\}$ (the first reduction); then it turns out that we can consider a pair $(P,P')$ separately (the second reduction); and finally, all such pairs can be encompassed by considering a single hyper-elliptic function $\Upsilon$ (the third reduction). In \cite{alkauskas-super2} we pose a problem asking when this happens for projective and polynomial superflows. \\

Let $n=3$, and take $\xi_{1}=1$, $\xi_{2}=\frac{1-\xi}{2}$ in Theorem \ref{thm-fin}. We obtain essentially (\ref{upsilon}), where $\Upsilon=-27\mathfrak{x}$.

\section{A singular orbit under $\phi_{\mathbb{O}_{5}}$}As the final topic of this chapter, let us return back to the case $n=5$. The polynomial $\mathscr{D}_{5}$ is a complicated polynomial in $\mathfrak{x}$ and $\xi_{j}$ :
\begin{eqnarray*}
\mathscr{D}_{5}(\mathfrak{x})&=&12500\mathfrak{x}^{5}\\
&+&\mathfrak{x}^{4}\cdot
(8000\xi_{1}^2\xi_{3}-6400\xi_{1}^3\xi_{2}+1024\xi_{1}^5
\\&&+9000\xi_{1}\xi_{2}^2-10000\xi_{1}\xi_{4}
-15000\xi_{2}\xi_{3})\\
&+&\mathfrak{x}^{3}\cdot(432\xi_{2}^5+4080\xi_{1}^2\xi_{2}^2\xi_{4}
+2240\xi_{1}^2\xi_{2}\xi_{3}^2
-8200\xi_{1}\xi_{2}\xi_{3}\xi_{4}\\
&&-3600\xi_{2}^3\xi_{4}+3300\xi_{2}^2\xi_{3}^2-512\xi_{1}^4\xi_{3}^2-768\xi_{1}^4\xi_{2}\xi_{4}
\\
&&+640\xi_{1}^3\xi_{3}\xi_{4}-200\xi_{1}^2\xi_{4}^2
+9000\xi_{3}^2\xi_{4}-3600\xi_{1}\xi_{3}^3\\
&&+8000\xi_{2}\xi_{4}^2-2520\xi_{1}\xi_{2}^3\xi_{3}+
576\xi_{1}^3\xi_{2}^2\xi_{3}-108\xi_{1}^2\xi_{2}^4
)\\
&+&\mathfrak{x}^{2}\cdot(2984\xi_{1}^2\xi_{2}\xi_{3}\xi_{4}^2
+96\xi_{1}^2\xi_{3}^3\xi_{4}
+64\xi_{2}^3\xi_{3}^3+64\xi_{1}^3\xi_{3}^4
-144\xi_{1}^3\xi_{4}^3\\
&&+96\xi_{1}\xi_{2}^3\xi_{4}^2+576\xi_{1}^4\xi_{3}\xi_{4}^2
+4080\xi_{1}\xi_{3}^2\xi_{4}^2-2520\xi_{2}\xi_{3}^3\xi_{4}
+2240\xi_{2}^2\xi_{3}\xi_{4}^2\\
&&+432\xi_{3}^5
+1424\xi_{1}\xi_{2}^2\xi_{3}^2\xi_{4}-320\xi_{1}^3\xi_{2}\xi_{3}^2\xi_{4}
-288\xi_{1}\xi_{2}\xi_{3}^4+72\xi_{1}^2\xi_{2}^3\xi_{3}\xi_{4}\\
&&-288\xi_{2}^4\xi_{3}\xi_{4}-16\xi_{1}^2\xi_{2}^2\xi_{3}^3
+640\xi_{1}\xi_{2}\xi_{4}^3-6400\xi_{3}\xi_{4}^3
-24\xi_{1}^3\xi_{2}^2\xi_{4}^2)\\
&+&\mathfrak{x}\cdot(1024\xi_{4}^5+72\xi_{1}^3\xi_{2}\xi_{3}\xi_{4}^3
-320\xi_{1}\xi_{2}^2\xi_{3}\xi_{4}^3+72\xi_{1}\xi_{2}\xi_{3}^3\xi_{4}^2\\
&&-16\xi_{1}^3\xi_{3}^3\xi_{4}^2-108\xi_{3}^4\xi_{4}^2
-108\xi_{1}^4\xi_{4}^4+576\xi_{2}\xi_{3}^2\xi_{4}^3\\
&&-512\xi_{2}^2\xi_{4}^4+64\xi_{2}^4\xi_{4}^3
+4\xi_{1}^2\xi_{2}^2\xi_{3}^2\xi_{4}^2-16\xi_{2}^3\xi_{3}^2\xi_{4}^2\\
&&-768\xi_{1}\xi_{3}\xi_{4}^4+576\xi_{1}^2\xi_{2}\xi_{4}^4
-24\xi_{1}^2\xi_{3}^2\xi_{4}^3-16\xi_{1}^2\xi_{2}^3\xi_{4}^3).
\end{eqnarray*}

Since $\mathfrak{x}$ serves as a free term in Theorem \ref{thm-fin} and hence as if plays the r\^{o}le of $\xi_{5}$ (which does not appear in our calculations), as expected, $\mathscr{D}_{5}$ is a homogeneous polynomial of degree $25$, if $(\xi_{1},\xi_{2},\xi_{3},\xi_{4},\mathfrak{x})$ are given weights, respectively, $(1,2,3,4,5)$.\\

As always, we call a particular orbit \emph{singular}, if $\mathscr{D}_{5}$ has at least a double root. If $\mathscr{D}_{5}$ has exactly one such and not a triple one, then the function $\Upsilon$ in fact can be expressed in terms of Weierstrass elliptic function. For example, let us choose a point $\m{x}_{q}=(x_{1},x_{2},x_{3},x_{3},x_{5})$ by letting
\begin{eqnarray*}
x_{1}=1,\quad x_{2}=1,\quad x_{3}=1,\quad x_{4}=2,\quad x_{5}=q. 
\end{eqnarray*} 
This gives 
\begin{eqnarray*}
\xi_{1}=7+q^2,\quad\xi_{2}=15+7q^2,\quad
\xi_{3}=13+15q^2,\quad \xi_{4}=4+13q^2.
\end{eqnarray*}

In this case $\mathscr{D}_{5}(\mathfrak{x})$ has a factor $(\mathfrak{x}-4q^2)^2$. We must make sure that the remaining third degree factor does not have a root $4q^2$.  Computer calculations with MAPLE show that this amounts to $q\neq 0$, $q^2\neq 1$, $q^2\neq 4$. We must also make sure that this factor has no repeated roots. This happens when
\begin{eqnarray}
19q^8-220q^6+9861^4-22841^2+2187\neq 0.\label{octo}
\end{eqnarray} 
This has $4$ real roots. 
\begin{cor}Suppose $q^2\neq 0,1,4$, and (\ref{octo}) is satisfied. Then the orbit of the point $\m{x}_{q}$ under the hyperoctahedral polynomial superflow $\phi_{\mathbb{O}_{5}}$ can be explicitly integrated in terms of Weierstrass elliptic functions.  
\end{cor}

%% file: appendix-pirm.tex
\chapter{Extremal homogeneous vector fields on spheres}
\label{app}

The next topic continues in the spirit of Section \ref{funda-period}, but is a little astray of our main theme of superflows, so we will be concise.\\

In the spirit of Section \ref{octa-dim}, we may pose extremity questions for homogeneous vector fields on spheres with an octahedral symmetry. However, these problems seem to be of interest for all homogeneous vector fields on spheres, with or without any orthogonal symmetry, hence we formulate problem for them first.\\

 The topic of vector fields on spheres is, of course, a deep and widely ramified subject \cite{miller}. For example, it was long been known that any $C^{\infty}$ vector field on $\m{S}^{2n}$, $n\in\mathbb{N}$, must vanish at some point.  One of central questions in this area is to calculate on $\m{S}^{n}$ the maximal possible amount of linearly independent vector fields. In particular, when this amount is equal to $n$. Or, in other words, when a vector bundle of $\m{S}^{n}$ is globally isomorphic to $\m{S}^{n}\times\mathbb{R}^{n}$. The theorem of R. Bott, J. Milnor \& M. Kervaire states that this happens if and only if $n=0,1,3$ or $7$ \cite{BM,conlon}. For instance, for $n=3$, $x^2+y^2+z^2+w^2=1$, such a maximal collection of mutually orthogonal vector fields is given by
\begin{eqnarray*}
(-y,x,w,-z),\\
(z,w,-x,-y),\\
(-w,z,-y,x).
\end{eqnarray*} 
 
Now, let $f_{j}\in\mathbb{R}_{\ell}[\m{x}]$, $\m{x}=(x_{1},\ldots, x_{n})$ be $\ell$-homogeneous polynomial functions, and let $\m{Q}=\sum\limits_{j=1}^{n}f_{j}\frac{\p}{\p x_{j}}$ be a vector field on a sphere $\mathbf{S}^{n-1}$: $\sum\limits_{j=1}^{n}x_{j}f_{j}=0$. Let $\mathcal{Y}=\mathcal{Y}(n,\ell)$ be a linear space of dimension 
\begin{eqnarray*}
n\binom{\ell+n-1}{n-1}-\binom{\ell+n}{n-1}=\ell\binom{n+\ell-1}{\ell+1}
\end{eqnarray*}
 of such vector fields. For $\m{Q}\in\mathcal{Y}$, define
$\beta_{2}(\m{Q})$ and $\beta_{\infty}(\m{Q})$ exactly as in Definition \ref{defin-const}. That is,
\begin{eqnarray*}
\beta_{2}=\frac{1}{|\m{S}^{n-1}|}\int\limits_{\m{S}^{n-1}}|\m{Q}|^{2}\d S,
\quad\beta_{\infty}=\sup\limits_{\m{S}^{n-1}}|\m{Q}|.
\end{eqnarray*} 
Now, let
\begin{eqnarray*}
\Xi(n,\ell)=\inf\limits_{\m{Q}\in\mathcal{Y}}\frac{\beta^{2}_{\infty}(\m{Q})}{\beta_{2}(\m{Q})},\quad\Upsilon(n,\ell)=\sup\limits_{\m{Q}\in\mathcal{Y}}\frac{\beta^{2}_{\infty}(\m{Q})}{\beta_{2}(\m{Q})}.
\end{eqnarray*}
We also define analogous constants $\Xi^{\mathbb{O}}(n,\ell)$ and $\Upsilon^{\mathbb{O}}(n,\ell)$ for $\m{Q}$ having an octahedral symmetry and being a vector field on a sphere.
\begin{prob}For each $(n,\ell)\in\mathbb{N}_{\geq 2}\times\mathbb{N}$, describe the interval $[\Xi(n,\ell),\Upsilon(n,\ell)]$, its arithmetic, asymptotic structure as $\ell\rightarrow\infty$. Analogous questions for $[\Xi^{\mathbb{O}}(n,\ell)$, $\Upsilon^{\mathbb{O}}(n,\ell)]$.
\end{prob}
For a different question related to extremity (or criticality) of vector fields on spheres, namely, the question of finding minimal volume in the space of tangent bundle of unit vector fields on odd-dimensional spheres, see \cite{BGM1,BGM2,medrano,Pe}.\\

We know the constants $\Xi^{\mathbb{O}}(3,\ell)$ and $\Xi^{\mathbb{O}}(3,\ell)$ do not exists for $\ell=2$ - there are no $2$-homogeneous vector fields with an octahedral symmetry (see Section \ref{octa-dim}). For $\ell=4$, these two constants coincide and are given by the value $\Omega_{\infty,\mathbb{O}}$; see (\ref{omega-2}). \\

For $(n,\ell)=(3,3)$, there exists the family of $3$-homogeneous vector fields with an octahedral symmetry, and it is given by (\ref{family}). As noted there, none of these flows is a flow on a unit sphere, hence the constants in question do not exist for $(n,\ell)=(3,3)$.
  
\begin{Example} Now we present an example how to calculate $\Xi(n,\ell)$ and $\Upsilon(n,\ell)$. For instance, consider the case $(n,\ell)=(2,4)$.\\

 We thus have a vector field $\m{Q}=\varpi\bl\varrho=yF\bl(-x)F$, where
\begin{eqnarray*}
F(x,y)=ax^3+bx^2y+cxy^2+dy^3.
\end{eqnarray*}
since $|\m{Q}|^2=F^2$, we want to find the maximum of $F$ subject to the condition $x^2+y^2=1$. Using the method of Lagrange multipliers, we need to find critical points for a function
\begin{eqnarray*}   
T(x,y,\lambda)=F(x,y)-\lambda(x^2+y^2-1).
\end{eqnarray*}
The set $\m{S}^{1}$ is compact, and so there are at least $2$ critical points. Since $F$ is odd, this implies if $(x_{0},y_{0})$ is a maximum, then $(-x_{0},-y_{0})$ is minumum, and the other way round.  But first, we can make the following simplification. Any orthogonal transformation of the unit circle (in this case, just a rotation) does not change the two constants $\beta_{2}$ and $\beta_{\infty}.$ Thus, we may suppose that one of the critical points is just $(x,y)=(1,0)$. Since
\begin{eqnarray}
\def\arraystretch{1.5}
\left\{\begin{array}{l}
T'_{x}(x,y,\lambda)=0,\\
T'_{y}(x,y,\lambda)=0,\\
x^2+y^2=1
\end{array}
\right.
\label{sys-lag}
\end{eqnarray}
 at this point, this gives $b=0$. Thus, without the loss of generality, we may suppose the latter holds. In this case
\begin{eqnarray*}
\beta_{2}=\frac{1}{16}(5a^2+c^2+5d^2+2ac).
\end{eqnarray*}
We can make the following crucial observation. Let us fix $\beta_{2}=C$. In the $(a,c,d)$ space, this defines the surface of an ellipsoid $\mathscr{E}$. For each point $\m{z}\in\mathscr{E}$, we have a function $F_{\m{z}}(x,y)$ and its maximal value $M(\m{z})$. Since $\mathscr{E}$ is compact and $M(\m{z})$ is continuous, both the infimum and the supremum are obtained at some $\m{z}$.\\

Now, suppose that $\m{z}=(a,c,d)$ is such that $M(\m{z})$ is maximal. If $F_{\m{z}}(1,0)$ is not a maximal value of this function, we can rotate once again to get this point to $(1,0)$ ($b$ stays $0$) to obtain $\bar{\m{z}}=(\bar{a},\bar{c},\bar{d})$ such that $M(\m{z})=M(\bar{\m{z}})=|\bar{a}|$. We immediately see that
\begin{eqnarray*}
\frac{\beta_{\infty}^{2}}{\beta_{2}}\leq\frac{16\bar{a}^2}{4\bar{a}^2+5\bar{d}^2+(\bar{a}+\bar{c})^2}\leq 4,
\end{eqnarray*}
and the bound can be achieved for $F=x^3-xy^2$.\\

Further, suppose that $\m{z}$ is such that $M(\m{z})$ is minimal. We may assume that $|M(\m{z})|=|a|$ (if needed, we can pass to $\bar{a}$). Now, if $a^{2}\neq 4C$ (which is true for such point), we can always move, while remaining on $\mathscr{E}$, in such a direction that $a^2$ decreases. This shows that for $F$, which solves the infimum problem, there exist another  critical point $(x_{0},y_{0})\neq (\pm 1,0)$. And second,
\begin{eqnarray}
F(x_{0},y_{0})^{2}=F(1,0)^{2}.
\label{crucial}
\end{eqnarray} 
 
 Calculations show that other extremum points exist only if $D=8c^2+9d^2-12ac>0$. We may suppose, without loss of generality, that $d=1$. Indeed, for extremal vector fields either $d\neq 0$, and we can divide by $d$ without altering the ratio $\frac{\beta_{\infty}^{2}}{\beta_{2}}$, or extremal values are achieved as $(a,c)=t(a_{0},c_{0})$ for $t\rightarrow\infty$.\\

 Then, apart from the critical points given by $(x,y,\lambda)=(\pm 1,0,\pm\frac{3}{2}a)$, there are four more points to consider. They are given by $(x_{0},y_{0},\lambda)$, where 
\begin{eqnarray*}
(x^{\pm}_{0})^{2}=\frac{3+2c^2-2ac\pm\sqrt{9+8c^2-12ac}}{6+6c^2-12ac+6a^2}.
\end{eqnarray*}
These are all $4$ roots of the polynomial
\begin{eqnarray}
Q(z)=(9+9a^2-18ac+9c^2)z^4+(6ac-9-6c^2)z^2+c^2.
\label{qq}
\end{eqnarray}
Since its leading coefficient is $9+9(a-c)^2$, it is of degree $4$. By a direct calculation,
\begin{eqnarray*}
F(x_{0},y_{0})=\frac{c+(3a-c)x_{0}^2}{3x_{0}}.
\end{eqnarray*}
Since we can always swap $(x_{0},y_{0})$ and $(-x_{0},-y_{0})$, without loss of generality, based on thee crucial remark (\ref{crucial}), we may suppose that $F(x_{0},y_{0})=a$. This is tantamount to a quadratic equation $P(x_{0})=0$, where
\begin{eqnarray}
P(z)=(3a-c)z^2-3az+c.
\label{pp}
\end{eqnarray}
This means that the two polynomials $P$ and $Q$, given by (\ref{qq}) and (\ref{pp}), have a common root, and so their resultant vanishes. Thus,
\begin{eqnarray*}
\mathrm{Res}_{z}(P,Q)=c^2(27a^3-27a-9ac^2-2c^3)(3a-2c)^3=0.
\end{eqnarray*}
Without going into technical details, we can prove that the key property to find $\Xi$ is the cubic factor $27a^3-27a-9ac^2-2c^3$. Cases $c=0$ and $c=\frac{3}{2}a$ are easily handled the same way, and these leads to minimal values $\frac{8}{5}$ and $\frac{64}{41}$, respectively.\\

We are thus lead to the following problem: to find the minimum of 
\begin{eqnarray*}
\frac{\beta_{\infty}^{2}}{\beta_{2}}=\frac{16a^2}{2ac+5a^2+c^2+5},
\end{eqnarray*}
subject to the condition $27a^3-27a-9ac^2-2c^3=0$, and $a,c\geq 0$. Via the same method of Lagrange multipliers, we obtain
\begin{eqnarray*}
a=\frac{3\sqrt{30}-2\sqrt{5}}{10},\quad c=\frac{12\sqrt{5}-3\sqrt{30}}{10}.
\end{eqnarray*}
Since we must also consider the case $(a,c)\rightarrow\infty$, this given the limit values $c=\frac{3}{2}a$, or $c=-3a$. The ratio $\frac{\beta_{\infty}^{2}}{\beta_{2}}$ is then equal to $\frac{64}{41}$ and $2$, respectively. This gives the constant $\Xi(2,4)$. Thus we have proved the following. 
\begin{prop}We have equalities
\begin{eqnarray*}
\Upsilon(2,4)=4,\quad \Xi(2,4)=2-\frac{2\sqrt{6}}{9}=1.455668946_{+}.
\end{eqnarray*}
The first and the second equalities are achieved for the vector fields $yF\bl (-xF)$ where
\begin{eqnarray*}
F=x^3-xy^2,\text{ and } F=(3\sqrt{6}-2)x^3+(12-3\sqrt{6})xy^2+2\sqrt{5}y^3,
\end{eqnarray*}respectively.  
\end{prop}
Figure \ref{fig-extrema} shows uniformly scaled functions $F(\cos(\phi),\sin(\phi))$ for two extremal values of $F$.
\begin{figure}
\includegraphics[scale=0.32,angle=-90]{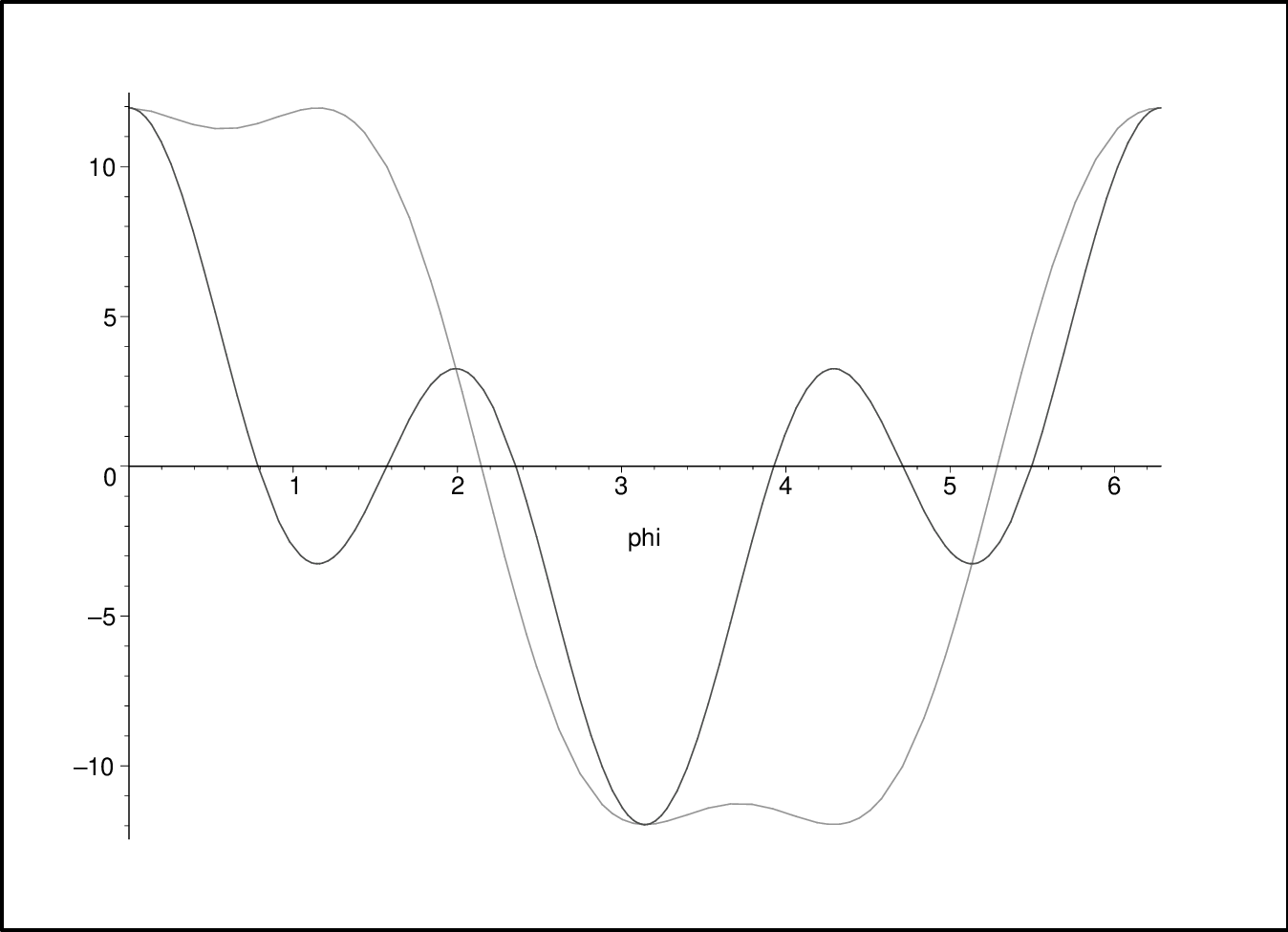}
\caption{The value of $F$, related to the vector field $yF\bl(-x)F$ on the unit sphere, paramatrized by $\pi\in\mathbb[0,2\pi]$ for two extremal homogeneous cubic polynomials}
\label{fig-extrema}
\end{figure}
Thus, the situation might be quite technical for spheres of higher order
\end{Example}